\documentclass[11pt]{article}
\usepackage[margin = 1in]{geometry}

\usepackage{lipsum}
\usepackage{amsfonts}
\usepackage{graphicx}
\usepackage{epstopdf}

\usepackage{mathtools,mathrsfs,mathabx}
\usepackage{empheq}
\usepackage{amsfonts}

\usepackage{amsgen,amsthm,amsmath,amstext,amsbsy,amsopn,amssymb,subfigure,stmaryrd,bbm,dsfont}
\usepackage{comment}
\usepackage{enumitem}
\usepackage{booktabs}
\usepackage{blkarray}
\usepackage{array}
\usepackage{algorithm}
\usepackage{algpseudocode}
\usepackage{url}
\usepackage{longtable}
\usepackage{mathtools}
\usepackage{multirow}
\usepackage{color}
\usepackage{hyperref}

\ifpdf
  \DeclareGraphicsExtensions{.eps,.pdf,.png,.jpg}
\else
  \DeclareGraphicsExtensions{.eps}
\fi

\newtheorem{Example}{Example}  
\newtheorem{Definition}{Definition}
\newtheorem{Theorem}{Theorem}

\newtheorem{Lemma}{Lemma}
\newtheorem{Remark}{Remark}
\newtheorem{Corollary}{Corollary}

\newtheorem{Proposition}{Proposition}

\DeclareMathAlphabet\mathbfcal{OMS}{cmsy}{b}{n}

\newcommand{\be}{\begin{equation}}
\newcommand{\ee}{\end{equation}}
\newcommand{\bea}{\begin{eqnarray}}
\newcommand{\eea}{\end{eqnarray}}
\newcommand{\beas}{\begin{eqnarray*}}
	\newcommand{\eeas}{\end{eqnarray*}}

\newcommand{\bbR}{\mathbb{R}}
\newcommand{\bbS}{\mathbb{S}}

\newcommand{\cS}{\mathcal{S}}

\newcommand{\cV}{\mathcal{V}}
\newcommand{\cH}{\mathcal{H}}

\newcommand{\cA}{{\mathcal{A}}}

\newcommand{\0}{{\mathbf{0}}}

\renewcommand{\u}{{\mathbf{u}}}
\renewcommand{\v}{{\mathbf{v}}}

\newcommand{\y}{{\mathbf{y}}}
\newcommand{\q}{{\mathbf{q}}}

\newcommand{\X}{{\mathbf{X}}}
\newcommand{\B}{{\mathbf{B}}}
\newcommand{\C}{{\mathbf{C}}}
\newcommand{\D}{{\mathbf{D}}}
\renewcommand{\H}{{\mathbf{H}}}
\newcommand{\I}{{\mathbf{I}}}

\renewcommand{\O}{{\mathbf{O}}}
\newcommand{\Q}{{\mathbf{Q}}}
\newcommand{\R}{{\mathbf{R}}}
\renewcommand{\L}{{\mathbf{L}}}
\renewcommand{\S}{{\mathbf{S}}}
\newcommand{\U}{{\mathbf{U}}}
\newcommand{\K}{{\mathbf{K}}}
\newcommand{\J}{{\mathbf{J}}}
\newcommand{\V}{{\mathbf{V}}}
\newcommand{\W}{{\mathbf{W}}}

\newcommand{\A}{{\mathbf{A}}}
\newcommand{\Y}{{\mathbf{Y}}}
\newcommand{\Z}{{\mathbf{Z}}}

\newcommand{\G}{{\mathbf{G}}}

\newcommand{\cR}{{\cal R}}
\newcommand{\cN}{{\cal N}}
\newcommand{\cM}{{\cal M}}

\newcommand{\bOmega}{\boldsymbol{\Omega}}
\newcommand{\bSigma}{\boldsymbol{\Sigma}}
\newcommand{\bDelta}{\boldsymbol{\Delta}}

\newcommand{\bepsilon}{{\boldsymbol{\epsilon}}}

\newcommand{\rank}{{\rm rank}}

\newcommand{\tr}{{\rm tr}}

\newcommand{\Hess}{{\rm Hess}}
\newcommand{\grad}{{\rm grad}}

\newcommand{\F}{{\rm F}}

\newcommand{\st}{{\rm St}}
\renewcommand{\log}{{\rm Log}}
\newcommand{\sign}{{\rm sign}}

\newcommand{\argmin}{\mathop{\rm arg\min}}

\newcommand{\bbO}{\mathbb{O}}

\newcommand{\nc}{\normalcolor}

\hypersetup{
    colorlinks=true,
    linkcolor=blue,
    filecolor=blue,      
    urlcolor=cyan,
    citecolor=blue
}

\makeatletter
\newcommand*{\rom}[1]{\expandafter\@slowromancap\romannumeral #1@}
\makeatother

\allowdisplaybreaks

\graphicspath{ {images/} }

\begin{document}
	\title{Nonconvex Matrix Factorization is Geodesically Convex: Global Landscape Analysis for Fixed-rank Matrix Optimization From a Riemannian Perspective}
	
	\author{Yuetian Luo$^1$ ~ and ~ Nicol{\'a}s Garc{\'i}a Trillos$^2$ }
	
	\date{}
	\maketitle

	\footnotetext[1]{Department of Statistics, Rutgers University.  (\texttt{yl2785@stat.rutgers.edu})}
	\footnotetext[2]{Department of Statistics, University of Wisconsin-Madison.  (\texttt{garciatrillo@wisc.edu})}
	
	\bigskip

\begin{abstract}
In this paper we study the landscape of a general matrix optimization problem with a fixed-rank positive semidefinite (PSD) constraint. We perform the Burer-Monteiro factorization, i.e., factorize a PSD matrix $\X$ as $\Y \Y^\top$, and consider a particular Riemannian quotient geometry in a search space that has a total space equipped with the Euclidean metric. When the original objective $f$ satisfies standard restricted strong convexity and smoothness properties, we characterize the global landscape of the factorized objective under the Riemannian quotient geometry. In particular, we show that the entire search space can be divided into three regions: ($\cR_1$) the region near the target parameter of interest, where the factorized objective is geodesically strongly convex and smooth; ($\cR_2$) the region containing neighborhoods of all strict saddle points; ($\cR_3$) the remaining regions, where the factorized objective has a large gradient. Our results cover both noisy and noiseless settings in applications of interest. To the best of our knowledge, this is the first global landscape analysis of the Burer-Monteiro factorized objective under the Riemannian quotient geometry. Our results provide a fully geometric explanation for the superior performance of vanilla gradient descent under the Burer-Monteiro factorization. When $f$ satisfies a weaker restricted strict convexity property, we show there exists a neighborhood near local minimizers such that the factorized objective is geodesically convex. To prove our main results we provide a comprehensive landscape analysis of a matrix factorization problem with a least squares objective, which serves as a critical bridge in establishing the results in the general setting. Our conclusions are also based on a result of independent interest stating that the geodesic ball centered at $\Y$ with a radius one-third of the least singular value of $\Y$ is a geodesically convex set under the Riemannian quotient geometry, a result that, as a corollary, also implies a quantitative bound of the convexity radius in the Bures-Wasserstein space. The convexity radius obtained in this paper is sharp up to constants.
\end{abstract}

{\bf Keywords:} Matrix factorization, global landscape analysis, Riemannian optimization, quotient geometry

{\bf MSC subject classification:} 58C05,90C06,90C22

\section{Introduction}
In this paper, we consider the following optimization problem 
\begin{equation} \label{eq: original-formulation}
	\min_{ \substack{  0  \preccurlyeq  \X \, \in  \, \bbS^{p \times p} , \\ \rank(\X) = r }} f(\X), \quad 0 < r \leq p.
\end{equation} Without loss of generality, we assume $f$ is symmetric in $\X \in \R^{p\times p}$, i.e., $f(\X) = f(\X^\top)$; otherwise, we can set $\tilde{f}(\X) = \frac{1}{2}(f(\X) + f(\X^\top)) $ and have $\tilde{f}(\X) = f(\X)$ for all $\X \succcurlyeq 0 $ \cite{bhojanapalli2016dropping}. In addition, we assume $f$ is twice continuously differentiable (in the usual sense) with respect to $\X$. To accelerate the computation of \eqref{eq: original-formulation} while coping with the rank constraint, a line of research has studied the following nonconvex factorization formulation (also dubbed as {\it Burer-Monteiro factorization} in the literature \cite{burer2005local}):
\begin{equation} \label{eq: factorization-objective}
	\min_{\Y \in \bbR^{p \times r}_*} \bar{h} (\Y) := f(\Y \Y^\top),
\end{equation}
i.e., $\X$ gets factorized as $\Y\Y^\top$ for a rectangular matrix $\Y$. Here,  $\bbR^{p \times r}_*$ denotes the set of $p$-by-$r$ matrices with full column rank.

The nonconvex factorization formulation \eqref{eq: factorization-objective} has been shown to be effective in many settings. For example, when $f$ is convex and smooth, \cite{bhojanapalli2016dropping} showed vanilla gradient descent (GD) on $\Y$ in \eqref{eq: factorization-objective} converges locally to the global minimizer at the classic sublinear convergence rate. When $f$ is well-conditioned, e.g., it satisfies certain restricted strong convexity and smoothness properties (see the forthcoming Definition \ref{def: RSC-RSM}), a number of works have demonstrated that GD in \eqref{eq: factorization-objective} or in its asymmetric version, i.e., when $\X = \L \R^\top$ ($\L \in \bbR^{p_1 \times r}$, $\R \in \bbR^{p_2 \times r}$), achieves linear convergence provided the scheme is initialized closely enough to the global minimizer \cite{bhojanapalli2016dropping,chen2015fast,candes2015phase,li2019rapid,ma2019implicit,sun2015guaranteed,tu2016low,charisopoulos2021low,zhao2015nonconvex,zheng2015convergent,ding2020leave,tong2020accelerating}. More surprisingly, perturbed GD or GD with random initialization has also been observed to have fast global convergence performance when carrying out the Burer-Monteiro factorization \cite{bi2022local,chen2019gradient,ye2021global}. 

The fact that the objective in \eqref{eq: factorization-objective} is nonconvex has led researchers to investigate the reasons behind the observed superior performance of vanilla GD. A large body of works, for example, has shown that the factorization in \eqref{eq: factorization-objective} actually does not introduce spurious local minima when the objective $f$ is well-conditioned  \cite{bhojanapalli2016global,chen2019model,ge2017no,li2019non,park2017non,zhang2021general,zhang2019sharp,zhang2018primal,zhu2018global}. Similar benign landscape results have been proved for the Burer-Monteiro factorization in solving semidefinite programs \cite{boumal2020deterministic,journee2010low,mei2017solving,ling2019landscape,waldspurger2020rank}. These works have taken a big step in demystifying the observed performance of the Burer-Monteiro factorization, but they are nonetheless restrictive in the sense that they mainly focus on the landscape analysis of the objective near the global minimum or at stationary points. Such a partial characterization is limited as it does not yield, in general, explicit guarantees on the rate of convergence to the global minimizer \cite{du2017gradient,ge2015escaping,lee2019first}. 

There have been a few attempts to study the global geometry of related optimization problems by characterizing their landscapes in the whole search space, rather than solely in regions near the global minimizer or near stationary points. For example, \cite{ge2015escaping}, \cite{sun2018geometric,cai2022nearly} and \cite{sun2016complete2} studied the global landscape geometries of tensor decomposition, phase retrieval and complete dictionary recovery problems, respectively. The global Euclidean geometries of \eqref{eq: factorization-objective} and their asymmetric versions were studied in \cite{li2016symmetry} and \cite{zhu2017global}, respectively. In particular, \cite{zhu2017global} considered a general well-conditioned objective $f$ and showed that the landscape of $\bar{h}(\Y)$ is benign and that the whole search space can be characterized as follows: when $\Y$ is far from the global optimizer, then either the magnitude of the gradient of $\bar{h}(\Y)$ is large, or the Hessian evaluated at $\Y$ has a negative eigenvalue; when $\Y$ is close to the global minimizer, then $\bar{h}(\Y)$ satisfies certain regularity condition which can facilitate local linear convergence of GD. \cite{li2016symmetry} considered a more restricted setting and assumed $f$ is a least squares objective, but they obtained a stronger local geometric result: $\bar{h}(\Y)$ is strongly convex in certain directions near the global minimizer. These results are encouraging and provide some evidence on why vanilla GD enjoys fast convergence performance when $f$ is well-conditioned. However, the description of local geometry near the global minimizer provided in \cite{li2016symmetry,zhu2017global} is not quite intuitive for two reasons: first, it is unclear { whether the regularity condition proved in \cite{zhu2017global} is a principled property that can be found in other problems beyond the setting of Burer-Monteiro factorization}; second, it is not obvious how to interpret the restricted directions that make the Hessian positive-definite in \cite{li2016symmetry}. On the other hand, notions such as convexity or strong convexity are classical notions that, when present, are typically used to obtain strong bounds on the computational complexity of many optimization algorithms. Thus, based on the existing results in the literature, we ask the central question that we explore in this paper:
\vskip.1cm
\noindent \fbox{ \parbox{0.98\textwidth}{
{\it {Can we provide a more straightforward and geometric explanation for why the Burer-Monteiro factorization works, e.g., using classical notions such as convexity or strong convexity?} }
}}
\vskip.1cm
  In addition, except for the landscape analysis for the special noisy matrix trace regression with a least squares objective in \cite[Section IV.B]{li2016symmetry}, the analyses in \cite{li2016symmetry,zhu2017global} mainly focus on the noiseless setting. Specifically, \cite{li2016symmetry,zhu2017global} assume there exists a rank $r$ positive semidefinite (PSD) parameter matrix of interest $\X^*$ which is the global minimizer of \eqref{eq: original-formulation} and satisfies $\nabla f(\X^*) = \0$. {\cite{ma2022noisy} also studies the landscape of noisy low-rank matrix optimization, but the emphasis is on the region near the global minimum or at stationary points.} We note the assumption $\nabla f(\X^*) = \0$ often holds in applications in noiseless settings, while do not hold in the noisy case; see the upcoming application in Example \ref{em: matrix-trace-regression}. Moreover, the guarantee in \cite[Section IV.B]{li2016symmetry} is customized to the noisy matrix trace regression problem and does not apply in more generality. \nc Thus, it is natural to ask:
\vskip.1cm
\noindent \fbox{ \parbox{0.98\textwidth}{
{\it Can we analyze the global optimization landscape of \eqref{eq: factorization-objective} in the general noisy setting? }
}}
\vskip.1cm

Finally, most of the previous results for landscape analysis in the literature are restricted to the setting where $f$ is well-conditioned. {The work \cite{dong2022analysis} weakens this assumption, only requiring $f$ to satisfy the restricted strong convexity and smoothness properties at the parameter of interest $\X^*$, but their landscape analysis (Proposition 4.7 therein) is local to the point $\X^*$}. Thus we wonder:
\vskip.1cm
\noindent \fbox{ \parbox{0.98\textwidth}{
{\it Can we study the landscape of \eqref{eq: factorization-objective} at other points under a weaker assumption on $f$?}
}}
\vskip.1cm

In this work, we provide affirmative answers to the above three questions. First, we note that, under the Euclidean formulation, if $\Y$ is a stationary point/local minimizer of \eqref{eq: factorization-objective}, then $\Y\O$ is also a stationary point/local minimizer for any $\O \in \bbO_r$. This ambiguity makes $\bar{h}(\Y)$ unavoidably nonconvex in any neighborhood of a stationary point \cite[Proposition 2]{li2016symmetry} and it becomes a fundamental hurdle in \cite{li2016symmetry,zhu2017global} for providing a more intuitive landscape analysis. To tackle this difficulty, we resort to tools from Riemannian optimization \cite{absil2009optimization,boumal2020introduction} and consider a particular Riemannian quotient geometry on \eqref{eq: factorization-objective} \cite{journee2010low}; see Section \ref{sec: Riemannian-opt-background} for a brief introduction of Riemannian optimization on quotient manifolds. Specifically, we encode the invariance mapping, i.e., $\Y \mapsto \Y \O$, in an abstract search space by defining the equivalence classes $[\Y] = \{ \Y \O: \O \in \bbO_r \}$. Since the invariance mapping is performed via the Lie group $\bbO_r$ smoothly, freely and properly, we conclude that $\cM_{r+}^{q} := \widebar{\cM}_{r+}^{q}/\bbO_r$ is a quotient manifold of $\widebar{\cM}_{r+}^{q}:=\bbR^{p \times r}_*$ \cite[Theorem 21.10]{lee2013smooth}. Moreover, we equip $T_\Y \widebar{\cM}_{r+}^{q}$ with the metric $\bar{g}_\Y(\eta_\Y, \theta_\Y)= \tr( \eta_\Y^\top \theta_\Y)$ for any $\eta_\Y, \theta_\Y \in T_\Y \widebar{\cM}_{r+}^{q}$, where $T_\Y \widebar{\cM}_{r+}^{q} = \bbR^{p \times r}$ is the tangent space of $\widebar{\cM}_{r+}^{q}$ at $\Y$. Since $\bar{h}(\Y)$ is invariant along the equivalence classes of $\widebar{\cM}_{r+}^{q}$, \eqref{eq: factorization-objective} induces the following optimization problem on the quotient manifold $\cM_{r+}^q$:
\begin{equation} \label{eq: quotient-factorization-objective}
	\min_{[\Y] \in \cM_{r+}^q } h ([\Y]):= \bar{h}(\Y)
\end{equation}

Our main contribution is on providing a purely geometric and straightforward explanation of the success of gradient based optimization algorithms in the Burer-Monteiro factorization via performing landscape analysis of \eqref{eq: quotient-factorization-objective} under the Riemannian quotient geometry. Our results cover various scenarios of $f$ and allow noise as well. An informal statement of our main results (Theorems \ref{th: convexity-radius-Mq}-\ref{th: local-landscape-f-convex} and Corollary \ref{coro: benign-landscape-f-well-conditioned}) is provided in the following Theorem \ref{th: informal-statement-main-result}.

\begin{Theorem}[Informal Results] \label{th: informal-statement-main-result}
	\begin{itemize}
		\item[(a)] When $f$ satisfies restricted strong convexity and smoothness properties (see Definition \ref{def: RSC-RSM}), and the noise level at the rank $r$ matrix parameter of interest $\X^*$ is controlled, then the global landscape of \eqref{eq: quotient-factorization-objective} is benign in the following sense: given any $\Y \in \bbR^{p \times r}_*$, one of the following three properties holds: 
		
		(i) the magnitude of Riemannian gradient of $h([\Y])$ is large; 
		
		(ii) the Riemannian Hessian of $h([\Y])$ has a large negative eigenvalue and there is an explicit escaping direction; 
		
		(iii) when $\Y\Y^\top$ is close to $\X^*$, $h([\Y])$ is geodesically strongly convex and smooth. Moreover, the distance between the global minimizer and $\X^*$ is bounded by a quantity that depends on the noise level of the problem or equivalently on the magnitude of $\nabla f(\X^*)$.  
		\item[(b)] When $f$ satisfies a weaker restricted strict convexity property (see Definition \ref{def: restricted-strictly-convex}), there exists a neighborhood around a local minimizer such that $h([\Y])$ is geodesically convex.
	\end{itemize}
\end{Theorem}
To the best of our knowledge, this is the first landscape analysis for the Burer-Monteiro factorization under the Riemannian quotient geometry. Thanks to this Riemannian formulation, in either setting of $f$ we have that there exists a local region in which the factorized objective is either geodesically strongly convex or geodesically convex. Such results are novel and can not be obtained under the frameworks of \cite{li2016symmetry,zhu2017global} as $\bar{h}(\Y)$ is in essence nonconvex under the Euclidean geometry. Moreover, we also provide the global landscape analysis of \eqref{eq: quotient-factorization-objective} when $f$ is well-conditioned. The set of three properties described in Theorem \ref{th: informal-statement-main-result}(a) is known as the robust strict saddle property in the literature \cite{ge2015escaping,jin2017escape,sun2015nonconvex} and many algorithms are guaranteed to achieve fast global convergence when these properties hold \cite{sun2019escaping,criscitiello2019efficiently,ge2015escaping,jin2017escape,sun2018geometric}. Finally, since the horizontal lift of the Riemannian gradient of $h([\Y])$ is the same as the Euclidean gradient (see the forthcoming Lemma \ref{lm: gradient-hessian-exp-PSD}) of $\bar{h}(\Y)$\nc, the gradient descent algorithms under the Riemannian quotient geometry and the Euclidean geometry are exactly the same from a  computational point of view. Thus, our geometric landscape results give a fully geometric and useful/straightforward explanation for the superior performance of gradient-based algorithms under the Burer-Monteiro factorization. As a concrete application of our landscape analysis we present a rigorous iteration complexity bound for a version of the perturbed Riemannian gradient descent algorithm \cite{criscitiello2019efficiently,sun2019escaping}
in the context of the matrix approximation problem \eqref{eq: quotient-matrix-fac-denoising} below. The details of this result can be found in Algorithm \ref{alg:PRGD} and in Theorem \ref{th:application} at the end of Section \ref{sec: landscape-analysis-f-well-conditioned}. \nc

To prove our main results, we first show a novel geodesic convexity property of the Riemannian quotient manifold $\cM_{r+}^q$ which is also of independent interest. Specifically, we show the geodesic convexity radius of $\cM_{r+}^q$ at $[\Y] \in \cM_{r+}^q$ is of order at least $\sigma_r(\Y)/3$; this radius is sharp up to the constant $1/3$. In addition, we base our global landscape analysis of \eqref{eq: quotient-factorization-objective} when $f$ is well-conditioned on the global landscape analysis of the following optimization problem:
\begin{equation} \label{eq: quotient-matrix-fac-denoising}
	\min_{[\Y] \in \cM_{r+}^q } H ([\Y]) := \frac{1}{2} \| \Y \Y^\top - \X^* \|_\F^2,
\end{equation} where $\X^*$ is the rank $r$ PSD parameter matrix of interest mentioned before. We show that the landscape of \eqref{eq: quotient-matrix-fac-denoising} is benign in the same sense as in Theorem \ref{th: informal-statement-main-result}(a). Moreover, the optimization landscape of \eqref{eq: quotient-matrix-fac-denoising} is preserved for the general low-rank matrix optimization \eqref{eq: quotient-factorization-objective} when $f$ satisfies the restricted strong convexity and smoothness properties. When $f$ satisfies a weaker restricted strict convexity property, we show that the Riemannian Hessian of $h([\Y])$ is positive definite at a local minimizer, and there exists a neighborhood around the local minimizer such that $h([\Y])$ is geodesically convex. 

\subsection{Additional Related Literature} \label{sec: related-literature}
 Riemannian manifold optimization methods are powerful tools when solving optimization problems with geometric constraints \cite{absil2009optimization,boumal2020introduction}. A lot of progress in this topic was made for studying the convergence of Riemannian optimization algorithms when solving low-rank matrix estimation problems, including matrix completion \cite{keshavan2009matrix,boumal2011rtrmc,vandereycken2013low,dong2022analysis}, robust PCA \cite{zhang2018robust}, matrix trace regression \cite{wei2016guarantees,meyer2011linear,luo2020recursive}, blind deconvolution/phase retrieval \cite{huang2018blind,luo2020recursive}, and general fixed-rank matrix optimization \cite{mishra2014fixed}. 

There are some precedents for the study of geometric landscape of an optimization problem under the Riemannian formulation. For example, \cite{maunu2019well} and \cite{alimisis2022geodesic} provided landscape analyses for robust subspace recovery and block Rayleigh quotient of symmetric or PSD matrices over the Grassmannian manifold. \cite{ling2020solving} studied the landscape of orthogonal group synchronization over the Stiefel manifold after performing the Burer-Monteiro factorization. Geometric landscape analysis of a quartic-quadratic optimization problem under a spherical constraint was examined in \cite{zhang2021geometric}. The landscape analyses in these works were mainly performed at stationary points. Under the embedded geometry for the set of fixed-rank matrices, \cite{uschmajew2018critical} proved that the landscape of \eqref{eq: original-formulation} when $f$ is quadratic and satisfies the restricted strong convexity and smoothness properties is benign. There again their focus is on the landscape at stationary points. 

\subsection{Organization of the Paper}\label{sec: organization}
The rest of this article is organized as follows. After a brief introduction of notation, we introduce Riemannian optimization and Riemannian optimization under the quotient geometry in Section \ref{sec: Riemannian-opt-background}. Geometric properties of $\cM_{r+}^q$ are provided in Section \ref{sec: quotient-fixed-rank-matrix}. The global landscape analysis of $h([\Y])$ when $f$ is well-conditioned is given in Section \ref{sec: landscape-analysis-f-well-conditioned}. In that section we present some discussion on the implications of our global landscape analysis on the global convergence of optimization algorithms and provide corresponding pointers to the literature. \nc The local geometry of $h([\Y])$ when $f$ satisfies restricted strict convexity property is given in Section \ref{sec: landscape-convex-f}. In Section \ref{sec: H-landscape-analysis}, we present the global landscape analysis for $H([\Y])$. Proofs for the main results are provided in Sections \ref{proof-sec: H-landscape} and \ref{sec: proof-convex-radius}. Conclusion and future work are given in Section \ref{sec: conclusion}. Additional proofs and lemmas are presented in Appendices \ref{sec: proof-landscape-h-f-well-condition}-\ref{proof-sec: additional-lemmas}.

\subsection{Notation and Preliminaries} \label{sec: notation}

The following notation will be used throughout this article. We use $\bbR^{p_1 \times p_2}$, $\bbS^{p \times p}$, $\bbR^{p \times r}_*$ and $\st(r,p)$ to denote the spaces of $p_1$-by-$p_2$ real matrices, $p$-by-$p$ real symmetric matrices, $p$-by-$r$ real full column rank matrices, and $p$-by-$r$ real matrices with orthonormal columns, respectively. Let $\bbO_{p,r}$ be the set of $p$-by-$r$ column orthonormal matrices and $\bbO_r := \bbO_{r,r}$. Uppercase and lowercase letters (e.g., $A, a$), lowercase boldface letters (e.g., $\u$), uppercase boldface letters (e.g., $\U$) are often used to denote scalars, column vectors, and matrices, respectively. For any $a, b \in \bbR$, let $a \wedge b := \min\{a,b\}, a \vee b := \max\{a,b\}$. For any matrix $\X \in \mathbb{R}^{p_1\times p_2}$ with singular value decomposition (SVD) $\sum_{i=1}^{p_1 \land p_2} \sigma_i(\X)\u_i \v_i^\top$, where $\sigma_1(\X) \geq \sigma_2(\X) \geq \cdots \geq \sigma_{p_1 \wedge p_2} (\X)$, denote its Frobenius norm and spectral norm as $\|\X\|_\F = \sqrt{\sum_{i} \sigma^2_i(\X)}$ and $\|\X\| = \sigma_1(\X)$, respectively. Throughout the paper, the SVD (or eigendecomposition) of a rank $r$ matrix $\X$ (or symmetric matrix $\X$) refers to its economic or reduced version. Let $\X_{\max(r)}= \sum_{i=1}^{r} \sigma_i(\X)\u_i \v_i^\top$ be the best rank-$r$ approximation of $\X$ in the Frobenius norm. Also, denote $\tr(\X)$ and $\X^{-1}$ as the trace and inverse of $\X$, respectively. For any $\X \in \bbS^{p \times p}$ having eigendecomposition $\U \bSigma \U^\top$ with non-increasing eigenvalues on the diagonal of $\bSigma$, let $\lambda_i(\X)$ be the $i$-th largest eigenvalue of $\X$, $\lambda_{\min}(\X)$ be the least eigenvalue of $\X$, and $\X^{1/2} = \U \bSigma^{1/2} \U^\top$. We write $\X \succcurlyeq 0$ if $\X$ is a symmetric positive semidefinite (PSD) matrix. For any $p$-by-$r$ column orthonormal matrix $\U$, let $P_{\U} = \U\U^\top$ represent the orthogonal projector onto the column space of $\U$; we use $\U_\perp\in \bbR^{p \times (p-r)}$ for the orthonormal complement of $\U$. Finally, suppose $f: \bbR^{p_1 \times p_2} \to \bbR$ is a differentiable scalar function, let $\nabla f(\X)$ and $\nabla^2 f(\X)$ be its Euclidean gradient and Hessian, respectively. We define the bilinear form of the Euclidean Hessian of $f$ as $\nabla^2 f(\X)[\Z_1, \Z_2]:= \langle \nabla^2 f(\X)\Z_1, \Z_2 \rangle $ for any $\Z_1, \Z_2 \in \bbR^{p_1 \times p_2}$, where $\langle \cdot, \cdot \rangle$ is the standard Euclidean inner product.

\subsection{Riemannian Optimization Under Quotient Geometries} \label{sec: Riemannian-opt-background}
In this section, we first give a brief introduction to Riemannian optimization and then discuss how to perform Riemannian optimization under quotient geometries. 

Riemannian optimization concerns optimizing a real-valued function $f$ defined on a Riemannian manifold $\cM$. The calculations of Riemannian gradients and Riemannian Hessians are key ingredients to perform continuous optimization over $\cM$. Let $\X \in \cM$, and let $g_{\X}( \cdot, \cdot )$ be the Riemannian metric and $T_\X \cM$ be the tangent space of $\cM$ at $\X$. Then the {\it Riemannian gradient} of a smooth function $f:\cM \to \bbR$ at $\X$ is defined as the unique tangent vector, ${\rm grad}\, f(\X) \in T_\X \cM$, such that $g_\X( \grad \, f(\X),\xi_\X ) = {\rm D} \, f(\X)[\xi_\X], \forall\, \xi_\X \in T_\X \cM$, where ${\rm D} f(\X)[\xi_\X]$ is the differential of $f$ at point $\X$ along the direction $\xi_\X$. The {\it Riemannian Hessian} of $f$ at $\X\in\cM$ is a linear mapping ${\rm Hess} \,f(\X): T_\X\cM \to T_\X\cM$ defined as
\begin{equation} \label{def: Riemannain-Hessian}
	{\rm Hess}\, f(\X)[\xi_\X] = \widebar{\nabla}_{\xi_\X} {\rm grad}\, f \in T_\X \cM, \,\quad \forall \xi_\X \in T_\X \cM,
\end{equation}
where $\widebar{\nabla}$ is the {\it Riemannian connection} on $\cM$, which is a generalization of the directional derivative along a vector field to Riemannian manifolds \cite[Section 5.3]{absil2009optimization}. We say $\X \in \cM$ is a {\it Riemannian first-order stationary point (FOSP)} of $f$ if $\grad f(\X) = \0$ and call a Riemannian FOSP a {\it strict saddle} if the Riemannian Hessian evaluated at this point has a strict negative eigenvalue. Given a subset $\cS$ of $\cM$, we call $f: \cS \to \bbR$ {\it geodesically convex} if $\cS$ is a geodesically convex set (i.e., any two points in $\cS$ can be connected by a geodesic that is completely contained in $\cS$) \nc and $\Hess f(\X) \succcurlyeq 0$ for all $\X \in \cS$; we call $f: \cS \to \bbR$ {\it $\mu$-geodesically strongly convex} if $\cS$ is geodesically convex  and $\Hess f(\X) \succcurlyeq \mu {\rm Id}$ for all $\X \in \cS$, {where {\rm Id} denotes the identity operator}.

Next, we provide more details on how to perform Riemannian optimization on quotient manifolds. Quotient manifolds are often defined via an equivalence relation ``$\sim$'' that satisfies symmetric, reflexive and transitive properties \cite[Section 3.4.1]{absil2009optimization}. To be specific, suppose $\widebar{\cM}$ is an embedded submanifold equipped with an equivalence relation $\sim$. The {\it equivalence class} (or {\it fiber}) of $\widebar{\cM}$ at a given point $\X$ is defined by the set $[\X] = \{\X_1 \in \widebar{\cM}: \X_1 \sim \X \}$. The set $ \cM := \widebar{\cM}/\sim = \{[\X]: \X \in \widebar{\cM} \}$ is called a {\it quotient} of $\widebar{\cM}$ by $\sim$. The mapping $\pi: \widebar{\cM} \to \widebar{\cM}/\sim$, $ \X \mapsto [\X]$ is called the {\it quotient map} or {\it canonical projection} and the set $\widebar{\cM}$ is called the {\it total space} of the quotient $\widebar{\cM}/\sim$. If $\cM$ further admits a smooth manifold structure and $\pi$ is a smooth submersion, then we call $\cM$ a {\it quotient manifold} of $\widebar{\cM}$. 

Due to the abstractness of equivalence classes, the tangent space $T_{[\X]} \cM$ of $\cM$ at $[\X]$ calls for a representation in the tangent space $T_{\X}\widebar{\cM}$ of the total space $\widebar{\cM}$. By the equivalence relation $\sim$, the representation of elements in $T_{[\X]} \cM$ should be restricted to the directions in $T_\X \widebar{\cM}$ that do not induce displacement along the equivalence class $[\X]$. This can be achieved by decomposing $T_{\X}\widebar{\cM}$ into complementary spaces $T_{\X}\widebar{\cM} = \cV_{\X} \widebar{\cM} \oplus \cH_{\X} \widebar{\cM}$, where ``$\oplus$'' is the direct sum. Here, $\cV_{\X} \widebar{\cM}$ is called the {\it vertical space}, which contains {tangent vectors that generate displacements within the equivalence class $[\X]$}. $\cH_{\X} \widebar{\cM}$ is called the {\it horizontal space} of $T_{\X}\widebar{\cM}$, which is complementary to $\cV_\X \widebar{\cM}$ and provides a proper representation of the abstract tangent space $T_{[\X]} \cM$ \cite[Section 3.5.8]{absil2009optimization}. Once $\widebar{\cM}$ is endowed with $\cH_{\X}\widebar{\cM}$, a given tangent vector $\eta_{[\X]} \in T_{[\X]} \cM$ at $[\X]$ is uniquely represented by a horizontal tangent vector $\eta_\X \in \cH_\X \widebar{\cM}$ that satisfies ${\rm D} \pi (\X)[\eta_\X] = \eta_{[\X]}$ \cite[Section 3.5.8]{absil2009optimization}, {where ${\rm D} \pi$ denotes the differential of quotient map $\pi$}. The tangent vector $\eta_\X \in \cH_\X \widebar{\cM}$ is also called the {\it horizontal lift} of $\eta_{[\X]}$ at $\X$. {Parallel translation $\Gamma_{[\X']}^{[\X]}$ denotes a map that transports $v \in T_{[\X']} \cM$ to $ \Gamma_{[\X']}^{[\X]} v \in T_{[\X]} \cM$ along geodesic $[\X'] \to [\X]$ such that the transported vector stays ``constant" by satisfying a zero-acceleration condition; see its detailed definition in Section 5.4 of \cite{absil2009optimization}.}

Next, we introduce the notion of {\it Riemannian quotient manifolds}. Suppose the total space $\widebar{\cM}$ is endowed with a Riemannian metric $\bar{g}_\X$, and for every $[\X] \in \cM$ and every $\eta_{[\X]}, \theta_{[\X]} \in T_{[\X]} \cM$, the expression $\bar{g}_\X(\eta_\X, \theta_\X )$, i.e., the inner product of the horizontal lifts of $\eta_{[\X]}, \theta_{[\X]}$ at $\X$, does not depend on the choice of the representative $\X$. Then the metric $\bar{g}_\X$ in the total space induces a metric $g_{[\X]}$ on the quotient space, i.e., $g_{[\X]}(\eta_{[\X]}, \theta_{[\X]}):= \bar{g}_\X(\eta_\X, \theta_\X)$. The quotient manifold $\cM$ endowed with $g_{[\X]}$ is called a {\it Riemannian quotient manifold} of $\widebar{\cM}$ \cite[Section 3.6.2]{absil2009optimization}. Optimization on Riemannian quotient manifolds is particularly convenient because computation of representatives of Riemannian gradients and Hessians in the abstract quotient space can be directly performed by means of their analogues in the total space. To be specific, suppose $\bar{f}: \widebar{\cM} \to \bbR$ is an objective function in the total space that is invariant along the fibers of $\widebar{\cM}$, i.e., $\bar{f}(\X_1) = \bar{f}(\X_2)$ whenever $\X_1 \sim \X_2$. Then $\bar{f}$ induces a function $f: \cM \to \bbR$ on the quotient space. Furthermore, if the horizontal space is canonically chosen (as we do in this paper)\nc, i.e., $\cH_\X \widebar{\cM}$ is the orthogonal complement of $\cV_\X \widebar{\cM}$ in $T_\X \widebar{\cM}$ with respect to $\bar{g}_\X$, then the horizontal lift of the Riemannian gradient of $f$ is $\overline{\grad f([\X])} = \grad \bar{f} (\X)$ \cite[Section 3.6.2]{absil2009optimization}, where $\grad \bar{f}(\X)$ denotes the Riemannian gradient of $\bar{f}$ at $\X$ in the total space.

Finally, the Riemannian connection on the Riemannian quotient manifold $\cM$ can also be uniquely represented by the Riemannian connection in the total space $\widebar{\cM}$. Suppose $\eta, \theta$ are two vector fields on $\cM$ and $\eta_\X$ and $\theta_\X$ are the horizontal lifts of $\eta_{[\X]}$ and $\theta_{[\X]}$ in $\cH_\X\widebar{\cM}$. Then the horizontal lift of $\widebar{\nabla}_{\theta_{[\X]}} \eta$ on the quotient manifold is given by $ \overline{ \widebar{\nabla}_{\theta_{[\X]}} \eta } = P_\X^{\cH} ( \widebar{\nabla}_{\theta_\X} \bar{\eta}  ) $, where $\bar{\eta}$ denotes the horizontal lift of the vector field $\eta$ and $\widebar{\nabla}_{\theta_\X} \bar{\eta}$ is the Riemannian connection in the total space \cite[Proposition 5.3.3]{absil2009optimization}. We also define the bilinear form of the horizontal lift of the Riemannian Hessian as $\overline{\Hess f([\X])} [\theta_{\X}, \eta_\X ] := \bar{g}_\X\left( \overline{\Hess f([\X])[\theta_{[\X]}]}, \eta_\X \right) $ for any $\theta_\X, \eta_\X \in \cH_\X \widebar{\cM}$. Then, by recalling the definition of the Riemannian metric $g_{[\X]}$ in the quotient space, we have
\begin{equation} \label{eq:Hessian-lift-property}
	\begin{split}
		\overline{\Hess f([\X])} [\theta_{\X}, \eta_\X ] &= \bar{g}_\X\left(\overline{\Hess f([\X])[\theta_{[\X]}]}, \eta_\X \right)\\
		& = g_{[\X]} \left( \Hess f([\X])[\theta_{[\X]}], \eta_{[\X]}  \right) = \Hess f([\X])[\theta_{[\X]}, \eta_{[\X]}].
	\end{split}
\end{equation} 
So $\Hess f([\X])$ is completely characterized by $\overline{\Hess f([\X])}$ in the lifted horizontal space.

\section{Geometric Properties and Geodesic Convexity of Balls in $\cM_{r+}^q$ } \label{sec: quotient-fixed-rank-matrix}
Recall the quotient manifold we are working with is $\cM_{r+}^{q} := \widebar{\cM}_{r+}^{q}/\bbO_r$ and we equip the tangent space $T_\Y \widebar{\cM}_{r+}^{q}$ with the metric $\bar{g}_\Y(\eta_\Y, \theta_\Y)= \tr(\eta_\Y^\top \theta_\Y)$. The following result, Lemma \ref{lm: psd-quotient-manifold1-prop}, provides the corresponding vertical and horizontal spaces of $T_\Y \widebar{\cM}_{r+}^{q}$ and proves that $\cM_{r+}^{q}$ is a Riemannian quotient manifold endowed with the Riemannian metric $g_{[\Y]}$ induced from $\bar{g}_\Y$.
\begin{Lemma}[\cite{journee2010low,massart2020quotient}] \label{lm: psd-quotient-manifold1-prop} Given $\U \in \st(r,p)$ spanning the top $r$ eigenspace of $\Y\Y^\top$, the vertical and horizontal spaces of $T_\Y \widebar{\cM}_{r+}^{q}$ can be written as follows:
	\begin{equation*}
		\begin{split}
			\cV_\Y\widebar{\cM}_{r+}^{q}  &= \{ \theta_\Y: \theta_\Y = \Y \bOmega, \bOmega = - \bOmega^\top \in \bbR^{r \times r} \}, \\
			\cH_\Y\widebar{\cM}_{r+}^{q}  &=\{ \theta_\Y: \theta_\Y = \Y (\Y^\top \Y)^{-1} \S + \U_\perp \D, \S \in \bbS^{r \times r}, \D \in \bbR^{(p-r) \times r} \}.
		\end{split}
	\end{equation*} The dimensions of $\cV_\Y\widebar{\cM}_{r+}^{q}$ and $\cH_\Y\widebar{\cM}_{r+}^{q}$ are, respectively, $(r^2-r)/2$ and $(pr-(r^2-r)/2)$, and $\cV_\Y\widebar{\cM}_{r+}^{q}$ is orthogonal to $\cH_\Y\widebar{\cM}_{r+}^{q} $ with respect to $\bar{g}_\Y$. Finally, $\cM_{r+}^q$ is a Riemannian quotient manifold endowed with the metric $g_{[\Y]}$ induced from $\bar{g}_\Y$.
\end{Lemma}
Next, we discuss a characterization of geodesics in $\cM_{r+}^q$ that has been presented in \cite{massart2020quotient}.
\begin{Lemma} \label{lm: logarithm-map} Let $\Y_1, \Y_2 \in \bbR^{p \times r}_*$, and $\Q_U \bSigma \Q_V^\top$ be the SVD of $\Y_1^\top \Y_2$. Denote $\Q^* = \Q_V \Q_U^\top$. Then
\begin{itemize}
	\item  $\Y_2 \Q^* - \Y_1 \in \cH_{\Y_1}\widebar{\cM}_{r+}^{q}$, $\Q^*$ is one of the best orthonormal matrices aligning $\Y_1$ and $\Y_2$, i.e., $\Q^* \in \argmin_{\Q \in \bbO_r} \|\Y_2 \Q - \Y_1\|_\F$ and the geodesic distance between $[\Y_1]$ and $[\Y_2]$ is $d([\Y_1], [\Y_2]) = \|\Y_2 \Q^* - \Y_1\|_\F$;
	\item if $\Y_1^\top \Y_2$ is nonsingular, then $\Q^*$ is unique and the Riemannian logarithm $\log_{[\Y_1]}[\Y_2]$ is uniquely defined and its horizontal lift at $\Y_1$ is given by $\overline{\log_{[\Y_1]}[\Y_2]} = \Y_2 \Q^* - \Y_1$; moreover, the unique minimizing geodesic from $[\Y_1]$ to $[\Y_2]$ is $[\Y_1 + t(\Y_2 \Q^* - \Y_1)]$ for $t \in [0,1]$.
\end{itemize}
\end{Lemma}
\begin{proof}
	First, by Lemma \ref{lm: psd-quotient-manifold1-prop}, to guarantee $\Y_2 \Q^* - \Y_1 \in \cH_{\Y_1}\widebar{\cM}_{r+}^{q}$, it is enough to show $\Y_1^\top (\Y_2 \Q^* - \Y_1) \in \bbS^{r \times r} $. This holds as $\Y_1^\top (\Y_2 \Q^* - \Y_1) = \Q_U \bSigma \Q_V^\top \Q_V \Q_U^\top - \Y_1^\top \Y_1 = \Q_U \bSigma \Q_U^\top - \Y_1^\top \Y_1 \in \bbS^{r \times r}$. The rest of the results in the lemma can be found in \cite[Proposition 5.1, Theorem 4.7 and Proposition 4.4]{massart2020quotient}.
\end{proof}

Given any $\Y \in \bbR^{p \times r}_*$ and $x > 0$, let $B_x([\Y]) := \{ [\Y_1]: d([\Y_1],[\Y]) < x \}$ be the geodesic ball centered at $[\Y]$ with radius $x$. It is known that for any Riemannian manifold there exists a convex geodesic ball at every point \cite[Chapter 3.4]{do1992riemannian}. However, it is often unclear how large this convex geodesic ball can be in different examples. In the next result, we quantify the convexity radius around a point $[\Y]$ in the manifold $\cM_{r+}^q$.
\begin{Theorem} \label{th: convexity-radius-Mq}
	Given any $\Y \in \bbR^{p \times r}_*$, the geodesic ball centered at $[\Y]$ with radius $x \leq r_\Y := \sigma_r(\Y)/3$ \nc is geodesically convex. In fact, for any two points $[\Y_1], [\Y_2] \in B_{x }([\Y])$, there is a unique shortest geodesic joining them, which is entirely contained in $B_{x }([\Y])$. \nc
\end{Theorem}
It has been shown in \cite[Theorem 6.3]{massart2020quotient} that the injectivity radius of $\cM_{r+}^q$ at $[\Y]$ is $\sigma_r(\Y)$, and since the convexity radius is smaller \nc than the injectivity radius, the geodesic convexity radius we proved in Theorem \ref{th: convexity-radius-Mq} is optimal up to a universal constant. Moreover, when $r = p$, the geometry we considered for $\cM_{p+}^q$ is also known as the {\it Bures-Wasserstein geometry} on the set of symmetric positive definite matrices, i.e., $\bbS^{p \times p}_{++} := \{\X \in \bbS^{p \times p}: \X \succ 0, \rank(\X) = p \}$ \cite{malago2018wasserstein,bhatia2019bures,oostrum2022bures}. Distinct from the common Log-Euclidean metric \cite{arsigny2007geometric} or the affine invariant metric \cite{moakher2005differential} on $\bbS_{++}^{p \times p}$, the set $\bbS_{++}^{p \times p}$ under the Bures-Wasserstein geometry is not complete. So, to the best of our knowledge, Theorem \ref{th: convexity-radius-Mq} also provides the first explicit geodesic convexity radius for $\bbS_{++}^{p \times p}$ under the Bures-Wasserstein geometry.

\section{Global Landscape Analysis of \eqref{eq: quotient-factorization-objective} When $f$ satisfies Restricted Strong Convexity and Smoothness Properties } \label{sec: landscape-analysis-f-well-conditioned}
In this section, we consider $f$ satisfies the following $(2r,4r)$-restricted strong convexity (RSC) and smoothness (RSM) properties:
\begin{Definition} \label{def: RSC-RSM}
We say $f: \bbR^{p \times p}\to \bbR$ satisfies the $(2r,4r)$-restricted strong convexity and smoothness properties with parameter $0 \leq \delta < 1$ if for any $\X,\G \in \bbR^{p \times p}$ with $\rank(\X) \leq 2r$ and $\rank(\G) \leq 4r$, the Euclidean Hessian of $f$ satisfies
\begin{equation} \label{eq: RSC-RSM} (1-\delta) \|\G\|_\F^2 \leq \nabla^2 f(\X)[\G,\G] \leq (1+\delta) \|\G\|_\F^2.
\end{equation} 
\end{Definition}

The RSC and RSM properties are satisfied in a number of examples and have been studied in \cite{wang2017unified,zhu2018global,zhu2017global,li2019non}. An important example is the stylized PSD matrix trace regression problem that we discuss next.
\begin{Example}[Matrix Trace Regression] \label{em: matrix-trace-regression}
	In PSD matrix trace regression, the goal is to recover a rank $r$ $p$-by-$p$ PSD matrix $\X^*$ from the observation $\y = \cA(\X^*) + \bepsilon$, where $\mathcal{A}\in \mathbb{R}^{p\times p} \to \mathbb{R}^n$ is a known linear map and $\bepsilon \in \bbR^n$ is the observational noise. The objective is 
	\begin{equation*} 
\min_{\X \in \mathbb{S}^{p\times p}, \X \succcurlyeq \0, \rank(\X) = r } f(\X) := \frac{1}{2} \left\| \mathcal{A}(\X) - \y \right\|_2^2.
\end{equation*}So the Euclidean gradient of $f(\X)$ is $\nabla f(\X) = \cA^{ \top \nc } (\cA(\X)-\y)$, and the quadratic form of the Euclidean Hessian satisfies 
\begin{equation*}
	\nabla^2 f(\X)[\D,\D] = \|\cA(\D)\|_2^2, \quad \forall \D \in \bbR^{p \times p}.
\end{equation*} Here we use the notation $\cA^\top$ to denote the adjoint of the linear map $\cA$. Note, the evaluation of the Euclidean gradient at $\X^*$ is $\nabla f(\X^*) = -\cA^{\top \nc}( \bepsilon )$, which is in general non-zero when $\bepsilon \neq 0$.

If the linear map $\cA$ satisfies the $4r$-restricted isometry property ($4r$-RIP) \cite{candes2010tight}, i.e., $(1-R_{4r}) \|\Z\|^2_\F \leq \|\cA(\Z)\|_2^2 \leq (1+R_{4r}) \|\Z\|_\F^2$ holds for all $\Z$ of rank at most $4r$ with parameter $ 0\leq R_{4r} < 1$, then the RSC and RSM in Definition \ref{def: RSC-RSM} hold for $f$ with $\delta = R_{4r}$. 
\end{Example}

Next, we provide expressions for the Riemannian gradient and Hessian of \eqref{eq: quotient-factorization-objective} under the Riemannian quotient geometry.
\begin{Lemma}(Riemannian Gradient and Hessian of \eqref{eq: quotient-factorization-objective} \cite[Proposition 1]{luo2021geometric}) \label{lm: gradient-hessian-exp-PSD}
	 Suppose $\Y \in \bbR^{p \times r}_*$ and $\theta_{\Y} \in \cH_{\Y} \widebar{\cM}_{r+}^{q}$. Then, recalling that $f$ has been assumed to be symmetric, we have \nc
		\begin{equation*} 
		\begin{split}
			\overline{\grad\, h([\Y])} &= 2\nabla f(\Y \Y^\top) \Y, \\
			\overline{\Hess \, h([\Y])}[\theta_\Y, \theta_\Y] 
		&= \nabla^2 f(\Y \Y^\top)[\Y\theta_\Y^\top + \theta_\Y \Y^\top , \Y\theta_\Y^\top + \theta_\Y \Y^\top ] + 2\langle \nabla f(\Y \Y^\top ), \theta_\Y \theta_\Y^\top \rangle.
		\end{split}
		\end{equation*}
\end{Lemma}
{Note that the Riemannian gradient and Hessian formulas are exactly the same as for the Euclidean gradient and Hessian. Their only difference is that the Riemannian Hessian is restricted to be evaluated at vectors $\theta_{\Y}$ that belong to the horizontal space $\cH_{\Y} \widebar{\cM}_{r+}^{q}$.}

Suppose the rank $r$ matrix of interest $\X^*$ has eigendecomposition $\U^* \bSigma^* \U^{*\top}$. Let $\Y^* = \U^* \bSigma^{*1/2} \in \bbR^{p \times r}_*$ and $\kappa^* = \sigma_1(\Y^*)/\sigma_r(\Y^*)$ be the condition number of $\Y^*$. We are ready to present our main results, where we describe the global landscape of \eqref{eq: quotient-factorization-objective} when $f$ satisfies the RSC and RSM properties. We will split the landscape of $h([\Y])$ into the following five regions (not necessarily non-overlapping): for $\mu, \alpha, \beta, \gamma \geq 0$, define
\begin{equation} \label{def: regions}
	\begin{split}
		\cR_1 &:= \{ \Y \in \bbR^{p \times r}_*| d([\Y], [\Y^*] ) \leq \mu \sigma_r(\Y^*)/\kappa^*  \} \\
		\cR_2 &:= \left\{  \Y \in \bbR^{p \times r}_* \Big| \begin{array}{l}
			 d([\Y], [\Y^*]) > \mu \sigma_r(\Y^*)/\kappa^*, \|\overline{\grad\, H([\Y])} \|_\F \leq \alpha \mu \sigma^3_r(\Y^*)/(4\kappa^*) \\
			 \|\Y\| \leq \beta \|\Y^*\|, \|\Y \Y^\top\|_\F \leq \gamma \|\Y^* \Y^{*\top}\|_\F 
		\end{array}   \right\}, \\
		\cR'_3 & :=  \{  \Y \in \bbR^{p \times r}_* | \|\overline{\grad\, H([\Y])} \|_\F > \alpha \mu \sigma^3_r(\Y^*)/(4\kappa^*), \|\Y\| \leq \beta \|\Y^*\|, \|\Y \Y^\top\|_\F \leq \gamma \|\Y^* \Y^{*\top}\|_\F   \},\\
		\cR_3'' & := \{  \Y \in \bbR^{p \times r}_* | \|\Y\| > \beta \|\Y^*\|, \|\Y \Y^\top\|_\F \leq \gamma \|\Y^* \Y^{*\top}\|_\F \}, \\
		\cR_3''' & := \{  \Y \in \bbR^{p \times r}_* | \|\Y \Y^\top\|_\F > \gamma \|\Y^* \Y^{*\top}\|_\F \},
	\end{split}
\end{equation} where $H([\Y])$ is given in \eqref{eq: quotient-matrix-fac-denoising}.  As we will see later, the use of $\|\Y^*\|$, $\|\Y^* \Y^{*\top}\|_\F$ in the definitions of $\cR_2, \cR_3', \cR_3''$ and $\cR_3'''$ is motivated by the connection between the gradients and Hessians of \eqref{eq: quotient-factorization-objective} and \eqref{eq: quotient-matrix-fac-denoising} provided in Proposition \ref{prop: hH-grad-hessian-connection}. \nc

Since $\cR_1 \bigcup \cR_2 \bigcup \cR_3' \supseteq \{  \Y \in \bbR^{p \times r}_* |\|\Y\| \leq \beta \|\Y^*\|, \|\Y \Y^\top\|_\F \leq \gamma \|\Y^* \Y^{*\top}\|_\F  \} $, we can easily check the following lemma holds.
 \begin{Lemma}
	$\cR_1 \bigcup \cR_2 \bigcup \cR_3' \bigcup \cR''_3 \bigcup \cR'''_3 = \bbR^{p \times r}_* $. 
\end{Lemma}

In the following Theorem \ref{th: h-local-geodesic-convexity}, we show $h([\Y])$ is geodesically strongly convex and smooth in $\cR_1$ for proper choices of $\mu$ and $\delta$.
\begin{Theorem}[Local Geodesic Strong Convexity of \eqref{eq: quotient-factorization-objective}] \label{th: h-local-geodesic-convexity}
Suppose $ 0\leq  \mu \leq 1/3$. For any $\Y \in \cR_1$, we have 
 \begin{equation*}
 	\begin{split}
 		\lambda_{\min}( \overline{\Hess\, h([\Y])} ) & \geq   ( 2 \left( 1 - \mu/ \kappa^* \right)^2 - 14 \mu/3 ) \sigma^2_r(\Y^*) \\
 		& \quad \quad -   \left(  4\delta ( \sigma_1(\Y^*) +  \mu \sigma_r(\Y^*)/ \kappa^*  )^2 + 14 \delta  \mu \sigma^2_r(\Y^*)/3 + 2 \| ( \nabla f(\X^*) )_{\max(r)} \|_\F  \right) ,\\
 		\lambda_{\max} ( \overline{\Hess \, h([\Y])}  ) & \leq 4\left(   \sigma_1(\Y^*) + \mu \sigma_r(\Y^*)/\kappa^* \right)^2 + 14\mu \sigma^2_r(\Y^*)/3 \\
			 & \quad  + \left(  4\delta ( \sigma_1(\Y^*) +  \mu \sigma_r(\Y^*)/ \kappa^*  )^2 + 14 \delta  \mu \sigma^2_r(\Y^*)/3 + 2 \| ( \nabla f(\X^*) )_{\max(r)} \|_\F  \right),
 	\end{split}
 \end{equation*} where $\kappa^*:= \sigma_1(\Y^*)/\sigma_r(\Y^*)$ is the condition number of $\Y^*$.
 
 In particular, if $\mu$ is further chosen such that $ \left( 1 - \mu /\kappa^* \right)^2 - 7\mu/3 > 0$, {e.g., $\mu < \frac{13 - \sqrt{133}}{6}$}, and
 \begin{equation*}
 	\delta \leq \frac{(1-\mu/\kappa^*)^2 - 7\mu/3}{ 4\left(2 ( \kappa^* + \mu/\kappa^* )^2 + 7\mu/3 \right) }\quad  \text{ and }\quad  \| (\nabla f(\X^*))_{\max(r)} \|_\F \leq \left( (1-\mu/\kappa^*)^2 - 7\mu/3 \right) \sigma^2_r(\Y^*)/4,
 \end{equation*}
then 
  \begin{equation*}
 	\begin{split}
 		\lambda_{\min}( \overline{\Hess\, h([\Y])} ) & \geq   \left( (1-\mu/\kappa^*)^2 - 7\mu/3 \right) \sigma^2_r(\Y^*) > 0.
 		\end{split}
 \end{equation*} Recall $ (\nabla f(\X^*))_{\max(r)}$ denotes the best rank $r$ approximation of $\nabla f(\X^*)$. Thus $h([\Y])$ is geodesically strongly convex and smooth in $\cR_1$. 
 
 Moreover, if there is a Riemannian FOSP $[\widehat{\Y}]$ in $\cR_1$, then it is the unique local minima in $\cR_1$ and it satisfies: 
 \begin{equation} \label{ineq: dist-bound-hatY-Ystar}
 \begin{split}
 	d([\widehat{\Y}], [\Y^*] ) \leq & \frac{2}{ \left((1-\mu/\kappa^*)^2 - 7\mu/3\right) \sigma^2_r(\Y^*) } \| \nabla f(\Y^*\Y^{*\top})\Y^* \|_\F \\
 	\leq & \frac{2\|\Y^{*}\|}{ \left((1-\mu/\kappa^*)^2 - 7\mu/3\right) \sigma^2_r(\Y^*) } \| (\nabla f(\X^*))_{\max(r)} \|_\F.
 \end{split}
 \end{equation}
\end{Theorem}
\begin{Remark}
	Similar to the recent work \cite{zhang2020many} studying the threshold of $\delta$ guaranteeing the absence of spurious local minimizers in a local region around $\X^*$, here we also provide an explicit dependence of the radius of $\cR_1$ on $\delta$. Recall that here $\cR_1$ is the region where we can guarantee the geodesic strong convexity of $h([\Y])$. We notice the upper bound for $\delta$ decreases as $\mu$ increases. This suggests that a stronger requirement on $\delta$ is needed if we desire a larger radius for the region where we can guarantee geodesic strong convexity.
\end{Remark}

Next, we show $ \overline{\Hess\, h([\Y])}$ has at least one negative eigenvalue for any $[\Y] \in \cR_2$ under proper assumptions. Moreover, we can explicitly find a direction for escaping the strict saddle points.
\begin{Theorem}[Region with Negative Eigenvalue in Riemannian Hessian of \eqref{eq: quotient-factorization-objective}] \label{th: h-negative-eigenvalue-region}   Given any $\Y \in \cR_2$, let $\theta_\Y = \Y - \Y^* \Q $, where $\Q \in \bbO_r$ is the best orthonormal \nc matrix aligning $\Y^*$ and $\Y$. Then we have
\begin{equation*}
\begin{split}
	\overline{\Hess\, h([\Y])}[\theta_\Y, \theta_\Y] & \leq \Big( (\alpha - 2 ( \sqrt{2} - 1 ))  \sigma^2_r(\Y^*)  \\
	 & +    2 \delta \left( 2 \beta^2 \|\Y^*\|^2 + (1+\gamma) \| \Y^* \Y^{*\top} \|_\F  \right) +  2 \| ( \nabla f(\X^*) )_{\max(r)} \|_\F  \Big) \|\theta_\Y\|_\F^2.
\end{split}
\end{equation*}
In particular, if $\alpha < 2(\sqrt{2} - 1)$, 
 \begin{equation*}
	\delta \leq \frac{( 2 (\sqrt{2} -1) -\alpha ) \sigma^2_r(\Y^*)  }{8 \left( 2 \beta^2 \|\Y^*\|^2 + (1 + \gamma ) \| \Y^* \Y^{*\top} \|_\F   \right) } \quad \text{ and } \quad \| ( \nabla f(\X^*) )_{\max(r)} \|_\F \leq \frac{ 2 ( \sqrt{2} - 1 ) - \alpha }{8} \sigma^2_r(\Y^*),
\end{equation*} then we have
\begin{equation*}
	 \overline{\Hess\, h([\Y])}[\theta_\Y, \theta_\Y]  \leq  \frac{\alpha - 2 ( \sqrt{2} - 1 )}{2}  \sigma^2_r(\Y^*) \|\theta_\Y\|_\F^2 < 0.
\end{equation*}
So $\overline{ \Hess\, h([\Y])}$ has at least one negative eigenvalue and $\theta_\Y = \Y - \Y^* \Q$ is an escaping direction.
\end{Theorem}

Finally, we show the Riemannian gradient of $h([\Y])$ has large magnitude in all three regions $\cR'_3, \cR_3''$ and $ \cR_3'''$.

\begin{Theorem}[Regions with Large Riemannian Gradient of \eqref{eq: quotient-factorization-objective}] \label{th: h-large-gradient-norm}
	(i) Given any $\Y \in \cR_3'$, we have
	\begin{equation*}
		\|\overline{\grad\, h([\Y])} \|_\F \geq \alpha \mu \sigma^3_r(\Y^*)/(4\kappa^*) - \left( 2 \delta \beta ( 1 + \gamma) \|\Y^*\| \| \Y^* \Y^{*\top} \|_\F + 2 \beta \|\Y^*\| \| ( \nabla f(\X^*) )_{\max(r)} \|_\F \right);
	\end{equation*}
(ii) given any $\Y \in \cR_3''$, we have
	\begin{equation*}
		\|\overline{\grad\, h([\Y])} \|_\F > 2(\beta^3 - \beta) \|\Y^*\|^3 - \left( 2 \delta ( 1 + \gamma) \|\Y\| \| \Y^* \Y^{*\top} \|_\F + 2 \|\Y\| \| ( \nabla f(\X^*) )_{\max(r)} \|_\F \right);
	\end{equation*}
(iii) given any $\Y \in \cR_3'''$, we have
	\begin{equation*}
	\begin{split}
		\|\overline{\grad\, h([\Y])} \|_\F > &  \left( 2( \gamma- 1) - 2\delta(\gamma + 1) \right) \gamma^{1/2} \|\Y^*\Y^{*\top}\|^{3/2}_\F/\sqrt{r}  \\
		& - 2 \gamma^{1/2}\|\Y^*\Y^{*\top}\|^{1/2}_\F\| ( \nabla f(\X^*) )_{\max(r)} \|_\F/\sqrt{r} ;
	\end{split}
	\end{equation*}

In particular, if $\beta > 1$, $\gamma > 1$, $\delta \leq \delta_{\min}$ and $\| ( \nabla f(\X^*) )_{\max(r)} \|_\F \leq \Psi$, where
\begin{equation} \label{def: delta-min}
	\delta_{\min} = \frac{\alpha \mu }{32 \kappa^{*2} \beta (1 + \gamma) } \frac{\sigma_r^2(\Y^*)}{ \| \Y^* \Y^{*\top} \|_\F } \wedge  \frac{  \beta^2 - 1 }{ 4(1 + \gamma) } \frac{ \|\Y^*\|^2 }{ \| \Y^* \Y^{*\top} \|_\F } \wedge  \frac{  \gamma - 1 }{ 4( \gamma + 1 ) },
\end{equation} and
\begin{equation} \label{def: Psi}
	\Psi = \frac{\alpha \mu}{32 \kappa^{*2} \beta } \sigma^2_r(\Y^*) \wedge  \frac{ \beta^2 -1 }{4 } \|\Y^*\|^2 \wedge   \frac{\gamma-1}{4}  \|\Y^* \Y^{*\top} \|_\F, 
\end{equation}
 we have the gradient norm $\|\overline{\grad\, h([\Y])} \|_\F$ in regions $\cR_3', \cR_3''$ and $\cR_3'''$ are lower bounded by strict positive quantities given as follows
\begin{equation} \label{ineq: h-gradient-bound-all}
	\begin{split}
		\cR_3': &  \quad \|\overline{\grad\, h([\Y])} \|_\F > \alpha\mu \sigma^3_r(\Y^*)/ (8 \kappa^* ), \\
		\cR_3'': &  \quad  \|\overline{\grad\, h([\Y])} \|_\F > (\beta^3 - \beta) \|\Y^*\|^3, \\
		\cR_3''': & \quad    \|\overline{\grad\, h([\Y])} \|_\F > ( \gamma- 1) \gamma^{1/2} \|\Y^*\Y^{*\top}\|^{3/2}_\F/\sqrt{r}.
	\end{split}
\end{equation} 
\end{Theorem}

A direct corollary from Theorems \ref{th: convexity-radius-Mq}\nc, \ref{th: h-local-geodesic-convexity}, \ref{th: h-negative-eigenvalue-region} and \ref{th: h-large-gradient-norm} is given below.

\begin{Corollary}[Benign Landscape of \eqref{eq: quotient-factorization-objective} \label{coro: benign-landscape-f-well-conditioned} When $f$ Satisfies RSC and RSM] Suppose $\mu, \alpha, \beta,\gamma \geq 0$ in the definitions of $\cR_1, \cR_2, \cR_3', \cR_3''$ and $\cR_3'''$ satisfy $\mu \leq 1/3 $, $ \left( 1 - \mu /\kappa^* \right)^2 - 7\mu/3 > 0$, $ \alpha < 2(\sqrt{2} - 1) $, $\beta > 1$ and $\gamma > 1$. Then if 
\begin{equation} \label{ineq: overall-delta-condition}
	\delta \leq  \frac{(1-\mu/\kappa^*)^2 - 7\mu/3}{ 4\left(2 ( \kappa^* + \mu/\kappa^* )^2 + 7\mu/3 \right) } \wedge \frac{( 2 (\sqrt{2} -1) -\alpha ) \sigma^2_r(\Y^*)  }{8 \left( 2 \beta^2 \|\Y^*\|^2 + (1 + \gamma ) \| \Y^* \Y^{*\top} \|_\F   \right) } \wedge \delta_{\min}
\end{equation} and
\begin{equation} \label{ineq: overall-f-noise-condition}
	 \| (\nabla f(\X^*))_{\max(r)} \|_\F \leq \left( (1-\mu/\kappa^*)^2 - 7\mu/3 \right) \sigma^2_r(\Y^*)/4 \wedge \frac{ 2 ( \sqrt{2} - 1 ) - \alpha }{8} \sigma^2_r(\Y^*) \wedge \Psi,
\end{equation} where $\delta_{\min}$ and $\Psi$ are defined in \eqref{def: delta-min} and \eqref{def: Psi}, respectively, we have the global geometric landscape of \eqref{eq: quotient-factorization-objective} is benign in the following sense:
\begin{itemize}
	\item in $\cR_1$, which is a geodesically convex set by Theorem \ref{th: convexity-radius-Mq}, \nc  $h([\Y])$ is geodesically strongly convex and smooth;
	\item in $\cR_2$, $\overline{\Hess\, h([\Y])}$ has a negative eigenvalue and there exists an explicit escaping direction;
	\item in $\cR_3:= \cR_3' \bigcup \cR_3'' \bigcup \cR_3'''$, $h([\Y])$ has large gradient. 
\end{itemize} 

In addition, if there exists a Riemannian FOSP $[\widehat{\Y}]$ in $\cR_1$, then $[\widehat{\Y}]$ is the unique global minimizer of \eqref{eq: quotient-factorization-objective} and the bound of the distance between $[\widehat{\Y}]$ and $[\Y^*]$ is provided in \eqref{ineq: dist-bound-hatY-Ystar}.
\end{Corollary}
\begin{Remark} \label{rem:local-geometry}
	Corollary \ref{coro: benign-landscape-f-well-conditioned} provides the first global landscape analysis for Burer-Monteiro factorized matrix optimization objective under the Riemannian quotient geometry. Different from the previous global landscape analysis for matrix factorization under the Euclidean geometry \cite{li2016symmetry,zhu2017global}, we are able to show $h([\Y])$ is actually geodesically strongly convex and smooth in $\cR_1$ while $\bar{h}(\Y)$ is nowhere convex in $\cR_1$ under the Euclidean geometry. Since the Euclidean gradient/Hessian are exactly the same as Riemannian gradient/Hessian, the gradient descent or trust region method under the Riemannian quotient geometry and the Euclidean geometry are exactly the same given the same configuration, our geometric landscape analysis results give a fully geometric and useful/straightforward explanation of the success of vanilla gradient descent and trust region under the Burer-Monteiro factorization. Moreover, our results cover the setting $\nabla f(\X^*)$ is non-zero but with a relatively small magnitude as well.  
	
	 Under the embedded geometry of the set of fixed-rank matrices, \cite{uschmajew2018critical} analyzed the landscape of \eqref{eq: original-formulation} at stationary points when $f$ is quadratic and satisfies restricted strong convexity and smoothness properties. Comparing to their results, we provide a global geometric landscape analysis of \eqref{eq: quotient-factorization-objective} under the Riemannian quotient geometry and our results hold for a general $f$ satisfying RSC and RSM. 
\end{Remark}

\begin{Remark}[Conditions] Suppose $\kappa^* = O(1)$ and $ r= O(1)$, the conditions in \eqref{ineq: overall-delta-condition} and \eqref{ineq: overall-f-noise-condition} for $\delta$ and $  \| (\nabla f(\X^*))_{\max(r)} \|_\F$ can be summarized as $\delta \leq c_1$ and $ \| (\nabla f(\X^*))_{\max(r)} \|_\F \leq c_2 \sigma_r(\X^*)$ for some small universal positive constants $c_1, c_2$. The condition for $\delta$ is not sharp compared to the recent attempts on establishing a sharp threshold of $\delta$ to guarantee the absence of spurious local minimizers under the factorization formulation \cite{zhang2019sharp,zhang2021general,ma2022noisy}. On the other hand, our geometric landscape results are much stronger and finer than theirs. Meanwhile, we also note the condition for $  \| (\nabla f(\X^*))_{\max(r)} \|_\F$ is weak. This is because in typical statistical applications, e.g., the matrix trace regression in Example \ref{em: matrix-trace-regression}, $\| (\nabla f(\X^*))_{\max(r)} \|_\F$ often matches the information-theoretic lower bound for estimating $\X^*$ and $O(\sigma_r(\X^*))$ is the initialization requirement for common local algorithms to converge \cite{chen2015fast,luo2020recursive}, so we often have $\| (\nabla f(\X^*))_{\max(r)} \|_\F \ll  \sigma_r(\X^*)$ in the standard low-rank matrix recovery literature. Moreover, this condition automatically holds when $\nabla f(\X^*) = \0$, which appears in typical noiseless applications. Finally, we note the error bound between the local minimizer in $\cR_1$ and $\Y^*$ provided in \eqref{ineq: dist-bound-hatY-Ystar} also matches the information-theoretic lower bound in common applications \cite[Section 4]{chen2015fast}. 
\end{Remark}

{
\begin{Remark} 
	As we mentioned right after Lemma \ref{lm: gradient-hessian-exp-PSD}, the Riemannian gradient and Hessian under the quotient geometry coincide with the Euclidean ones.  Because of this, we were able to borrow some ideas and techniques from works like \cite{li2016symmetry,zhu2017global} to analyze the landscape in regions $\cR_2$ and $\cR_3$. However, there are two major differences between the analyses in \cite{li2016symmetry,zhu2017global} and ours that we would like to highlight. First, the analyses in \cite{li2016symmetry,zhu2017global} are restricted to the noiseless setting, i.e., $\nabla f(\X^*) = 0$, while our results cover the noisy case. For the noisy setting it was essential to carry out the fine-grained analysis in Proposition \ref{prop: hH-grad-hessian-connection} in Section \ref{sec: H-landscape-analysis} when trying to connect the landscape of optimization problems \eqref{eq: quotient-factorization-objective} and \eqref{eq: quotient-matrix-fac-denoising}. Second, the escaping directions proposed in \cite{li2016symmetry,zhu2017global} cannot be used in our Riemannian setting as those vectors do not lie in the horizontal space $\cH_{\Y} \widebar{\cM}_{r+}^{q}$. A natural idea to attempt to fix this issue is to project the proposed vectors onto $\cH_{\Y} \widebar{\cM}_{r+}^{q}$. Unfortunately, it is not clear that the projected vectors are still escaping directions. As a result, the regions $\cR_2$ and $\cR_3$ that we define are different from the ones in \cite{li2016symmetry,zhu2017global}. The specific escaping direction $\theta_\Y$ that we present in Theorem \ref{th: h-negative-eigenvalue-region} is inspired by \cite{ge2017no}, where the authors studied the landscape of \eqref{eq: factorization-objective} at stationary points under the Euclidean geometry. Here we show that their choice is actually an escaping direction under the Riemannian quotient geometry considered in this paper.  
\end{Remark}
}

{We emphasize that one major advantage of the landscape analysis in Corollary \ref{coro: benign-landscape-f-well-conditioned}, which is also why we think our results are more useful than existing ones, is that our analysis directly implies fast convergence of common optimization algorithms by harnessing existing results in the literature. For example, the local geodesic strong convexity and smoothness properties that we deduce immediately suggest that if there exists a Riemannian FOSP $[\widehat{\Y}]$ in $\cR_1$, then vanilla GD initialized in $\cR_1$ will stay in $\cR_1$ and converge linearly to $[\widehat{\Y}]$, as it would follow from the proof of \cite[Theorem 11.29]{boumal2020introduction}; see also \cite[Proposition 3.14]{goyens2024riemannian} for a recent result on the Riemannian trust region method. Furthermore, \cite{goyens2024riemannian} studies the global convergence of Riemannian trust region methods for generic Riemannian optimization problems with a robust strict saddle property. In particular, they show that for those types of problems the standard Riemannian trust region method can find approximate local minimizers in a number of iterations that depends double-logarithmically on the accuracy parameter; see Theorem 3.16 and Theorem 3.18 in \cite{goyens2024riemannian} for precise statements. Finally, in terms of first-order methods, perturbed (Riemannian) gradient descent with local improvement is also guaranteed to escape saddle points and exhibit fast global convergence under the robust strict saddle property \cite{ge2015escaping,jin2017escape,sun2019escaping,criscitiello2019efficiently}. 

To wrap up this section, we would like to take the matrix approximation problem \eqref{eq: quotient-matrix-fac-denoising} as a simple but concrete example to illustrate the implications of our landscape analysis on fast \textit{global} convergence of perturbed Riemannian gradient descent (PRGD) with local improvement. We note that the existing guarantees on PRGD's ability to escape saddle points rely on curvature and Lipschitz conditions of the gradient and Hessian \cite{criscitiello2019efficiently,sun2019escaping} of the objective function. However, since our manifold $\cM_{r+}^q$ is unbounded and not complete, in order to satisfy those conditions we have to ensure that the whole trajectory of PRGD stays away from the boundary of the manifold $\cM_{r+}^q$ and within a bounded set (see Lemma \ref{lm:curvature-lipschitz-constants} in Appendix \ref{proof:app-preliminary}). These conditions are difficult to check for the original PRGD algorithm, which motivates the slight modification appearing in Algorithm \ref{alg:PRGD}. The details for the modular PRGD algorithm in Algorithm \ref{alg:PRGD} are provided in Appendix \ref{sec:PRGD}. 
\begin{algorithm}[h]
\caption{Perturbed Riemannian Gradient Descent with Initial Safeguard and Local Improvement}
\noindent {\bf Input}: $\widetilde{\Y}_0 \in \bbR_{*}^{p \times r} $, two positive integers $T_0$ and $T_1$, step size $\eta$, Lipschitz constants $L, \rho$, curvature bound $K$, injectivity radius bound $\mathfrak{I}$, PRGD accuracy $\widetilde{\epsilon}$, success probability $\zeta$
\begin{algorithmic}[1]
	\For{$t=0, 1, \ldots, T_0-1$}
	\State $\widetilde{\Y}_{t + 1} = \widetilde{\Y}_{t} -   \eta (\widetilde{\Y}_t \widetilde{\Y}_t^\top - \X^*  ) \widetilde{\Y}_t$
	\EndFor
	\State Run perturbed Riemannian gradient descent $\Y_0 = \text{PRGD}(\widetilde{\Y}_{T_0}, L, \rho, K, \mathfrak{I}, \widetilde{\epsilon}, \delta)$ (see details in Appendix \ref{sec:PRGD})
		\For{$t=0, 1, \ldots, T_1-1$}
	\State $\Y_{t + 1} = \Y_{t} -   \eta (\Y_t \Y_t^\top - \X^*  )\Y_t$
	\EndFor
\end{algorithmic}
\noindent {\bf Output}: $\Y_{T_1}$.
\label{alg:PRGD}
\end{algorithm}

We have two quick remarks regarding Algorithm \ref{alg:PRGD}. First, since the Riemannian gradient is exactly the same as the Euclidean gradient, the algorithm we present can be viewed as the Riemannian gradient algorithm in the lifted horizontal space. Second, compared with the perturbed GD with local improvement \cite{jin2017escape} (Algorithm 3 there), here we initialize PRGD by running a few gradient steps to ensure that the iterates enter a nice region where their spectrum is well controlled. Based on this condition, we are able to show that the rest of the trajectory of PRGD will continue to stay in this nice region, facilitating the application of existing PRGD guarantees \cite{criscitiello2019efficiently,sun2019escaping}. We also note that our proof is modular in the sense that, as long as PRGD enters this nice region, then PRGD will stay there and exhibit fast convergence. Running a few gradient steps at the beginning is one provable and concrete way to ensure PRGD enters the nice region \cite{chen2023fast}, but there could be different ways to achieve this. The convergence guarantee of Algorithm \ref{alg:PRGD} is provided in Theorem \ref{th:application}, whose proof can be found in Appendix \ref{app:DegeneracyIterates}.

\begin{Theorem} \label{th:application}
	Consider the matrix approximation objective \eqref{eq: quotient-matrix-fac-denoising}. Recall $\U^* \in \st(r,p)$ spans the top $r$ left singular subspace of $\Y^*$ and $\U^*_{\perp}\in \st(p-r,p)$ is its orthogonal complement. Given any initialization $\widetilde{\Y}_0$, let $\K_0 = \U^{*\top } \widetilde{\Y}_0$ and $\J_0 = \U^{*\top }_{\perp} \widetilde{\Y}_0$. Suppose $\sigma_1(\widetilde{\Y}_0) \leq \sigma_1(\Y^*)$, $\sigma_1(\J_0) \leq \sigma_r(\Y^*)/2$, $0 < \sigma_r(\K_0) \leq \sigma_r(\Y^*)/2$ and $\sigma_1^2(\J_0) \leq \frac{\sigma^2_r(\Y^*)}{8 \sigma_1(\Y^*)} \sigma_r(\K_0)$. Then, with probability $1 - \zeta$, Algorithm \ref{alg:PRGD} with initialization $\widetilde{\Y}_0$, properly chosen $T_0, T_1, \eta$, and input parameters $(L, \rho, K, \mathfrak{I}, \widetilde{\epsilon},\zeta)$ for PRGD, can output $\Y_{T_1}$ with $d([\Y_{T_1}], [\Y^*]) \leq \epsilon$. Moreover, the number of calls to the Riemannian gradient during the whole algorithm is at most $C(\log(1/\sigma_r(\K_0)) + \log^4(p) + \log(1/\epsilon))$ for some universal constant $C$ depending only on $\sigma_r(\Y^*), \sigma_1(\Y^*)$, $r$ and $\zeta$.
\end{Theorem}

A few remarks on Theorem \ref{th:application} are provided next. First, the complexity of Algorithm \ref{alg:PRGD} consists of three terms. The first complexity term $\mathcal{O}(\log(1/\sigma_r(\K_0)))$ comes from the first stage of the algorithm and it ensures that after this many steps of Riemannnian GD, the iterates enter a nice region where their spectra are controlled; the complexity term $ \mathcal{O}(\log^4(p))$ is the number of iterations PRGD takes to escape saddle points and enter the local region $\cR_1$; the last complexity term $\mathcal{O}(\log(1/\epsilon))$ is the number of iterations Riemannnian GD needs to achieve $\epsilon$ accuracy in the local region; we refer readers to Appendix \ref{sec:proof-init}-\ref{sec:proof-local} for the detailed analysis of these three stages. Second, the current guarantee of PRGD in Theorem \ref{th:application} is for the matrix approximation objective \eqref{eq: quotient-matrix-fac-denoising}. This is because the gradient formula there is explicit, and it is relatively easy to control the spectrum of PRGD iterates. Extending such a result to a general $f$ satisfying RSC and RSM will be more involved, and we leave it to future exploration. Finally, we note that the initialization condition required for $\widetilde{\Y}_0$ in Theorem \ref{th:application} can be easily satisfied with random initialization and the order of $\sigma_r(\K_0)$ can be chosen to be $1/p$; see Lemma \ref{lem:RandomIntial} in Appendix \ref{sec:random-initialization} for details. Therefore, with random initialization, the overall computational complexity of Algorithm \ref{alg:PRGD} to achieve $\epsilon$ accuracy is $\mathcal{O}(\log(1/\epsilon) + \log^4(p))$ .
}

\section{Local Landscape Analysis of \eqref{eq: quotient-factorization-objective} When $f$ Satisfies Restricted Strict Convexity Property} \label{sec: landscape-convex-f}
In this section we describe the landscape of \eqref{eq: quotient-factorization-objective} under a weaker assumption on $f$. In particular, we assume $f$ satisfies the following restricted strict convexity property. 
\begin{Definition} \label{def: restricted-strictly-convex} We say $f : \bbR^{p \times p} \to \bbR$ satisfies the ($r,2r$)-restricted strict convexity property if for any $\X, \G \in \bbR^{p \times p}$ with $\rank(\X) \leq r$ and $\rank(\G) \leq 2r$, the Euclidean Hessian of $f$ satisfies $ \nabla^2 f(\X)[\G, \G] > 0$. \nc
\end{Definition}

It is clear that if $f$ satisfies $(2r,4r)$-restricted \textit{strong} convexity and smoothness property (for some $\delta \in [0,1)$), then it satisfies $(r, 2r)$-restricted strict convexity property. Under this weaker assumption on $f$, we prove the following local geometric landscape results for $h([\Y])$. \nc
\begin{Theorem}(Local Landscape of $h([\Y])$) \label{th: local-landscape-f-convex} Suppose $f$ satisfies the $(r,2r)$-restricted strict convexity property. Consider $\widehat{\Y} \in \bbR^{p \times r}$ such that $\widehat{\Y} \widehat{\Y}^\top$ is a local minimizer of $\min_{ \bbS^{p \times p} \ni \X \succcurlyeq 0, \rank(\X) \leq 2r  } f(\X)$ with rank $r$. Then there exists a neighborhood around $[\widehat{\Y}]$ on which $h([\Y])$ is geodesically convex.
\end{Theorem}

We note that, compared to the results in Section \ref{sec: landscape-analysis-f-well-conditioned}, the landscape analysis in Theorem \ref{th: local-landscape-f-convex} is local and it is challenging to work out the explicit local geodesic convexity radius in this setting. {Also, we highlight that the optimality condition for $\widehat{\Y} \widehat{\Y}^\top$ stated in Theorem \ref{th: local-landscape-f-convex} is slightly stronger than the local optimality condition in \eqref{eq: original-formulation}. }

\section{Global Landscape Analysis of $H([\Y])$ in \eqref{eq: quotient-matrix-fac-denoising}} \label{sec: H-landscape-analysis}
In this section, we provide the global landscape analysis of \eqref{eq: quotient-matrix-fac-denoising} under the Riemannian quotient geometry. This result is critical in establishing the global landscape analysis of \eqref{eq: quotient-factorization-objective} when $f$ satisfies RSC and RSM.

Recall $\X^*$ is a rank $r$ PSD matrix. Let us denote $\X^* = \Y^* \Y^{*\top}$ and let $\kappa^* = \sigma_1(\Y^*)/\sigma_r(\Y^*)$ be the condition number of $\Y^*$. First, it is clear $[\Y^*]$ is the unique global minimizer of \eqref{eq: quotient-matrix-fac-denoising}. Next, by Lemma \ref{lm: gradient-hessian-exp-PSD}, we have the following expressions of Riemannian gradient and Hessian of $H([\Y])$:
\begin{equation} \label{eq: gradient-Hessian-exp-H}
	\begin{split}
		\overline{\grad\, H([\Y])} &= 2(\Y \Y^\top -\X^*) \Y, \\
			\overline{\Hess \, H([\Y])}[\theta_\Y, \theta_\Y] 
		&=\|\Y\theta_\Y^\top + \theta_\Y \Y^\top \|_\F^2 + 2\langle \Y \Y^\top -\X^*, \theta_\Y \theta_\Y^\top \rangle.
	\end{split}
\end{equation} 

Next, we show for the special objective \eqref{eq: quotient-matrix-fac-denoising}, there is only one stationary point, which is $[\Y^*]$. 
\begin{Theorem} \label{th: unique-FOSP}
	$[\Y^*]$ is the unique Riemannian FOSP of \eqref{eq: quotient-matrix-fac-denoising}.
\end{Theorem}

The proof of Theorem \ref{th: unique-FOSP} is presented in Section \ref{proof-sec: H-landscape}.

Next, we show that the optimization problem \eqref{eq: quotient-matrix-fac-denoising} is geodesically strongly convex and smooth in a neighborhood of $\Y^*$.
\begin{Theorem}[Local Geodesic Strong Convexity of \eqref{eq: quotient-matrix-fac-denoising}] \label{th: H-geodesic-strong-convex} Suppose $ 0\leq  \mu \leq 1/3$. Then, for any $\Y \in \cR_1$,
\begin{equation*}
	\begin{split}
		\lambda_{\min}(\overline{\Hess\, H([\Y])} ) & \geq ( 2 \left( 1 - \mu/ \kappa^* \right)^2 - (14/3) \mu  ) \sigma^2_r(\Y^*) , \\
		 \lambda_{\max}(\overline{\Hess\, H([\Y])} ) & \leq   4\left(   \sigma_1(\Y^*) + \mu \sigma_r(\Y^*)/\kappa^* \right)^2 + 14\mu \sigma^2_r(\Y^*)/3.
	\end{split}
\end{equation*} In particular, if $\mu$ is further chosen such that $ \left( 1 - \mu /\kappa^* \right)^2 - 7\mu/3 > 0$, e.g., $\mu < \frac{13 - \sqrt{133}}{6}$ , we have $H([\Y])$ is geodesically strongly convex and smooth in $\cR_1$.
\end{Theorem}

In the next two theorems, we show that, for $\Y \notin \cR_1 $, either the Riemannian Hessian evaluated at $\Y$ has a large negative eigenvalue, or the norm of the Riemannian gradient is large.
\begin{Theorem}[Region with Negative Eigenvalue in the Riemannian Hessian of \eqref{eq: quotient-matrix-fac-denoising}] \label{th: H-negative-eigenvalue-region} Given any $\Y \in \cR_2$, let $\theta_\Y = \Y - \Y^* \Q $, where $\Q \in \bbO_r$ is the best orthonormal \nc matrix aligning $\Y^*$ and $\Y$. Then
$\overline{ \Hess\, H([\Y])}[ \theta_\Y, \theta_\Y  ] \leq  (\alpha - 2 ( \sqrt{2} - 1 ))  \sigma^2_r(\Y^*) \|\theta_\Y\|_\F^2$. In particular, if $\alpha < 2(\sqrt{2} - 1)$, we have $\overline{ \Hess\, H([\Y])}$ has at least one negative eigenvalue and $\theta_\Y$ is an escaping direction. 
\end{Theorem}

\begin{Theorem}[Regions with Large Riemannian Gradient of \eqref{eq: quotient-matrix-fac-denoising}] \label{th: H-large-gradient-norm}
\begin{equation*}
	\begin{split}
	(i) \quad	\|\overline{\grad\, H([\Y])} \|_\F &> \alpha \mu \sigma^3_r(\Y^*)/(4\kappa^*), \quad \forall \Y \in \cR_3'; \\
	(ii) \quad	\|\overline{\grad\, H([\Y])} \|_\F &\geq 2(\|\Y\|^3 - \|\Y\| \|\Y^*\|^2) > 2(\beta^3 - \beta) \|\Y^*\|^3, \quad \forall \Y \in \cR_3''; \\
	(iii) \quad	\langle \overline{\grad\, H([\Y])} , \Y \rangle & > 2 (1 - 1/\gamma )  \|\Y\Y^\top\|^2_\F, \quad \forall \Y \in \cR_3'''.
	\end{split}
\end{equation*}
	In particular, if $\beta > 1$ and $\gamma > 1$, we have the Riemannian gradient of $H([\Y])$ has large magnitude in all regions $\cR_3', \cR_3''$ and $\cR_3'''$.
\end{Theorem}

\begin{Remark}({\bf Comparison of Radii for the Positive Definiteness of Riemannian Hessians of \eqref{eq: quotient-matrix-fac-denoising} under Other Geometries}) \label{rm: comparison-convexity-radius}
	We note \eqref{eq: quotient-matrix-fac-denoising} can also be formulated as an optimization problem under the embedded geometry, 
	\begin{equation} \label{eq: embedded-formulation-matrix-denoising}
		 \min_{\X  \in \cM_{r+}^e } \widetilde{H}(\X) := \frac{1}{2} \| \X - \X^* \|_\F^2,
	\end{equation} where $ \cM_{r+}^e = \{ \X\in  \bbS^{p \times p}: \rank(\X) = r, \X \succcurlyeq 0 \}$ is an embedded submanifold in $\bbS^{p \times p}$ \cite{vandereycken2009embedded}. In Lemma \ref{lm: positive-definite-Hessian-radius-embedded-geometry} below, whose proof is presented in Appendix \ref{app:Lemma5}, we quantify the radius for the positive definiteness of the Riemannian Hessian under the embedded geometry.

	\begin{Lemma}(Radius for Positive Definiteness of $\Hess\,\widetilde{H}(\X)$ under the Embedded Geometry) \label{lm: positive-definite-Hessian-radius-embedded-geometry} Define $\cR_1' := \{ \X \in \cM_{r+}^e | \|\X - \X^*\|_\F \leq \mu' \sigma_r(\X^*) \}$, then for any $\X \in \cR_1'$, we have
	\begin{equation*}
		\Hess\, \widetilde{H}(\X)[\xi_\X, \xi_\X] \geq (1 - \frac{2\mu'}{1-\mu'} )\|\xi_\X\|_\F^2 , \quad \forall \xi_\X \in T_{\X} \cM_{r+}^e.
	\end{equation*}
	\end{Lemma}

	Suppose $\X$ and $\X^*$ have decompositions $\Y\Y^\top$ and $\Y^*\Y^{*\top}$, respectively. The condition $ \|\X - \X^*\|_\F \leq \mu' \sigma_r(\X^*)$ in Lemma \ref{lm: positive-definite-Hessian-radius-embedded-geometry} implies $d([\Y], [\Y^*]) \leq \sqrt{\frac{1}{2(\sqrt{2} - 1)}} \mu' \sigma_r(\Y^*) $ by Lemma \ref{lm: distUY-UUYY-transfer} Eq. \eqref{ineq: distUY-UUYY-transfer-1}. So, compared with the radius of $\cR_1$, the radius for the positive definiteness of Riemannian Hessian under the embedded geometry is in general bigger by a factor of the condition number of $\Y^*$.
	
	More generally, since the spectra of Riemannian Hessians of an optimization problem under two different geometries are sandwiched by each other at Riemannian FOSPs \cite{luo2021geometric}, the positive definiteness of the Riemannian Hessian at $\X^*$ under one geometry implies the positive definiteness of the Riemannian Hessian under any another geometry. Moreover, because there always exists a convex geodesic ball at every point for any Riemannian manifold \cite[Chapter 3.4]{do1992riemannian}, we have that under any geometry for fixed-rank PSD matrices, there exists a neighborhood around $\X^*$ such that the optimization problem \eqref{eq: embedded-formulation-matrix-denoising} is geodesically strongly convex. In Theorem \ref{th: H-geodesic-strong-convex}, we provide the geodesic strong convexity radius under $\cM_{r+}^q$, it is interesting to figure out the radius under other common geometries, such as the embedded one, and explore under which geometry, the geodesic strong convexity radius of \eqref{eq: embedded-formulation-matrix-denoising} achieves its maximum.
\end{Remark}

{Finally, we note that when $f$ satisfies RSC and RSM properties, then we can estimate the discrepancies between the Riemannian gradients and Hessians of $H([\Y])$ (as defined in \eqref{eq: quotient-matrix-fac-denoising}) and $h([\Y])$ (see \eqref{eq: quotient-factorization-objective}).
\begin{Proposition} \label{prop: hH-grad-hessian-connection}
	Suppose $f$ satisfies the $(2r,4r)$-restricted strong convexity and smoothness properties with parameter $\delta$ given in Definition \ref{def: RSC-RSM}. For any $\Y \in \bbR^{p \times r}_*$, we have
	\begin{equation} \label{ineq: gradient-connection}
		\| \overline{\grad\, H([\Y])} - \overline{\grad\, h([\Y])}  \|_\F \leq 2 \delta \|\Y\| \|\Y\Y^\top - \X^*\|_\F + 2 \|\Y\| \| ( \nabla f(\X^*) )_{\max(r)} \|_\F. 
	\end{equation}
	Moreover, for any $\theta_\Y \in \cH_\Y \widebar{\cM}_{r+}^{q}$, we have
	\begin{equation}
	\begin{split}
		&\left| \overline{\Hess\, H([\Y])}[\theta_\Y, \theta_\Y] - \overline{\Hess\, h([\Y])}[\theta_\Y, \theta_\Y]  \right|\\
		 \leq & \delta  \|\Y\theta_\Y^\top + \theta_\Y \Y^\top \|_\F^2 + 2 \delta \|\Y\Y^\top - \X^*\|_\F \|\theta_\Y \theta_\Y^\top\|_\F + 2 \| ( \nabla f(\X^*) )_{\max(r)} \|_\F \|\theta_\Y \theta_\Y^\top\|_\F.
	\end{split}
	\end{equation}
\end{Proposition}}

\section{Proofs Theorem \ref{th: unique-FOSP}-Theorem \ref{th: H-large-gradient-norm} in Section \ref{sec: H-landscape-analysis}} \label{proof-sec: H-landscape}

\subsection{Proof of Theorem \ref{th: unique-FOSP} }

	Recall $\X^*$ has eigendecomposition $\U^* \bSigma^* \U^{*\top}$, where recall by this we mean the economic/reduced version of the eigendecomposition\nc.  For any FOSP $\Y \in \bbR^{p \times r}_*$, let $\B = \U^{*\top} \Y$ and $\W = \Y - P_{\U^*}\Y$. Then $\Y = \U^* \B + \W $. So
	\begin{equation} \label{eq: gradient-decomposition}
		\begin{split}
			(\Y \Y^\top - \X^*) \Y &=  \left( (\U^* \B + \W)(\U^* \B + \W)^\top - \X^* \right) ( \U^* \B + \W ) \\
			& = \left( (\U^* \B + \W)(\U^* \B + \W)^\top - \X^* \right) \U^* \B+ \left( (\U^* \B + \W)(\U^* \B + \W)^\top - \X^* \right)  \W \\
			& \overset{ \W^\top \U^* = 0 } = ( \U^* \B + \W ) \B^\top \B - \U^* \bSigma^* \B + ( \U^* \B + \W ) \W^\top \W.
		\end{split}
	\end{equation}

Since $\Y$ is a FOSP, we have $(\Y \Y^\top - \X^*) \Y=0$. We split the analysis into the following three cases. \nc

{\bf Case 1: } $\W = \0$. In this case, $\Y = \U^* \B \in \bbR^{p \times r}_*$, so $\B$ is invertible and $\Y$ lies in the column span of $\U^*$. The fact that $(\Y \Y^\top - \X^*) \Y = \0$ and \eqref{eq: gradient-decomposition} imply $$\U^*(\B \B^\top - \bSigma^*) \B = \0 \Longrightarrow \B \B^\top - \bSigma^* = \0.$$
Thus, $\Y \Y^\top = \U^* \B \B^\top \U^{*\top} = \U^* \bSigma^* \U^{*\top} = \X^*$, which implies $[\Y] = [\Y^*]$ by Lemma \ref{lm: completeness-quotient-PSD}.

{\bf Case 2: } $\W \neq \0$. Since $\W$ lies in the column space of $\U^*_\perp$,  $(\Y \Y^\top - \X^*) \Y = \0$ and \eqref{eq: gradient-decomposition} imply $$P_{\U^*_\perp} (\Y \Y^\top - \X^*) \Y = \0 \Longrightarrow \W (\B^\top \B + \W^\top \W)  = \0.$$ Since $\Y^\top \Y = \B^\top \B + \W^\top \W$, we have $\W \Y^\top \Y = \0$. As $\Y \in \bbR^{p \times r}_*$, $\Y^\top \Y$ is positive definite, this implies $\W = \0$ and contradicts with the assumption. So $\W \neq \0$ can not happen at FOSPs.

From the above, we deduce that only Case 2 can happen at FOSPs and thus $[\Y^*]$ is the unique FOSP in \eqref{eq: quotient-matrix-fac-denoising}.

\subsection{Proof of Theorem \ref{th: H-geodesic-strong-convex}}
Denote by $\Q$ the best orthonormal \nc matrix that aligns $\Y$ and $\Y^*$, i.e., $\Q^* \in \argmin_{\Q \in \bbO_r} \|\Y^* \Q - \Y\|_\F$. Then by the assumption on $\Y$ we have 
\begin{equation} \label{ineq: Y-Y*-dist-bound}
 	\| \Y - \Y^* \Q \| \nc \leq	\| \Y - \Y^* \Q \|_\F =  d([\Y],[\Y^*]) \leq \mu  \sigma_r(\Y^*)/\kappa^*.  
\end{equation}
Thus 
\begin{equation} \label{ineq: spectrum-bound-of-Y}
	\begin{split}
		\sigma_r(\Y) &= \sigma_r(\Y - \Y^* \Q + \Y^* \Q ) \geq \sigma_r(\Y^*) - \|\Y - \Y^* \Q\|  \overset{ \eqref{ineq: Y-Y*-dist-bound} } \geq (1 - \mu/ \kappa^* ) \sigma_r(\Y^*),\\
		\sigma_1(\Y) &= \sigma_1(\Y - \Y^* \Q + \Y^* \Q  ) \leq \sigma_1(\Y^*) + \|\Y - \Y^* \Q\| \overset{ \eqref{ineq: Y-Y*-dist-bound} } \leq \sigma_1(\Y^*) +  \mu \sigma_r(\Y^*)/ \kappa^*,
	\end{split}
\end{equation}
 where the first inequality in each line follows from Weyl's theorem \cite[Theorem  4.29]{stewart1998matrix}.\nc

	To provide bounds for the spectrum of $\overline{\Hess\, H([\Y])}$, we just need to compute lower and upper bounds for $\overline{\Hess\, H([\Y])}[\theta_\Y, \theta_\Y]$ for any $\theta_\Y \in \cH_\Y \widebar{\cM}_{r+}^{q}$. First,
	\begin{equation*}
		\begin{split}
			\overline{ \Hess\, H([\Y])}[\theta_\Y, \theta_\Y] \overset{ \eqref{eq: gradient-Hessian-exp-H}} = & \|\Y\theta_\Y^\top + \theta_\Y \Y^\top \|_\F^2 + 2\langle \Y \Y^\top -\X^*, \theta_\Y \theta_\Y^\top \rangle \\
			\overset{ \text{Lemma } \ref{lm: norm-bound-Ytheta-thetaY} } \geq & 2 \sigma^2_r(\Y) \|\theta_\Y\|_\F^2  +  2\langle \Y \Y^\top -\X^*, \theta_\Y \theta_\Y^\top \rangle \\
			\geq & 2 \sigma^2_r(\Y) \|\theta_\Y\|_\F^2 - 2 \| \Y \Y^\top -\X^*\|_\F \|\theta_\Y\|_\F^2 \\
			\overset{ \eqref{ineq: spectrum-bound-of-Y}, \text{Lemma } \ref{lm: distUY-UUYY-transfer} \text{ Eq. } \eqref{ineq: distUY-UUYY-transfer-3} } \geq & 2 \left( 1 - \mu/ \kappa^* \right)^2 \sigma^2_r(\Y^*)\|\theta_\Y\|_\F^2 - 2*7/3 \|\Y^*\|  \mu  \sigma_r(\Y^*) \|\theta_\Y\|_\F^2/ \kappa^*\\
			\geq & ( 2 \left( 1 - \mu/ \kappa^* \right)^2 - (14/3) \mu  ) \sigma^2_r(\Y^*) \|\theta_\Y\|_\F^2. 
		\end{split} 
	\end{equation*}
	Likewise, 
	\begin{equation*}
		\begin{split}
			\overline{ \Hess\, H([\Y])}[\theta_\Y, \theta_\Y] \overset{ \eqref{eq: gradient-Hessian-exp-H}} = & \|\Y\theta_\Y^\top + \theta_\Y \Y^\top \|_\F^2 + 2\langle \Y \Y^\top -\X^*, \theta_\Y \theta_\Y^\top \rangle \\
			\overset{ \text{Lemma } \ref{lm: norm-bound-Ytheta-thetaY} } \leq & (4 \sigma^2_1(\Y) + 2\| \Y \Y^\top -\X^*\|_\F)  \|\theta_\Y\|_\F^2 \\
			\overset{ \eqref{ineq: spectrum-bound-of-Y}, \text{Lemma }\ref{lm: distUY-UUYY-transfer}  \text{ Eq. } \eqref{ineq: distUY-UUYY-transfer-3} }\leq &  \left( 4\left(   \sigma_1(\Y^*) + \mu \sigma_r(\Y^*)/\kappa^* \right)^2 + 14\mu \sigma^2_r(\Y^*)/3 \right)  \|\theta_\Y\|_\F^2.
		\end{split}
	\end{equation*}
	
	From the above we conclude that when $\mu$ is chosen such that $ 2 \left( 1 - \mu /\kappa^* \right)^2 - (14/3) \mu > 0$, we have $H([\Y])$ in \eqref{eq: quotient-matrix-fac-denoising} is geodesically strongly convex and smooth in $\cR_1$ as $\cR_1$ is a geodesically convex set by Theorem \ref{th: convexity-radius-Mq}.

\subsection{Proof of Theorem \ref{th: H-negative-eigenvalue-region}}
	First, notice $\theta_\Y \in \cH_{\Y}  \widebar{\cM}_{r+}^q $ and  $\|\theta_\Y \|_\F = d([\Y], [\Y^*])$ by Lemma \ref{lm: logarithm-map}. In addition, a simple calculation yields
	\begin{equation} \label{eq: dist-relation}
		\Y \Y^\top - \X^* + \theta_\Y \theta_\Y^\top = \Y \theta_\Y^\top + \theta_\Y \Y^\top.
	\end{equation}
	
 Then \eqref{eq: gradient-Hessian-exp-H} implies 
	\begin{equation} \label{eq: H-gradient-quadratic}
	\begin{split}
		\langle\overline{\grad\, H([\Y])}, \theta_\Y \rangle &= \langle 2 (\Y \Y^\top - \X^* ) \Y, \theta_\Y \rangle  \\
		& = \langle \Y \Y^\top - \X^*, \theta_\Y \Y^\top + \Y \theta_\Y^\top \rangle \overset{ \eqref{eq: dist-relation} }  = \langle  \Y \Y^\top - \X^* , \theta_\Y \theta_\Y^\top + \Y \Y^\top -\X^* \rangle,
	\end{split}
	\end{equation} and 
	\begin{equation*}
		\begin{split}
			\overline{ \Hess\, H([\Y])}[\theta_\Y, \theta_\Y] = & \|\Y\theta_\Y^\top + \theta_\Y \Y^\top \|_\F^2 + 2\langle \Y \Y^\top -\X^*, \theta_\Y \theta_\Y^\top \rangle \\
			\overset{ \eqref{eq: dist-relation} } = & \|  \Y \Y^\top - \X^* + \theta_\Y \theta_\Y^\top \|_\F^2 +    2\langle \Y \Y^\top -\X^*, \theta_\Y \theta_\Y^\top \rangle \\
			= & \|\theta_\Y \theta_\Y^\top \|_\F^2 + \|  \Y \Y^\top - \X^* \|_\F^2 + 4 \langle \Y \Y^\top -\X^*, \theta_\Y \theta_\Y^\top \rangle \\
			= & \|\theta_\Y \theta_\Y^\top \|_\F^2 - 3 \| \Y \Y^\top - \X^* \|_\F^2 + 4 \langle   \Y \Y^\top -\X^* , \Y \Y^\top -\X^* + \theta_\Y \theta_\Y^\top \rangle \\
			\overset{ \eqref{eq: H-gradient-quadratic} } = &  \|\theta_\Y \theta_\Y^\top \|_\F^2 - 3 \| \Y \Y^\top - \X^* \|_\F^2 + 4 \langle\overline{\grad\, H([\Y])}, \theta_\Y \rangle \\
		 \overset{\text{Lemma } \ref{lm: distUY-UUYY-transfer} \text{ Eq. } \eqref{ineq: distUY-UUYY-transfer-2}}	\leq & - \|\Y \Y^\top - \X^*\|_\F^2   + 4 \langle\overline{\grad\, H([\Y])}, \theta_\Y \rangle \\
		\overset{\text{Lemma } \ref{lm: distUY-UUYY-transfer} \text{ Eq. } \eqref{ineq: distUY-UUYY-transfer-1}} \leq & - 2 ( \sqrt{2} - 1 ) \sigma^2_r(\Y^*) \|\theta_\Y\|_\F^2 + 4 \| \overline{\grad\, H([\Y])}  \|_\F \|\theta_\Y\|_\F \\
		 \overset{(a)}\leq & (\alpha - 2 ( \sqrt{2} - 1 ))  \sigma^2_r(\Y^*) \|\theta_\Y\|_\F^2,
		\end{split}
	\end{equation*} here (a) is because $\|\overline{\grad\, H([\Y])} \|_\F \leq  \alpha \mu \sigma^3_r(\Y^*)/(4\kappa^*) $ implies $$\|\overline{\grad\, H([\Y])} \|_\F \leq \alpha d([\Y],[\Y^*]) \sigma^2_r(\Y^*)/4 =  \alpha \|\theta_\Y\|_\F \sigma^2_r(\Y^*)/4.$$ This finishes the proof of this theorem.

\subsection{Proof of Theorem \ref{th: H-large-gradient-norm}}
	First, if $\Y \in \cR_3'$, the fact that $\|\overline{\grad\, H([\Y])} \|_\F$ is large holds by definition. Next, we show that the norm of the gradient is large when $\Y \in \cR_3''$ and when $\Y \in \cR_3'''$.

{\bf (When $\Y \in \cR_3''$):} Suppose $\U \bSigma \V^\top$ and $\U^* \bSigma^* \V^{*\top}$ are the SVDs of $\Y$ and $\Y^*$, respectively. Then
\begin{equation*}
	\begin{split}
		\| \overline{\grad\, H([\Y])}\|_\F \overset{ \eqref{eq: gradient-Hessian-exp-H}}= & 2\|(\Y\Y^\top - \Y^* \Y^{*\top}) \Y\|_\F \\
		= & 2 \| \U \bSigma^3 \V^\top - \U^* \bSigma^{*2} \U^{*\top} \U \bSigma \V^\top \|_\F \\
		= & 2\| \U \bSigma^3 - \U^* \bSigma^{*2} \U^{*\top} \U \bSigma \|_\F \\
		\geq & \| P_\U (\U \bSigma^3 - \U^* \bSigma^{*2} \U^{*\top} \U \bSigma )   \|_\F \\
		= & 2\| \bSigma^3 -  \U^\top \U^*  \bSigma^{*2} \U^{*\top} \U \bSigma  \|_\F \\
		\geq & 2\left| (\bSigma^3 -  \U^\top \U^*  \bSigma^{*2} \U^{*\top} \U \bSigma)_{[1,1]} \right| = 2\left| \|\Y\|^3 - (\U^\top \U^*  \bSigma^{*2} \U^{*\top} \U \bSigma)_{[1,1]} \right| \\
		\geq & 2(\|\Y\|^3 - \left| (\U^\top \U^*  \bSigma^{*2} \U^{*\top} \U \bSigma)_{[1,1]}\right|)\\
		\geq & 2(\|\Y\|^3 - \|\U^\top \U^*  \bSigma^{*2} \U^{*\top} \U \bSigma\|)\\
		\geq & 2(\|\Y\|^3 - \|\Y\| \|\Y^*\|^2) \overset{(a)} > 2(\beta^3 - \beta) \|\Y^*\|^3,
	\end{split}
\end{equation*} here (a) is because $\Y \in \cR_3''$, and the subscript $[1,1]$ indicates the entry in the first row and first column of the corresponding matrix\nc.
	
{\bf (When $\Y \in \cR_3'''$):} In this case, we have
\begin{equation*} 
	\begin{split}
		\langle \overline{\grad\, H([\Y])} , \Y \rangle \overset{ \eqref{eq: gradient-Hessian-exp-H}}= & \langle 2(\Y \Y^\top -\X^*) \Y , \Y \rangle  \\
		= & \langle 2(\Y \Y^\top -\X^*) , \Y\Y^\top \rangle  \\
		\geq & 2 \|\Y\Y^\top\|_\F^2 - 2|\langle \X^* , \Y\Y^\top \rangle| \geq 2 \|\Y\Y^\top\|_\F^2 - 2  \|\Y\Y^\top\|_\F \|\Y^*\Y^{*\top}\|_\F\\
		> & 2 (1 - 1/\gamma )  \|\Y\Y^\top\|^2_\F.
	\end{split}
\end{equation*}

\section{Proof of Theorem \ref{th: convexity-radius-Mq} } \label{sec: proof-convex-radius}
We first present two preliminary results for proving Theorem \ref{th: convexity-radius-Mq} and then present the proof of Theorem \ref{th: convexity-radius-Mq} at the end of this section. The first result is about the totally normal neighborhood of $\cM_{r+}^q$ at $[\Y]$.

\begin{Lemma}(Totally Normal Neighborhood at $[\Y ]$) \label{lm: totally-normal-neigh}
	For any $\Y \in \bbR^{p \times r}_*$ and $0 < x \leq \frac{1}{3} \sigma_r(\Y)$, $B_x([\Y])$ is a totally normal neighborhood of $[\Y]$ with parameter $\zeta = \sigma_r(\Y) - x$, i.e., for any $[\Y'] \in B_x([\Y])$, we have the injectivity radius at $[\Y']$ is greater or equal to $\zeta$ and $B_\zeta([\Y']) \supset B_{x}([\Y])$.
\end{Lemma}
\begin{proof} See Appendix \ref{proof-sec: lm-totally-normal-neigh}. \end{proof}

Next, we present a key proposition in proving Theorem \ref{th: convexity-radius-Mq} which is based on \cite[Chapter 3.4, Lemma 4.1 and Proposition 4.2]{do1992riemannian}. The result in \cite{do1992riemannian}, which is a classic result in Riemannian geometry, holds for generic Riemannian manifolds, but it is only when revisiting its proof that we can provide an \textit{explicit} quantitative estimate for the radius of geodesic convexity around an arbitrary point in the manifold $\cM_{r+}^q$. \nc  

We introduce the following notation. Given any $\Y \in \bbR^{p \times r}_*$ and $x > 0$ we use $S_x([\Y]) := \{ [\Y_1]: d([\Y_1], [\Y]) = x \}$ to denote the geodesic sphere of radius $x$ centered at $[\Y]$.

\begin{Proposition} \label{prop: convexity-support-prop}
	Given $\Y \in \bbR^{p \times r}_*$, any geodesic in $\cM_{r+}^q$ that is tangent at $[\Y']$ to the geodesic sphere $S_\rho([\Y])$ of radius $\rho$ with $\rho < c_\Y:= \sigma_r(\Y)$ stays out of the geodesic ball $B_\rho([\Y])$ for some neighborhood of $[\Y']$.
\end{Proposition}
\begin{proof}
 Denote $T_1 S_{\rho}([\Y])$ as the unit tangent bundle restricted to geodesic sphere $S_{\rho}([\Y])$, that is:
	\begin{equation*}
		T_1 S_{\rho}([\Y]) = \left\{ ([\Y'], \xi_{[\Y']} ): [\Y'] \in  S_{\rho}([\Y]), \xi_{[\Y']} \in T_{[\Y']} \cM_{r+}^q, g_{[\Y'] }(\xi_{[\Y']}, \xi_{[\Y']} ) = 1 \right\}.
	\end{equation*}

Let $\gamma: I \times T_1 S_{\rho}([\Y]) \to \cM_{r+}^q$, $I = (-\epsilon, \epsilon )$ for some small enough $\epsilon > 0$, be a differentiable map such that $t \mapsto \gamma(t,[\Y'], \xi_{[\Y']} )$ is the geodesic that at the instant $t = 0$ passes through $[\Y']$ with velocity $\xi_{[\Y']}$, and $\|\xi_{\Y'}\|_\F= 1$. Since $[\Y'] \in S_{\rho}([\Y])$ and $\rho <  \sigma_r(\Y) $, by Lemma \ref{lm: positive-det}, $\Y^{'\top} \Y$ is nonsingular, so we can define $u(t,[\Y'], \xi_{[\Y']} ) = \log_{[\Y]}( \gamma(t,[\Y'], \xi_{[\Y']} )  )$. In addition, let $F: I \times T_1 S_{\rho}([\Y]) \to \bbR$ be 
\begin{equation} \label{eq: F-expression}
	\begin{split}
		F(t,[\Y'], \xi_{[\Y']} ) &= g_{[\Y]}( u(t,[\Y'], \xi_{[\Y']} ), u(t,[\Y'], \xi_{[\Y']} ) ) \\
		&= \bar{g}_\Y\left( \overline{\log_{[\Y]} ( \gamma(t,[\Y'], \xi_{[\Y']} )  ) }, \overline{\log_{[\Y]} ( \gamma(t,[\Y'], \xi_{[\Y']} )  ) } \right)\\
		& \overset{\text{Lemma } \ref{lm: logarithm-map} }=\bar{g}_\Y\left(  (\Y' + t \xi_{\Y'} )\Q_t - \Y, (\Y' + t \xi_{\Y'} )\Q_t - \Y \right) \\
		& = \|(\Y' + t \xi_{\Y'} )\Q_t - \Y\|_\F^2\\
		& = t^2 + 2t \langle \Y', \xi_{\Y'} \rangle- 2 \tr(\bSigma_t) +  \|\Y'\|_\F^2 + \|\Y\|_\F^2,
	\end{split}
\end{equation} where $\Y^\top (\Y' + t \xi_{\Y'} )$ has SVD $\Q_{Ut} \bSigma_t \Q_{Vt}^\top$ and $\Q_t = \Q_{V t} \Q_{Ut}^\top$. Geometrically, $F$ measures the square distance from $[\Y]$ to points along the geodesic $\gamma$. We now discuss how the function $F$ behaves around $t=0$.

We have both $u$ and $F$ are differentiable and $\frac{\partial F}{\partial t} = 2 g_{[\Y]}\left( \partial u/\partial t, u  \right) $. If a geodesic $\gamma$ is tangent to the geodesic sphere $S_\rho([\Y])$ at the point $[\Y'] = \gamma(0, [\Y'], \xi_{[\Y']})$, then from the Gauss Lemma \cite[Chapter 3, Lemma 3.5]{do1992riemannian}, we have 
\begin{equation*}
	\frac{\partial F}{\partial t}( 0, [\Y'], \xi_{[\Y']} ) = 2 g_{[\Y]}\left( \frac{\partial u}{\partial t}(0, [\Y'], \xi_{[\Y']}) , u(0, [\Y'], \xi_{[\Y']}) \right) = 0.
\end{equation*} 
This means $(0, [\Y'], \xi_{[\Y']} )$ is a stationary point of $F$ at $t = 0$ for fixed $[\Y']$ on $S_r([\Y])$ and fixed $\xi_{[\Y']}$ with unit norm. If we can show that we have $\frac{\partial^2 F}{\partial t^2}(0,[\Y'], \xi_{[\Y']} ) > 0$, then we would be able to conclude that $t = 0$ is a local minimizer of $F(\cdot, [\Y'], \xi_{[\Y']} )$. Given the geometric interpretation of the function $F$, this would further imply that there exists a neighborhood of $[\Y']$ such that the geodesic $\gamma$ stays out of the geodesic ball $B_\rho([\Y])$. We thus focus on proving the positivity of the second derivative. 

Let $\Q_{Ut} = [\q^t_{u1}, \ldots, \q^t_{ur}]$, $\Q_{Vt} = [\q^t_{v1},\ldots, \q^t_{vr}]$, $\bSigma_t = {\rm diag}(\sigma_1, \ldots, \sigma_r)$. From \eqref{eq: F-expression}, for any $([\Y'], \xi_{[\Y']}) \in T_1 S_\rho([\Y])$ we have
\begin{equation}\label{ineq: sec-deriv-F-upper-bound}
	\begin{split}
		\frac{\partial^2 F}{\partial t^2}(0,[\Y'], \xi_{[\Y']} ) &= 2  - 2 \frac{d^2 \tr(\bSigma_t)}{dt^2}\Big|_{t=0} \\
		& =   2  - 2  \sum_{i=1}^r \frac{d^2 \sigma_i(t)}{dt^2}\Big|_{t=0} \\
		& \overset{ \text{Lemma }\ref{lm: variation-formula-for-singular-values} } = 2 - 2 \left\langle \Y^\top \xi_{\Y'}, \frac{d (\sum_{i=1}^r \q^t_{ui} \q^{t\top}_{vi}) }{dt}\Big|_{t=0}  \right \rangle\\
		& =  2 - 2 \left\langle \Y^\top \xi_{\Y'}, \frac{d (\Q_{U_t} \Q_{Vt}^\top ) }{dt}\Big|_{t=0}  \right \rangle.
	\end{split}
\end{equation} Since $\Q_{Ut} \Q_{Vt}^\top$ forms a smooth path in the space of $\bbO_r$, we must have that $\frac{d (\Q_{U_t} \Q_{Vt}^\top ) }{dt}\Big|_{t=0} \in T_{\Q_{U0} \Q_{V0}^\top} \bbO_r$. From the known expression for the tangent spaces of the Lie group of orthonormal matrices (e.g., see \cite{edelman1998geometry}), we deduce that $\frac{d (\Q_{U_t} \Q_{Vt}^\top ) }{dt}\Big|_{t=0} = \Q_{U0} \Q_{V0}^\top \bOmega$ for some skew-symmetric matrix $\bOmega$, i.e., $\bOmega = - \bOmega^\top$. Therefore,
\begin{equation} \label{ineq: sec-deriv-upper-bound-inner-part}
	\begin{split}
		\left|\left\langle \Y^\top \xi_{\Y'}, \frac{d (\Q_{U_t} \Q_{Vt}^\top ) }{dt}\Big|_{t=0}  \right \rangle \right| &= \left| \langle \Y^\top \xi_{\Y'},   \Q_{U0} \Q_{V0}^\top \bOmega\rangle  \right| \\
		& =  \left| \langle \Q_{V0} \Q_{U0}^\top \Y^\top \xi_{\Y'},  \bOmega\rangle  \right|\\
		& \overset{(a)}=  \left| \langle( \Q_{V0} \Q_{U0}^\top \Y^\top - \Y^{'\top}) \xi_{\Y'},  \bOmega\rangle  \right|\\
		& \leq \|( \Q_{V0} \Q_{U0}^\top \Y^\top - \Y^{'\top}) \xi_{\Y'}\|_\F \|\bOmega\|_\F\\
		& \overset{(b)}\leq \|( \Q_{V0} \Q_{U0}^\top \Y^\top - \Y^{'\top}) \|_\F \|\bOmega\|_\F \\
		& \overset{(c)}= \rho \|\bOmega\|_\F  = \rho \|\Q_{U0} \Q_{V0}^\top\bOmega\|_\F\\
		& \overset{\text{Lemma } \ref{lm: procrustes-condition-number} } \leq \frac{ \sqrt{2}\rho \|\xi_{\Y'} \|_\F }{(\sigma^2_r(\Y ) + \sigma^2_{r-1}( \Y))^{1/2}} \\
		& \leq \frac{\rho }{\sigma_r(\Y)} < 1,\\
	\end{split}
\end{equation} here (a) is because $\xi_{\Y'} \in \cH_{\Y'} \widebar{\cM}_{r+}^q$, and by Lemma \ref{lm: psd-quotient-manifold1-prop} we know $\Y^{'\top} \xi_{\Y'}$ is a symmetric matrix, denoted by $\S$, and thus $\langle \S, \bOmega \rangle = 0$ given that $\bOmega = - \bOmega^\top$; (b) is because $ \|( \Q_{V0} \Q_{U0}^\top \Y^\top - \Y^{'\top}) \xi_{\Y'}\|_\F \leq \|( \Q_{V0} \Q_{U0}^\top \Y^\top - \Y^{'\top})\|_\F \|\xi_{\Y'}\|_\F$ and $\|\xi_{\Y'}\|_\F = 1$; and (c) is because $\Y' \in  S_{\rho}([\Y])$ and $ \Q_{V0} \Q_{U0}^\top$ is the optimal alignment between $\Y$ and $\Y'$. 

Combining \eqref{ineq: sec-deriv-F-upper-bound} and \eqref{ineq: sec-deriv-upper-bound-inner-part}, we have shown $\frac{\partial^2 F}{\partial t^2}(0,[\Y'], \xi_{[\Y']} ) > 0$ for any $([\Y'], \xi_{[\Y']}) \in T_1 S_\rho([\Y])$. This finishes the proof of this proposition as was discussed before.
\end{proof}

In Appendix \ref{proof-sec: add-lemma-convexity-radius}, we present a series of technical lemmas that we have used in the proof of Proposition \ref{prop: convexity-support-prop}. For example, in Lemma \ref{lm: variation-formula-for-singular-values} we provide a simple formula for the second derivatives of singular values of a smoothly varying one-parameter family of matrices. Moreover, in Lemma \ref{lm: procrustes-condition-number} we give a sharp bound on the condition number of the orthogonal Procrustes problem,  complementing in this way some results in the literature, e.g. \cite{soderkvist1993perturbation,dorst2005first}. 

We are now ready to present the proof of Theorem \ref{th: convexity-radius-Mq}.

\begin{proof}[Proof of Theorem \ref{th: convexity-radius-Mq}]
	First, by Lemma \ref{lm: totally-normal-neigh}, $B_{r_\Y }([\Y])$ is a totally normal neighborhood of $[\Y]$, so for any $[\Y_1], [\Y_2] \in B_{r_\Y }([\Y])$, there exists a unique minimizing geodesic $\gamma$ joining them, and its length is less than $2 r_\Y$. 
	Let us now consider $[\Y_1], [\Y_2] \in B_x([\Y]) \subseteq B_{r_\Y }([\Y])$. We show that $\gamma$ is contained in $B_{x }([\Y])$.
	\nc
	
	Notice that for any point $[\Y']$ on $\gamma$, $\min(d([\Y'],[\Y_1]), d([\Y'],[\Y_2])) < r_\Y$. Then we have
	\begin{equation} \label{ineq: tildeY-distance-bound}
	\begin{split}
		d([\Y'],[\Y]) &\leq \min( d([\Y'], [\Y_1]) + d([\Y_1],[\Y]) , d([\Y'],[\Y_2]) + d([\Y_2],[\Y])  )\\
		& < 2 r_\Y  = c'_\Y := 2\sigma_r(\Y)/3. 
	\end{split}
 	\end{equation}
	
	Let $[\widetilde{\Y}]$ be the point in $\gamma$ such that the maximum distance from $[\Y]$ to $\gamma$ is attained, and denote this distance by $\rho$. If $[\widetilde{\Y}]$ is either $[\Y_1]$ or $[\Y_2]$, then we are done. If not, we have that the points of $\gamma$ in any neighborhood of $[\widetilde{\Y}]$ remain in the closure of $B_{\rho}([\Y])$ while $\gamma$ is tangential to $S_\rho([\Y])$ at $[\widetilde{\Y}]$. Since $\rho < c_\Y:= \sigma_r(\Y)$ by \eqref{ineq: tildeY-distance-bound},  this contradicts Proposition \ref{prop: convexity-support-prop}. This finishes the proof.
\end{proof}

\section{Conclusion and Discussion} \label{sec: conclusion}
In this paper, we have studied the optimization landscape of the Burer-Monteiro factorized objective for a general fixed-rank PSD matrix optimization problem under the Riemannian quotient geometry. When $f$ satisfies the restricted strong convexity and smoothness properties, we show the landscape of the factorized objective is benign by characterizing its geometry in the entire domain. When $f$ satisfies a weaker restricted strict convexity property, we show there exists a neighborhood near the local minimizer such that the factorized objective is geodesically convex. 

There are many interesting extensions to the results in this paper to be explored in the future. First, the current requirement on $\delta$ to guarantee the benign global landscape of \eqref{eq: quotient-factorization-objective} in Corollary \ref{coro: benign-landscape-f-well-conditioned} may not be sharp and it would be interesting to explore whether we can establish similar landscape results with a sharper dependence on $\delta$. Second, it is well known that geometry plays a central role in Riemannian optimization. Picking a proper metric that maximizes the geodesic convexity radius of the objective is favorable. As we have mentioned in Remark \ref{rm: comparison-convexity-radius}, it would be interesting to explore under what geometries the geodesic strong convexity radii of \eqref{eq: quotient-factorization-objective} and \eqref{eq: quotient-matrix-fac-denoising} are maximized. Third, the landscape of the Burer-Monteiro factorized objective can be more complicated when $f$ does not satisfy RSC and RSM, see \cite{yalccin2022factorization}. An interesting future research direction is to explore under what assumptions on $f$ will the Burer-Monteiro factorization continue to work for efficient optimization. Finally, in this work, we have mainly focused on the exact-parameterization setting, i.e., the number of columns of $\Y$ is equal to the rank of the parameter of interest $\X^*$. It would be very interesting to try to characterize the landscape of \eqref{eq: quotient-factorization-objective} in the over-parameterized setting, i.e., when the number of columns of $\Y$ is greater than the rank of $\X^*$. In that case, the optimality conditions for $\X^*$ need to be carefully developed as $\X^*$ is merely a boundary point of the working manifold.

\section*{Acknowledgements}
We thank the Editor Kim-Chuan Toh, the Associate Editor, and two anonymous referees for their helpful suggestions, which helped improve the presentation and quality of this paper. This work was started while the authors were visiting the Simons Institute to participate in the programs ``Geometric Methods in Optimization and Sampling" and ``Computational Complexity of Statistical Inference" during the Fall of 2021. The authors would like to thank the institute for its hospitality and support. NGT was supported by NSF-DMS grant 2236447 and 2005797, and would also like to thank the IFDS at UW-Madison and NSF through TRIPODS grant 2023239 for their support. YL and NGT would like to thank Anru Zhang for helpful discussions during the initial stage of the project and Inge S{\"o}derkvist for helpful discussions of his paper \cite{soderkvist1993perturbation}.

\appendix

\section{Proofs in Section \ref{sec: landscape-analysis-f-well-conditioned}} \label{sec: proof-landscape-h-f-well-condition}

\subsection{Proof of Theorem \ref{th: h-local-geodesic-convexity} }
Suppose the best orthonormal matrix that aligns $\Y$ and $\Y^*$ is $\Q$. Then by definition, $\Y \in \cR_1$ implies 
\begin{equation} \label{ineq: Y-Y*-dist-bound-1}
	\| \Y - \Y^* \Q \|_\F =  d([\Y],[\Y^*]) \leq \mu  \sigma_r(\Y^*)/\kappa^*.  
\end{equation}
Thus 
\begin{equation} \label{ineq: spectrum-bound-of-Y1}
	\begin{split}
		\sigma_r(\Y) &= \sigma_r(\Y - \Y^* \Q + \Y^* \Q ) \geq \sigma_r(\Y^*) - \|\Y - \Y^* \Q\|  \overset{ \eqref{ineq: Y-Y*-dist-bound-1} } \geq (1 - \mu/ \kappa^* ) \sigma_r(\Y^*),\\
		\sigma_1(\Y) &= \sigma_1(\Y - \Y^* \Q + \Y^* \Q  ) \leq \sigma_1(\Y^*) + \|\Y - \Y^* \Q\| \overset{ \eqref{ineq: Y-Y*-dist-bound-1} } \leq \sigma_1(\Y^*) +  \mu \sigma_r(\Y^*)/ \kappa^*.
	\end{split}
\end{equation}

	By Proposition \ref{prop: hH-grad-hessian-connection}, for any $\theta_\Y \in \cH_\Y \widebar{\cM}_{r+}^{q}$, we have
	\begin{equation}
	\begin{split}
		&\left| \overline{\Hess\, H([\Y])}[\theta_\Y, \theta_\Y] - \overline{\Hess\, h([\Y])}[\theta_\Y, \theta_\Y]  \right| \\
		 \leq & \delta  \|\Y\theta_\Y^\top + \theta_\Y \Y^\top \|_\F^2 + 2 \delta \|\Y\Y^\top - \X^*\|_\F \|\theta_\Y \theta_\Y^\top\|_\F + 2 \| ( \nabla f(\X^*) )_{\max(r)} \|_\F \|\theta_\Y \theta_\Y^\top\|_\F \\
		\overset{ (a)  } \leq & 4\delta \|\Y\|^2 \|\theta_\Y\|_\F^2  + 14 \delta  \|\Y^*\| d( [\Y], [\Y^*] )  \|\theta_\Y\|_\F^2/3 + 2 \| ( \nabla f(\X^*) )_{\max(r)} \|_\F \|\theta_\Y\|_\F^2 \\
		\overset{ \eqref{ineq: spectrum-bound-of-Y1} } \leq & 4\delta ( \sigma_1(\Y^*) +  \mu \sigma_r(\Y^*)/ \kappa^*  )^2 \|\theta_\Y\|_\F^2  + 14 \delta  \mu \sigma^2_r(\Y^*)\|\theta_\Y\|_\F^2/3  +  2 \| ( \nabla f(\X^*) )_{\max(r)} \|_\F \|\theta_\Y\|_\F^2 \\
		= & \left(  4\delta ( \sigma_1(\Y^*) +  \mu \sigma_r(\Y^*)/ \kappa^*  )^2 + 14 \delta  \mu \sigma^2_r(\Y^*)/3 + 2 \| ( \nabla f(\X^*) )_{\max(r)} \|_\F  \right) \|\theta_\Y\|_\F^2,
	\end{split}
	\end{equation} here (a) is by Lemma \ref{lm: norm-bound-Ytheta-thetaY} and Lemma \ref{lm: distUY-UUYY-transfer} Eq. \eqref{ineq: distUY-UUYY-transfer-3}.
	Thus, 
	\begin{equation*}
		\begin{split}
			& \overline{\Hess\, h([\Y])}[\theta_\Y, \theta_\Y]\\
			\geq  & \overline{\Hess\, H([\Y])}[\theta_\Y, \theta_\Y] - \left| \overline{\Hess\, H([\Y])}[\theta_\Y, \theta_\Y] - \overline{\Hess\, h([\Y])}[\theta_\Y, \theta_\Y]  \right| \\
			 \overset{ \text{Theorem } \ref{th: H-geodesic-strong-convex} } \geq & \Big(  ( 2 \left( 1 - \mu/ \kappa^* \right)^2 - 14 \mu/3 ) \sigma^2_r(\Y^*) \\
			&\quad  -  \left(  4\delta ( \sigma_1(\Y^*) +  \mu \sigma_r(\Y^*)/ \kappa^*  )^2 + 14 \delta  \mu \sigma^2_r(\Y^*)/3 + 2 \| ( \nabla f(\X^*) )_{\max(r)} \|_\F  \right)   \Big) \|\theta_\Y\|_\F^2,
		\end{split}
	\end{equation*}  and
	\begin{equation*}
		\begin{split}
			& \overline{\Hess\, h([\Y])}[\theta_\Y, \theta_\Y]\\
			\leq &  \overline{\Hess\, H([\Y])}[\theta_\Y, \theta_\Y] + \left| \overline{\Hess\, H([\Y])}[\theta_\Y, \theta_\Y] - \overline{\Hess\, h([\Y])}[\theta_\Y, \theta_\Y]  \right| \\
			 \overset{ \text{Theorem } \ref{th: H-geodesic-strong-convex} } \leq & \Big(4\left(   \sigma_1(\Y^*) + \mu \sigma_r(\Y^*)/\kappa^* \right)^2 + 14\mu \sigma^2_r(\Y^*)/3 \\
			 &  + \left(  4\delta ( \sigma_1(\Y^*) +  \mu \sigma_r(\Y^*)/ \kappa^*  )^2 + 14 \delta  \mu \sigma^2_r(\Y^*)/3 + 2 \| ( \nabla f(\X^*) )_{\max(r)} \|_\F  \right) \Big)  \|\theta_\Y\|_\F^2.
		\end{split}
	\end{equation*}

If  $\mu$ is chosen such that $ \left( 1 - \mu /\kappa^* \right)^2 - 7\mu/3 > 0$ and
 \begin{equation*}
 	\delta \leq \frac{(1-\mu/\kappa^*)^2 - 7\mu/3}{ 4\left(2 ( \kappa^* + \mu/\kappa^* )^2 + 7\mu/3 \right) }\quad  \text{ and }\quad  \| (\nabla f(\X^*))_{\max(r)} \|_\F \leq \left( (1-\mu/\kappa^*)^2 - 7\mu/3 \right) \sigma^2_r(\Y^*)/4,
 \end{equation*}
then 
  \begin{equation*}
 	\begin{split}
 		\lambda_{\min}( \overline{\Hess\, h([\Y])} ) & \geq   \left( (1-\mu/\kappa^*)^2 - 7\mu/3 \right) \sigma^2_r(\Y^*).
 		\end{split}
 \end{equation*} Therefore, $h([\Y])$ is geodesically strongly convex in $\cR_1$ since $\cR_1$ is geodesically convex by Theorem \ref{th: convexity-radius-Mq}. 
 
 Let $\tau :=  \left((1-\mu/\kappa^*)^2 - 7\mu/3 \right)\sigma^2_r(\Y^*) $. Suppose $\widehat{\Y}$ is a Riemannian FOSP in $\cR_1$, then we have $\widehat{\Y}$ is the unique Riemannian FOSP in $\cR_1$. This is because if $\widehat{\Y}'$ is another FOSP in $\cR_1$ and $\Q'$ is the best orthonormal matrix that aligns $\widehat{\Y}$ and $\widehat{\Y}'$. Then by the geodesic strong convexity, we have \cite[Chapter 11]{boumal2020introduction}: 
 \begin{equation*}
 	\begin{split}
 		h([\widehat{\Y}']) &\geq h([\widehat{\Y}]) + \langle \overline{\grad \, h([\widehat{\Y}])} , \widehat{\Y}'\Q' - \widehat{\Y} \rangle +  \frac{\tau}{2} \| \widehat{\Y}' \Q' - \widehat{\Y} \|_\F^2, \\
 		h([\widehat{\Y}]) &\geq h([\widehat{\Y}']) + \langle \overline{\grad \, h([\widehat{\Y}'])} , \widehat{\Y}\Q'^\top - \widehat{\Y}' \rangle +  \frac{\tau}{2} \| \widehat{\Y}' \Q' - \widehat{\Y} \|_\F^2.
 	\end{split}
 \end{equation*}  Notice $\overline{\grad \, h([\widehat{\Y}])} = \overline{\grad \, h([\widehat{\Y}'])} = \0$ by assumption and sum over the above two equations yields, 
 \begin{equation*}
 	\tau \| \widehat{\Y}' \Q' - \widehat{\Y} \|_\F^2 \leq 0.
 \end{equation*} Since $\tau > 0$, we have $\widehat{\Y}' \Q' = \widehat{\Y}$ and $[\widehat{\Y}] = [\widehat{\Y}']$. Moreover, $\widehat{\Y}$ is a local minimizer as for any other $\Y \in \cR_1$:
 \begin{equation*}
 	h([\Y]) \geq h([\widehat{\Y}]) + \langle \overline{\grad \, h([\widehat{\Y}])} , \Y\Q'' - \widehat{\Y} \rangle +  \frac{\tau}{2} \| \Y \Q'' - \widehat{\Y} \|_\F^2 =   h([\widehat{\Y}]) +   \frac{\tau}{2} \| \Y \Q'' - \widehat{\Y} \|_\F^2,
 \end{equation*} where $\Q''$ is the best orthonormal matrix that aligns $\widehat{\Y}$ and $\Y$.
 
 Now, let $\Q$ be the best orthonormal matrix that aligns $\widehat{\Y}$ and $\Y^*$. By a similar argument as above we have 
 \begin{equation*}
 	\begin{split}
 		h([\Y^*]) &\geq h([\widehat{\Y}]) + \langle \overline{\grad \, h([\widehat{\Y}])} , \Y^*\Q - \widehat{\Y} \rangle +  \frac{\tau}{2} \| \Y^* \Q - \widehat{\Y} \|_\F^2, \\
 		h([\widehat{\Y}]) &\geq h([\Y^*]) + \langle \overline{\grad \, h([\Y^*])} , \widehat{\Y}\Q^\top - \Y^* \rangle +  \frac{\tau}{2} \| \Y^* \Q - \widehat{\Y} \|_\F^2.
 	\end{split}
 \end{equation*} Notice $\overline{\grad \, h([\widehat{\Y}])} = \0$ by assumption and sum over the above two equations yields
 \begin{equation} \label{ineq: geodesic-convexity-distance-bound}
 	\begin{split}
 		\tau \| \Y^* \Q - \widehat{\Y} \|_\F^2 &\leq \langle \overline{\grad \, h([\Y^*])} , \Y^* - \widehat{\Y}\Q^\top\rangle \\
 		&= \langle 2\nabla f(\Y^* \Y^{*\top}) \Y^* , \Y^* - \widehat{\Y}\Q^\top\rangle \\
 		& \leq  2\|\nabla f(\Y^* \Y^{*\top}) \Y^*\|_\F  \|\Y^* - \widehat{\Y}\Q^\top\|_\F.
 	\end{split}
 \end{equation} So \eqref{ineq: geodesic-convexity-distance-bound} yields
 \begin{equation*}
 \begin{split}
 	\| \Y^* \Q - \widehat{\Y} \|_\F &\leq \frac{2}{\tau} \|\nabla f(\Y^* \Y^{*\top}) \Y^*\|_\F   \leq \max_{\bDelta: \|\bDelta\|_\F \leq 1 } \frac{2}{\tau} \langle \nabla f(\Y^* \Y^{*\top}) \Y^* , \bDelta \rangle \\
 	& = \frac{2}{\tau} \max_{\bDelta: \|\bDelta\|_\F \leq 1 } \langle \nabla f(\Y^* \Y^{*\top}) , \bDelta \Y^{*\top} \rangle  \\
 	& \overset{\text{Lemma } \ref{lm: charac of Schatten-q norm} } \leq  \max_{\bDelta: \|\bDelta\|_\F \leq 1 } \frac{2}{\tau} \|\bDelta\Y^{*\top}\|_\F \| (\nabla f(\X^*))_{\max(r)} \|_\F  \\
 	 & \leq \frac{2}{\tau} \|\Y^{*}\| \| (\nabla f(\X^*))_{\max(r)} \|_\F.
 \end{split}
 \end{equation*}
 
\subsection{Proof of Theorem \ref{th: h-negative-eigenvalue-region} }
First, we know $\theta_\Y  \in \cH_{\Y}  \widebar{\cM}_{r+}^q $ by Lemma \ref{lm: logarithm-map}. By Theorem \ref{th: H-negative-eigenvalue-region}, we have $\overline{ \Hess\, H([\Y])}[ \theta_\Y, \theta_\Y  ] \leq  (\alpha - 2 ( \sqrt{2} - 1 ))  \sigma^2_r(\Y^*) \|\theta_\Y\|_\F^2$.

In addition, Proposition \ref{prop: hH-grad-hessian-connection} implies
\begin{equation*}
	\begin{split}
		&\left| \overline{\Hess\, H([\Y])}[\theta_\Y, \theta_\Y] - \overline{\Hess\, h([\Y])}[\theta_\Y, \theta_\Y]  \right| \\
		 \leq & \delta  \|\Y\theta_\Y^\top + \theta_\Y \Y^\top \|_\F^2 + 2 \delta \|\Y\Y^\top - \X^*\|_\F \|\theta_\Y \theta_\Y^\top\|_\F + 2 \| ( \nabla f(\X^*) )_{\max(r)} \|_\F \|\theta_\Y \theta_\Y^\top\|_\F \\
		 \leq & 4\delta \|\Y \theta_\Y^{ \top \nc}\|_\F^2  + 2 \delta  \|\Y\Y^\top - \X^*\|_\F \|\theta_\Y\|_\F^2 + 2 \| ( \nabla f(\X^*) )_{\max(r)} \|_\F \|\theta_\Y\|_\F^2 \\
		 \overset{(a)} \leq & 2 \delta \left( 2 \beta^2 \|\Y^*\|^2 + (1+\gamma) \| \Y^* \Y^{*\top} \|_\F  \right) \|\theta_\Y\|_\F^2 +  2 \| ( \nabla f(\X^*) )_{\max(r)} \|_\F \|\theta_\Y\|_\F^2,
	\end{split}
\end{equation*} where (a) is because $\Y \in \cR_2$.
	Thus,
\begin{equation*}
	\begin{split}
		 \overline{\Hess\, h([\Y])}[\theta_\Y, \theta_\Y] &\leq  \overline{\Hess\, H([\Y])}[\theta_\Y, \theta_\Y] + \left| \overline{\Hess\, H([\Y])}[\theta_\Y, \theta_\Y] - \overline{\Hess\, h([\Y])}[\theta_\Y, \theta_\Y]  \right| \\
		 & \leq  (\alpha - 2 ( \sqrt{2} - 1 ))  \sigma^2_r(\Y^*) \|\theta_\Y\|_\F^2 \\
		 & \quad +  2 \delta \left( 2 \beta^2 \|\Y^*\|^2 + (1+\gamma) \| \Y^* \Y^{*\top} \|_\F  \right) \|\theta_\Y\|_\F^2 +  2 \| ( \nabla f(\X^*) )_{\max(r)} \|_\F \|\theta_\Y\|_\F^2 .
	\end{split}
\end{equation*}

 So if 
\begin{equation*}
	\delta \leq \frac{( 2 (\sqrt{2} -1) -\alpha ) \sigma^2_r(\Y^*)  }{8 \left( 2 \beta^2 \|\Y^*\|^2 + (1 + \gamma ) \| \Y^* \Y^{*\top} \|_\F   \right) } \text{ and } \| ( \nabla f(\X^*) )_{\max(r)} \|_\F \leq \frac{ 2 ( \sqrt{2} - 1 ) - \alpha }{8} \sigma^2_r(\Y^*),
\end{equation*}
then
\begin{equation*}
	\left| \overline{\Hess\, H([\Y])}[\theta_\Y, \theta_\Y] - \overline{\Hess\, h([\Y])}[\theta_\Y, \theta_\Y]  \right| \leq \frac{ 2 (\sqrt{2} -1) -\alpha  }{2} \sigma^2_r(\Y^*) \|\theta_\Y\|_\F^2.
\end{equation*}
	
	Thus,
\begin{equation*}
	\begin{split}
		 \overline{\Hess\, h([\Y])}[\theta_\Y, \theta_\Y] &\leq  \overline{\Hess\, H([\Y])}[\theta_\Y, \theta_\Y] + \left| \overline{\Hess\, H([\Y])}[\theta_\Y, \theta_\Y] - \overline{\Hess\, h([\Y])}[\theta_\Y, \theta_\Y]  \right| \\
		 & \leq  \frac{\alpha - 2 ( \sqrt{2} - 1 )}{2}  \sigma^2_r(\Y^*) \|\theta_\Y\|_\F^2.
	\end{split}
\end{equation*}

\subsection{Proof of Theorem \ref{th: h-large-gradient-norm}}
	We prove the results for the three regions separately.
	
{\bf When $\Y \in \cR_3'$.} By Proposition \ref{prop: hH-grad-hessian-connection},
\begin{equation*}
	\begin{split}
		\| \overline{\grad\, H([\Y])} - \overline{\grad\, h([\Y])}  \|_\F & \leq  2 \delta \|\Y\| \|\Y\Y^\top - \X^*\|_\F + 2 \|\Y\| \| ( \nabla f(\X^*) )_{\max(r)} \|_\F \\
		& \leq 2 \delta \beta ( 1 + \gamma) \|\Y^*\| \| \Y^* \Y^{*\top} \|_\F + 2 \beta \|\Y^*\| \| ( \nabla f(\X^*) )_{\max(r)} \|_\F.  
	\end{split}
\end{equation*} Thus, combining the above result with Theorem \ref{th: H-large-gradient-norm}, we have
\begin{equation*}
\begin{split}
	\|\overline{\grad\, h([\Y])} \|_\F \geq & \| \overline{\grad\, H([\Y])} \|_\F - \| \overline{\grad\, H([\Y])} - \overline{\grad\, h([\Y])}  \|_\F \\
	> &  \alpha \mu \sigma^3_r(\Y^*)/(4\kappa^*) - \left( 2 \delta \beta ( 1 + \gamma) \|\Y^*\| \| \Y^* \Y^{*\top} \|_\F + 2 \beta \|\Y^*\| \| ( \nabla f(\X^*) )_{\max(r)} \|_\F \right).  
\end{split}
\end{equation*}
In particular, if $\delta \leq \frac{\alpha \mu }{32 \kappa^{*2} \beta (1 + \gamma) } \frac{\sigma_r^2(\Y^*)}{ \| \Y^* \Y^{*\top} \|_\F }  $ and $\| ( \nabla f(\X^*) )_{\max(r)} \|_\F \leq \frac{\alpha \mu}{32 \kappa^{*2} \beta } \sigma^2_r(\Y^*) $, we have
\begin{equation} \label{ineq: h-gradient-bound-1}
	\|\overline{\grad\, h([\Y])} \|_\F > \alpha\mu \sigma^3_r(\Y^*)/ (8 \kappa^* ).
\end{equation}

{\bf When $\Y \in \cR_3''$.}  By Proposition \ref{prop: hH-grad-hessian-connection},
\begin{equation*}
	\begin{split}
		\| \overline{\grad\, H([\Y])} - \overline{\grad\, h([\Y])}  \|_\F & \leq  2 \delta \|\Y\| \|\Y\Y^\top - \X^*\|_\F + 2 \|\Y\| \| ( \nabla f(\X^*) )_{\max(r)} \|_\F \\
		& \leq 2 \delta ( 1 + \gamma) \|\Y\| \| \Y^* \Y^{*\top} \|_\F + 2 \|\Y\| \| ( \nabla f(\X^*) )_{\max(r)} \|_\F. 
	\end{split}
\end{equation*}
Thus 
	\begin{equation*}
	\begin{split}
		\|\overline{\grad\, h([\Y])} \|_\F \geq & \| \overline{\grad\, H([\Y])} \|_\F - \| \overline{\grad\, H([\Y])} - \overline{\grad\, h([\Y])}  \|_\F \\
		 \overset{\text{Theorem } \ref{th: H-large-gradient-norm} } \geq &  2(\|\Y\|^3 - \|\Y\| \|\Y^*\|^2) - \left( 2 \delta ( 1 + \gamma) \|\Y\| \| \Y^* \Y^{*\top} \|_\F + 2 \|\Y\| \| ( \nabla f(\X^*) )_{\max(r)} \|_\F \right);
	\end{split}
	\end{equation*}
In particular, if $\delta \leq \frac{  \beta^2 - 1 }{ 4(1 + \gamma) } \frac{ \|\Y^*\|^2 }{ \| \Y^* \Y^{*\top} \|_\F } < \frac{1}{4(1+ \gamma)} \frac{\|\Y\|^2 - \|\Y^*\|^2}{ \| \Y^* \Y^{*\top} \|_\F}  $ and $\| ( \nabla f(\X^*) )_{\max(r)} \|_\F \leq \frac{ \beta^2 -1 }{4 } \|\Y^*\|^2 < \frac{\|\Y\|^2 - \|\Y^*\|^2}{4} $, we have
\begin{equation} \label{ineq: h-gradient-bound-2}
	\|\overline{\grad\, h([\Y])} \|_\F \geq \|\Y\|^3 - \|\Y\| \|\Y^*\|^2 > (\beta^3 - \beta) \|\Y^*\|^3.
\end{equation}

{\bf When $\Y \in \cR_3'''$.} We have
\begin{equation*}
	\begin{split}
		& \left| \langle \overline{\grad\, H([\Y])} -\overline{\grad\, h([\Y])} , \Y \rangle \right| \\
		\overset{\text{Lemma } \ref{lm: gradient-hessian-exp-PSD}, \eqref{eq: gradient-Hessian-exp-H}  } \leq & 2\left|  \langle \nabla f(\Y \Y^\top) - \nabla f(\X^*) -  (\Y \Y^\top -\X^*), \Y \Y^\top \rangle   \right| + 2\left|  \langle \nabla f(\X^*), \Y \Y^\top \rangle   \right| \\
		\overset{ \text{Lemmas } \ref{lm: RIP-imply-gradient-bound}, \ref{lm: charac of Schatten-q norm} } \leq & 2 \delta \|\Y \Y^\top\|_\F \| \Y \Y^\top -\Y^* \Y^{*\top} \|_\F + 2 \|\Y \Y^\top\|_\F \| ( \nabla f(\X^*) )_{\max(r)} \|_\F \\
		\overset{(a)} < & 2\delta(1+ 1/\gamma ) \|\Y \Y^\top\|^2_\F  + 2 \|\Y \Y^\top\|_\F \| ( \nabla f(\X^*) )_{\max(r)} \|_\F,
	\end{split}
\end{equation*} here (a) is because $\Y \in \cR_3'''$.

Thus 
	\begin{equation} \label{ineq: gradient-inner-product-bound}
	\begin{split}
		\langle \overline{\grad\, h([\Y])} , \Y \rangle & \geq \langle \overline{\grad\, H([\Y])}  , \Y \rangle - \left| \langle \overline{\grad\, H([\Y])} -\overline{\grad\, h([\Y])} , \Y \rangle \right|\\
		& \overset{\text{Theorem } \ref{th: H-large-gradient-norm} } > 2(1 -1/\gamma) \|\Y\Y^\top\|_\F^2 \\
		& \quad \quad \quad  - \left( 2\delta(1+ 1/\gamma ) \|\Y \Y^\top\|^2_\F  + 2 \|\Y \Y^\top\|_\F \| ( \nabla f(\X^*) )_{\max(r)} \|_\F \right).
	\end{split}
	\end{equation}

Moreover, since 
	\begin{equation*}
	\begin{split}
		\langle \overline{\grad\, h([\Y])} , \Y \rangle \leq \|\overline{\grad\, h([\Y])}\|_\F  \|\Y\|_\F & \leq \|\overline{\grad\, h([\Y])}\|_\F \sqrt{r} \|\Y\|  \\
		& \overset{(a)} \leq  \sqrt{r} \|\overline{\grad\, h([\Y])}\|_\F \| \Y \Y^\top \|_\F^{1/2},
	\end{split}
	\end{equation*} where (a) is because $\|\Y\| \leq \| \Y \Y^\top \|_\F^{1/2}$, we have
	\begin{equation*}
	\begin{split}
		& \|\overline{\grad\, h([\Y])}\|_\F \geq \langle \overline{\grad\, h([\Y])} , \Y \rangle    \| \Y \Y^\top \|_\F^{-1/2} /\sqrt{r} \\
		\overset{ \eqref{ineq: gradient-inner-product-bound} } > & \left( 2( 1- 1/\gamma)  - 2\delta (1+1/\gamma) \right)  \|\Y\Y^\top\|^{3/2}_\F/\sqrt{r}  - 2  \|\Y\Y^\top\|^{1/2}_\F \| ( \nabla f(\X^*) )_{\max(r)} \|_\F/\sqrt{r} \\
		\overset{(a)}> & \left( 2( \gamma- 1) - 2\delta(\gamma + 1) \right) \gamma^{1/2} \|\Y^*\Y^{*\top}\|^{3/2}_\F/\sqrt{r} - 2 \gamma^{1/2}\|\Y^*\Y^{*\top}\|^{1/2}_\F\| ( \nabla f(\X^*) )_{\max(r)} \|_\F/\sqrt{r} . 
	\end{split}
	\end{equation*}
	Here (a) is because $\Y \in \cR_3'''$.	
	
In particular, if $\delta \leq \frac{  \gamma - 1 }{ 4( \gamma + 1 ) }  $ and $\| ( \nabla f(\X^*) )_{\max(r)} \|_\F \leq \frac{1}{4} (\gamma-1) \|\Y^* \Y^{*\top} \|_\F $, we have
\begin{equation} \label{ineq: h-gradient-bound-3}
	\|\overline{\grad\, h([\Y])} \|_\F >  ( \gamma- 1) \gamma^{1/2} \|\Y^*\Y^{*\top}\|^{3/2}_\F/\sqrt{r}.
\end{equation}

Finally, \eqref{ineq: h-gradient-bound-all} is a combination of the results in \eqref{ineq: h-gradient-bound-1}, \eqref{ineq: h-gradient-bound-2} and \eqref{ineq: h-gradient-bound-3}.

{
\subsection{Proof of Theorem \ref{th:application} } 
\label{app:DegeneracyIterates}

\subsubsection{Perturbed Riemannian Gradient Descent} \label{sec:PRGD}
There are two versions of perturbed Riemannian gradient algorithms developed in the literature \cite{criscitiello2019efficiently,sun2019escaping}. \cite{criscitiello2019efficiently} performs the perturbed GD in the tangent space, while \cite{sun2019escaping} executes all steps on the manifold and requires the retraction to be the Riemannian exponential map. Therefore, the regularity assumptions required in these two works are also slightly different. Here we focus on Algorithm 1 of \cite{sun2019escaping}, while we expect we can also use the PRGD developed in \cite{criscitiello2019efficiently}. The Algorithm \ref{alg:PRGD-sun} below is an adaptation of Algorithm 1 of \cite{sun2019escaping} for the optimization problem \eqref{eq: quotient-matrix-fac-denoising} in the lifted horizontal space. This is because the horizontal geodesic in our setting is simply a straight line; see \cite{massart2020quotient} Theorem 2.2. 
\begin{algorithm}
\caption{Perturbed Riemannian gradient algorithm $(\Y_0,L, \rho, K, \mathfrak{I}, \epsilon, \zeta)$}
{\noindent \bf Input} Initial point $\Y_0 \in \widebar{\cM}_{r+}^{q}$, parameters $L, \rho, K, \mathfrak{I}$, accuracy $\epsilon$, probability of success $\zeta$ (parameters defined in Assumptions 1, 2, 3 and assumption of Theorem 1 of \cite{sun2019escaping}).
\begin{algorithmic}[1]
\State Set constants: $\hat{c} \ge 4$, $C := C(K, L, \rho)$ (defined in Lemma 2 and proof of Lemma 8 of \cite{sun2019escaping})
\Statex \hspace{1em} and $\sqrt{c_{\max}} \leq \frac{1}{56\hat{c}^2}$, $\omega = \frac{\sqrt{c_{\max}}}{\chi^2}\epsilon$, $\chi = 3\max\left\{\log\left(\frac{\bar{d}L H([\Y_0])}{\hat{c}\epsilon^2\zeta}\right), 4\right\}$, where $\bar{d} = (pr - (r^2 -r)/2)$ is the dimension of $\cM^q_{r+}$.
\State Set threshold values: $f_{\text{thres}} = \frac{c_{\max}}{\chi^3} \sqrt{\frac{\epsilon^3}{\rho}}$, $g_{\text{thres}} = \frac{\sqrt{c_{\max}}}{\chi^2} \epsilon$, $t_{\text{thres}} = \frac{\chi}{c_{\max}} \frac{L}{\sqrt{\rho\epsilon}}$, $t_{\text{noise}} = -t_{\text{thres}} - 1$.
\State Set $t = 0$ and stepsize: $\eta = \frac{c_{\max}}{L}$.
\While{1}
    \If{$\|\overline{\grad\, H([\Y_t])}\| \leq g_{\text{thres}}$ \textbf{and} $t - t_{\text{noise}} > t_{\text{thres}}$}
        \State $t_{\text{noise}} \leftarrow t$, $\widetilde{\Y}_t \leftarrow \Y_t$, $\Y_t \leftarrow \Y_t + \xi_t$, $\xi_t$ uniformly sampled from $\{ \theta: \theta \in \cH_{\Y_t}\widebar{\cM}_{r+}^{q}, \|\theta\|_{\F} \leq \omega \}$.
    \EndIf
    \If{$t - t_{\text{noise}} = t_{\text{thres}}$ \textbf{and} $H([\Y_t]) - H([\widetilde{\Y}_{t_{\text{noise}}}]) > -f_{\text{thres}}$}
        \State \textbf{output} $\widetilde{\Y}_{t_{\text{noise}}}$
    \EndIf
    \State $\Y_{t+1} \leftarrow \Y_t -\left( \min\left\{\eta, \frac{\mathfrak{I}}{\|\overline{\grad\, H([\Y_t])}\|_{\F}} \right\} \overline{\grad\, H([\Y_t])} \right)$.
    \State $t \leftarrow t + 1$.
\EndWhile
\end{algorithmic}\label{alg:PRGD-sun}
\end{algorithm}
The basic idea of perturbed Riemannnian GD is that when the magnitude of the Riemannian gradient is large, then we run vanilla Riemannnian GD, but when the magnitude of the Riemannian gradient is small, then we perform a random perturbation on the current iterate, and, after that, we run $t_{\text{thres}}$ steps of vanilla Riemannnian GD. The whole algorithm terminates when the function value does not have a sufficiently large decrease. 
\subsubsection{Preliminary Analysis} \label{proof:app-preliminary}
In this section, we present some preparatory results needed in the proof of Theorem \ref{th:application}. First, given two positive real numbers $0 < c_1 \leq c_2$, let us define 
\begin{equation} \label{def:Rc1c2}
	\begin{split}
		\cR_{c_1, c_2} = \{ \Y \in \bbR^{p \times r}_*: c_1 \leq \sigma_r(\Y) \leq \sigma_1(\Y) \leq c_2 \}.
	\end{split}
\end{equation} The next lemma provides sectional curvature and Lipschitz constant bounds for the Riemannian gradient and Riemannian Hessian at points in $\cR_{c_1, c_2}$. These quantities will be used as inputs in the PRGD algorithm. The proof will be provided in Appendix \ref{proof:sec-app-lemma}.
\begin{Lemma} \label{lm:curvature-lipschitz-constants}
	For the objective $H([\Y])$ in \eqref{eq: quotient-matrix-fac-denoising}, we have
	\begin{itemize}
		\item (i) $\left\|  \overline{\grad\, H([\Y])} - \Gamma_{[\Y']}^{[\Y]} \overline{\grad\, H([\Y'])} \right\|_{\F} \leq (6c_2^2 + 2 \sigma_1^2(\Y^*)) d([\Y], [\Y'] )$ for all $Y, Y' \in \cR_{c_1, c_2}$, where $\Gamma_{[\Y']}^{[\Y]}$ denotes a parallel transport in the horizontal space that transports $v \in  \cH_{\Y'}\widebar{\cM}_{r+}^{q} $ to $ \Gamma_{[\Y']}^{[\Y]} v \in  \cH_\Y\widebar{\cM}_{r+}^{q} $; see its detailed definition in Section 5.4 of \cite{absil2009optimization}.
		\item (ii) $\left \| \overline{\Hess \, H([\Y])} -  \Gamma_{[\Y]}^{[\Y']} \overline{\Hess \, H([\Y'])} \Gamma_{[\Y']}^{[\Y]} \right \| \leq 4 \left( 3 c_2 (\kappa^2 +1) + \sigma^2_1(\Y^*) \kappa/c_1  \right) d([\Y], [\Y'] )$ for all $Y, Y' \in \cR_{c_1, c_2}$. Here $\kappa :=c_2/c_1$ and $\|\cdot\|$ denotes the operator norm with respect to the Riemannian metric, i.e.,  $\| \overline{\Hess \, H([\Y])}\| = \sup_{\theta_{\Y} \in   \cH_\Y\widebar{\cM}_{r+}^{q} , \|\theta_{\Y}\|_{\F} \leq 1}  \overline{\Hess \, H([\Y])} [\theta_{\Y}, \theta_{\Y}] $.
		\item (iii) $\left| K([\Y])[\theta_{[\Y]}, \theta'_{[\Y]}] \right| \leq \frac{3}{2 c_1^2}$ for all $Y \in \cR_{c_1, c_2}$, $\theta_{[\Y]}, \theta'_{[\Y]} \in  T_{\Y} {\cM}_{r+}^{q}$. Here $K([\Y])$ denotes the sectional curvature at $[\Y]$; see its detailed definition in Section 7 of \cite{lee2018introduction}.
	\end{itemize}
\end{Lemma}

Moreover, we argue that, without loss of generality, we can assume $\X^*$ is a diagonal matrix with only the top $r$ diagonal values being positive. Recall $\X^*$ has full SVD of the form $[\U^* \, \U^*_{\perp}] \Sigma [\U^* \, \U^*_{\perp}]^{\top}$ for some diagonal matrix $\Sigma$ and 
\begin{equation*}
	\|\Y \Y^\top - \X^*\|_\F^2 = \| ([\U^* \, \U^*_{\perp}]^{\top } \Y) ([\U^* \, \U^*_{\perp}]^{\top } \Y)^\top - \Sigma \|_\F^2.
\end{equation*} 
So, if we replace all $\widetilde{\Y}_i$s and $\Y_i$s in the algorithm by $\widetilde{\W}_i = [\U^* \, \U^*_{\perp}]^{\top } \widetilde{\Y}_i$ and $\W_i = [\U^* \, \U^*_{\perp}]^{\top }{\Y}_i$, then we can see that $\widetilde{\W}_i$ and $\W_i$ are obtained by running the same algorithm with $\X^*$ being replaced by $ \Sigma$. In addition, the spectra of $\Y_i$s and $\W_i$s are exactly the same. So, without loss of generality, we will simply assume $\X^*$ is a diagonal PSD matrix. With this in mind, the initialization condition reduces to 
\begin{equation} \label{ineq:new-initilization-condition}
	 \sigma_1(\widetilde{\Y}_0) \leq \sigma_1(\Y^*),\, \sigma_1(\J_0) \leq \sigma_r(\Y^*)/2,\, 0 < \sigma_r(\K_0) \leq \sigma_r(\Y^*)/2, \, \sigma_1^2(\J_0) \leq \frac{\sigma^2_r(\Y^*)}{8 \sigma_1(\Y^*)} \sigma_r(\K_0),
\end{equation} where $\widetilde{\Y}_0 = [\K_0^\top, \J_0^\top ]^\top $.

For the analysis of the first and third stages of the algorithm, we choose $\eta = \frac{1}{40 \kappa^{*^4} \sigma^2_1(\Y^*) }$, where $\kappa^* = \sigma_1(\Y^*)/\sigma_r(\Y^*)$. The choices for $T_0$ and $T_1$ will be given in the corresponding analysis.

The rest of this section is divided as follows. In Appendix \ref{sec:proof-init}-\ref{sec:proof-local}, we provide the analysis for the three stages of the algorithm separately. Proofs of technical lemmas are provided in the subsequent subsections.

\subsubsection{Step 1 for the proof of Theorem \ref{th:application}: Initialization Safeguard } \label{sec:proof-init}
The proof for this part is similar to the analysis of Theorem B.2 in \cite{chen2023fast}. Here we provide a clean version in our context for the reader's convenience. First, since $\sigma_1(\widetilde{\Y}_0) \leq \sigma_1(\Y^*)$, by Lemma \ref{lm:iterate-control}(i), we have
\begin{equation*}
	\begin{split}
		\sigma_1(\Y_{t+1}) \leq (1 + \eta \sigma^2_1(\Y^*) - \eta \sigma^2_1(\Y_t) ) \sigma_1(\Y_t) \leq \sigma_1(\Y^*),
	\end{split}
\end{equation*} where the second inequality is because $g(s) = (1 + \eta \sigma^2_1(\Y^*) - \eta s^2)s$ is increasing when $s \in [0, 1/\sqrt{3\eta}]$ and $\sigma_1(\Y_t) \leq \sigma_1(\Y^*) \leq  1/\sqrt{3\eta}$.

 Similarly, by Lemma \ref{lm:iterate-control}(ii), it is also straightforward to prove by induction that $\sigma_1(\J_{t+1}) \leq \sigma_1(\J_{t}) \leq \sigma_1(\J_0)$ for $t \geq 0.$ Let $T = \min\{t\geq 0| \sigma_r(\K_t) \geq \sigma_r(\Y^*)/2 \}$. We prove by induction that, for any $t < T$, we have $\sigma_r(\K_{t+1}) \geq  (1 + \eta \sigma^2_r(\Y^*)/2 ) \sigma_r(\K_{t})$. When $t = 0$, we know 
\begin{equation} \label{ineq:induction-start}
	\sigma_1^2(\J_t) \leq \frac{\sigma^2_r(\Y^*)}{8 \sigma_1(\Y^*)} \sigma_r(\K_t)
\end{equation} from the initialization condition. Recall that we have argued that we can assume $\X^*$ is a diagonal PSD matrix, and only the top $r$ diagonal values are positive. Let ${\bf \Lambda}_r = \X^*_{[1:r,1:r]}$. The Riemannnian GD update rule gives 
\begin{equation*}
	\K_{t+1} = \K_t + \eta {\bf \Lambda}_t \K_t - \eta \K_t (\Y_t^\top \Y_t) = {\bf K}_{t} + \eta({\bf \Lambda}_r - {\bf K}_t{\bf K}_t^\top){\bf K}_t - \eta {\bf K}_t {\bf J}_t^\top{\bf J}_t.
\end{equation*} Thus, 
\begin{equation} \label{ineq:induction}
\begin{split}
	\sigma_r({\bf K}_{t+1}) & \geq \sigma_r \left({\bf K}_{t} + \eta({\bf \Lambda}_r - {\bf K}_t{\bf K}_t^\top){\bf K}_t\right) - \eta\sigma_1\left({\bf K}_t{\bf J}_t^\top{\bf J}_t\right) \\
& \overset{\text{Lemma }\ref{lm:K-lower-bound} } \geq \left(1 - 2\eta^2\sigma_1\left({\bf \Lambda}_r{\bf K}_t{\bf K}_t^\top\right)\right)\left(1 + \eta\sigma_r^2(\bf{Y}^*)\right)\sigma_r({\bf K}_t)\left(1 - \eta\sigma_r^2({\bf K}_t)\right) - \eta\sigma_1\left({\bf K}_t{\bf J}_t^\top{\bf J}_t\right) \\
& \overset{(a)} \geq \left(1 - 2\eta^2 (\sigma_1^2(\bf{Y}^*))^2\right)\left(1 + \eta\sigma_r^2(\bf{Y}^*)\right)\left(1 - \eta\sigma_r^2(\bf{Y}^*)/4\right)\sigma_r({\bf K}_t) - \frac{\eta\sigma_r^2(\bf{Y}^*)}{8}\sigma_r({\bf K}_t) \\
& \overset{(b)}  \geq \left(1 + \eta\sigma_r^2(\bf{Y}^*)/2 \right) \cdot \sigma_r({\bf K}_t),
\end{split}	
\end{equation}
 where (a) is because $\sigma_r^2(\K_t) \leq \sigma_r^2(\Y^*)/4$, $\sigma_1(\K_t) \leq \sigma_1(\Y^*)$ and \eqref{ineq:induction-start}, and (b) is because $\eta = \frac{1}{40 \kappa^{*^4} \sigma^2_1(\Y^*) }$. Thus, the induction holds when $t = 0$. Suppose the claim holds for general $t$. Then, at $t+1$, we would have \eqref{ineq:induction-start} holds with $t$ being replaced by $t+1$ since $\sigma_1(\J_{t+1}) \leq \sigma_1(\J_{t}) \leq \sigma_1(\J_0)$ for $t \geq 0$, and $\sigma_r(\K_t)$ only increases. So, if $\sigma_r({\bf K}_{t+1}) < \sigma_r(\Y^*)/2$, we can continue the induction. Following the same analysis as in \eqref{ineq:induction}, we can get the result. Therefore, we know $T \leq \left \lceil \frac{\log\left( \sigma_r(\Y^*)/(2\sigma_r(\K_0) ) \right)}{\log \left( 1 + \eta\sigma_r^2(\bf{Y}^*)/2  \right)} \right\rceil$. Hence, after at most $T$ steps, the iterate enters the following nice region:
 \begin{equation} \label{def:R-nice-0}
	\begin{split}
		&\cR_{\text{nice}} = \Big \{ \Y = [\K^\top \, \J^{\top}]^{\top}: \K \in \bbR^{r \times r}, \J \in \bbR^{(p-r) \times r}, \sigma_1(\Y) \leq \frac{11}{10} \sigma_1(\Y^*),\\
		& \quad \quad  \sigma_1(\J) \leq \frac{1}{2} \sigma_r(\Y^*), \sigma_r(\K) \geq \frac{1}{2} \sigma_r(\Y^*) \Big \}.
	\end{split}
\end{equation}
 The following key lemma (see the proof in Appendix \ref{sec:AdditionalLemmas}) shows that if $\Y_t \in \cR_{\text{nice}}$, then $\Y_{t+1} \in \cR_{\text{nice}}$.
 \begin{Lemma}\label{lm:obsorbing0}
	Given any $\Y_t \in \cR_{\text{nice}}$. Let ${\Y}_{t+1} = \Y_{t} - 2\eta (\Y_{t} \Y_{t}^{\top} - \X^*)\Y_{t}$, where $\X^*$ is a rank $r$ diagonal PSD matrix and $\Y^* \Y^{*\top } = \X^*$. Then ${\Y}_{t+1} \in \cR_{\text{nice}}$ as long as the stepsize satisfies $\eta \leq  \frac{1}{40 \kappa^{*^4} \sigma^2_1(\Y^*) } $, where $\kappa^* = \sigma_1(\Y^*)/\sigma_r(\Y^*)$.
\end{Lemma}
Therefore, by inductively applying Lemma \ref{lm:obsorbing0}, we conclude that after $T_0 = \left \lceil \frac{\log\left( \sigma_r(\Y^*)/(2\sigma_r(\K_0) ) \right)}{\log \left( 1 + \eta\sigma_r^2(\bf{Y}^*)/2  \right)} \right\rceil$ number of iterations, $\widetilde{\Y}_{T_0} \in \cR_{\text{nice}}$.

\subsubsection{Step 2 for the proof of Theorem \ref{th:application}: Enter a local region via PRGD} \label{sec:proof-middle}
Let us recall the regions $\cR_1, \cR_2$ and $\cR_3$ we have defined in \eqref{def: regions}. We take $\alpha = \sqrt{2} - 1$, $\beta = 2$, $\gamma = 2$ and $\mu = 1/5$. From Theorems \ref{th: H-geodesic-strong-convex}-\ref{th: H-large-gradient-norm}, we then know that
\begin{equation} \label{eq:H-landscape-simplified}
	\begin{split}
		&\|\overline{\grad\, H([\Y])} \|_\F \geq \frac{\sqrt{2} - 1}{20 \kappa^*} \sigma_r^3(\Y^*), \forall \Y \in \cR_3, \\
		&\lambda_{\min}(\overline{ \Hess\, H([\Y])}) \leq -(\sqrt{2} -1) \sigma_r^2(\Y^*), \forall \Y \in \cR_2, \\
		& \frac{26}{75}  \sigma_r^2(\Y^*) \leq \lambda_{\min}(\overline{\Hess\, H([\Y])} ) \leq  \lambda_{\max}(\overline{\Hess\, H([\Y])} ) \leq \frac{502}{75} \sigma^2_1(\Y^*), \forall \Y \in \cR_1.
	\end{split}
\end{equation}

Let us also define a slightly enlarged nice region in which the iterates of PRGD will lie:
\begin{equation} \label{def:R-nice}
	\begin{split}
		&\widetilde{\cR}_{\text{nice}} = \Big \{ \Y = [\K^\top \, \J^{\top}]^{\top}: \K \in \bbR^{r \times r}, \J \in \bbR^{(p-r) \times r}, \sigma_1(\Y) \leq \frac{6}{5} \sigma_1(\Y^*),\\
		& \quad \quad  \sigma_1(\J) \leq \frac{2}{3} \sigma_r(\Y^*), \sigma_r(\K) \geq \frac{1}{3} \sigma_r(\Y^*) \Big \}.
	\end{split}
\end{equation} From Algorithm \ref{alg:PRGD-sun}, we know there are two types of gradient steps we take: (Type 1) first perturb the current iterate and then run Riemannian GD; (Type 2) run vanilla Riemannian GD without adding any perturbations. Both types of iterates can be written in a unified way as
\begin{equation*}
	\begin{split}
		\text{first, update }\quad \Y_{t+0.5} = \Y_t + \xi_t, \quad \text{then} \quad \Y_{t+1} = \Y_{t+0.5}- 2\bar{\eta}_t (\Y_{t+0.5}\Y^{\top }_{t+0.5} - \X^*) \Y_{t+0.5},
	\end{split}
\end{equation*} where $\bar{\eta}_t =  \min\left\{\eta, \frac{\mathfrak{I}}{\|\overline{\grad\, H([\Y_t])}\|_{\F}} \right\} $. The difference is that in Type 1, we sample $\xi_t$ uniformly at random from $\{ \theta: \theta \in \cH_{\Y_{t}}\widebar{\cM}_{r+}^{q}, \|\theta\|_{\F} \leq \omega \}$; while in Type 2, we simply take $\xi_t = 0$. The following key lemma is an extension of Lemma \ref{lm:obsorbing0}, which shows, even for the perturbed gradient descent, that if $\Y_t \in \cR_{\text{nice}}$, then $\Y_{t+1} \in \cR_{\text{nice}}$ as long as the scale of perturbation is small relative to the stepsize.

\begin{Lemma}\label{lm:obsorbing}
	Given any $\Y_t \in \cR_{\text{nice}}$. Let $\Y_{t+0.5} = \Y_t + \xi$ for any $\xi$ such that $\|\xi\|_\F \leq a \sigma_r(\Y^*)$ and ${\Y}_{t+1} = \Y_{t+0.5} - 2\bar{\eta}_t (\Y_{t+0.5} \Y_{t+0.5}^{\top} - \X^*)\Y_{t+0.5}$ where $\bar{\eta}_t =  \min\left\{\eta, \frac{ \sigma_r(\Y^*)/3 }{\|\overline{\grad\, H([\Y_t])}\|_{\F}} \right\} $ and $\X^*$ is a rank $r$ diagonal PSD matrix and $\Y^* \Y^{*\top } = \X^*$. Then $\Y_{t+0.5} \in \widetilde{\cR}_{\text{nice}}$ and ${\Y}_{t+1} \in \cR_{\text{nice}}$ as long as $\eta \leq  \frac{1}{40 \kappa^{*^4} \sigma^2_1(\Y^*) } $ and $a \leq \frac{1}{50} \min \{ \eta, \frac{1}{15 \kappa^* \sigma^2_1(\Y^*) \sqrt{r} }\} \sigma_r^2(\Y^*)$ where $\kappa^* = \sigma_1(\Y^*)/\sigma_r(\Y^*)$.
\end{Lemma}

Via the initialization guarantee established in Step 1, we know that after $T_0$ steps, $\widetilde{\Y}_{T_0} \in \cR_{\text{nice}}$. Now, let us specify the rest of the parameters for the PRGD algorithm:
\begin{equation} \label{eq:PRGD-parameter1}
	\begin{split}
		& L = \frac{266}{25} \kappa^{*4} \sigma^2_1(\Y^*), \quad \rho = 250 \kappa^{*2} \sigma_1(\Y^*), \quad K = \frac{27}{2  \sigma_r^2(\Y^*)}, \quad \mathfrak{I} = \frac{1}{3} \sigma_r(\Y^*).
	\end{split}
\end{equation} We also let $C(K, L, \rho)$ be the constant defined in Lemma 2 and proof of Lemma 8 of \cite{sun2019escaping} and let $\hat{\rho} = \max\{ \rho, C(K, L, \rho)\}$. Finally, the accuracy parameter of the PRGD is chosen as $\tilde{\epsilon}$ such that
\begin{equation}\label{eq:PRGD-parameter2}
	\begin{split}
		& \tilde{\epsilon} < \min \left\{A, B, C, D, E  \right\},\\
		&\text{ where } A= \frac{\hat{\rho}}{56 \max\{c_2(K), c_3(K)\}\eta L} \log \left( \frac{\bar{d}L}{\sqrt{\hat{\rho}\tilde{\epsilon}}\zeta} \right), B = \left( \frac{\mathfrak{I}\hat{\rho}}{12\hat{c}\sqrt{\eta L}} \log \left( \frac{\bar{d}L}{\sqrt{\hat{\rho}\tilde{\epsilon}}\zeta} \right) \right)^2,\\
		& C = \frac{\sqrt{2} - 1}{20 \kappa^*} \sigma^3_r(\Y^*) , D = \frac{(\sqrt{2}-1)^2 \sigma_r^4(\Y^*)}{\hat{\rho}} , E = \frac{1}{50} \frac{\chi^2}{\sqrt{c_{\max}}} \left( \frac{c_{\max}}{L} \wedge \frac{1}{15 \kappa^* \sigma^2_1(\Y^*) \sqrt{r} }\right) \sigma_r(\Y^*).
	\end{split}
\end{equation} Here $c_{\max},\chi$ are quantities specified in Algorithm \ref{alg:PRGD-sun}, $\eta$ is the choice of the stepsize in Algorithm \ref{alg:PRGD-sun}, $\bar{d} = pr - (r^2-r)/2$, $c_2(K)$ and $c_3(K)$ are two constants depending only on $K$ defined in Lemma 4 of \cite{sun2019escaping}.
The reason for the choice of $\tilde{\epsilon}$ in \eqref{eq:PRGD-parameter2} is given as follows:
\begin{itemize}
	\item (i) $\tilde{\epsilon} \leq  \min \left\{A, B\right\}$ is required in Theorem 1 of \cite{sun2019escaping}.
	\item (ii) $\tilde{\epsilon} \leq  \min \left\{C, D \right\}$ is assumed to guarantee that the $(\tilde{\epsilon}, - \sqrt{\hat{\rho} \tilde{\epsilon} } )$-approximate second-order stationary point, i.e., $\|\overline{\grad\, H([\Y])} \|_\F \leq \tilde{\epsilon}$ and $\lambda_{\min}(\overline{ \Hess\, H([\Y])}) \leq - \sqrt{\hat{\rho} \tilde{\epsilon}}$, that PRGD converges to does not lie in $\cR_2 \cup \cR_3 $ based on \eqref{eq:H-landscape-simplified}. In other words, this ensures that the output of the algorithm lies in $\cR_1$.
	\item (iii) $\tilde{\epsilon} \leq  E$ ensures that the required perturbation scale is small enough so that the assumption regarding $a$ in Lemma \ref{lm:obsorbing} is satisfied. 
\end{itemize}
We note that the $\tilde{\epsilon}$ in \eqref{eq:PRGD-parameter2} can be chosen to be a universal constant depending only on $\sigma_r(\Y^*)$, $\sigma_1(\Y^*)$, $\zeta$, and $r$. For simplicity, we denote it as $\tilde{\epsilon} = C(\sigma_r(\Y^*), \sigma_1(\Y^*), \zeta, r)$.

Based on the above choices of parameters, we can inductively show that the whole trajectory of PRGD will lie in $\widetilde{\cR}_{\text{nice}}$. First, we note that, given the choice of $\tilde{\epsilon}$, it is easy to check that the perturbation scale $\omega$ satisfies $\omega \leq  \frac{1}{50} \min \{ \eta, \frac{1}{15 \kappa^* \sigma^2_1(\Y^*) \sqrt{r} }\} \sigma_r^3(\Y^*)$, where $\eta = \frac{c_{\max}}{L}$ is the stepsize taken in the PRGD algorithm. Moreover, by the choice of $c_{\max}$ in the algorithm, we know
\begin{equation*}
	\eta = \frac{c_{\max}}{L} \leq \frac{25}{(56*16)^2 * 266* \kappa^{*4} * \sigma^2_1(\Y^*)} \leq  \frac{1}{40 \kappa^{*4} \sigma^2_1(\Y^*) }.
\end{equation*} Thus, the conditions in Lemma \ref{lm:obsorbing} are satisfied and from it we know $\Y_{t+0.5} \in \widetilde{\cR}_{\text{nice}}$ and ${\Y}_{t+1} \in \cR_{\text{nice}}\subseteq \widetilde{\cR}_{\text{nice}}$. Since, by design, after every perturbation step PRGD will run Riemannnian GD for $t_{\text{thres}} \geq 1$ many steps. Thus, by induction, we know the whole trajectory of PRGD will lie in $ \widetilde{\cR}_{\text{nice}}$.

By Lemma \ref{lm:curvature-lipschitz-constants} and the fact that, for any $\Y \in \bbR^{p \times r}_*$, the injectivity radius of $\cM_{r+}^q$ at $[\Y]$ is $\sigma_r(\Y)$ (see Theorem 6.3 of \cite{massart2020quotient}), we know that the Lipschitz constants of the Riemannian gradient, Riemannian Hessian, (denoted as $L', \rho'$) sectional curvature (denoted as $K'$) and injectivity radius (denoted as $\mathfrak{I}'$) satisfy  
\begin{equation*}
	\begin{split}
		L' \leq \frac{266}{25} \sigma_1^2(\Y^*) \leq L, \rho' \leq \rho, K' \leq K \text{ and } \mathfrak{I}' \geq \frac{1}{3} \sigma_r(\Y^*).  
	\end{split}
\end{equation*}
Thus, the assumptions in Theorem 1 of \cite{sun2019escaping} are satisfied and by the choice of $\tilde{\epsilon}$, Theorem 1 of \cite{sun2019escaping} implies that, with probability at least $1 - \zeta$, the output of PRGD will be an approximate second-order stationary point lying in $\cR_1$. Moreover, this point is reached in
\begin{equation*}
	\mathcal{O} \left( \frac{L H([\widetilde{\Y}_{T_0}])}{\tilde{ \epsilon }^2} \log^4 \left( \frac{L pr H([\widetilde{\Y}_{T_0}]) }{\tilde{\epsilon }^2 \zeta} \right) \right)
\end{equation*} many iterations, where $\tilde{\epsilon} = C(\sigma_r(\Y^*), \sigma_1(\Y^*), \zeta, r)$. Note that the number of iterations that PRGD needs is at most $C \log^4(p)$ for some $C > 0$ depending only on $\sigma_r(\Y^*), \sigma_1(\Y^*), \zeta, r$.

\subsubsection{Step 3 for the proof of Theorem \ref{th:application}: Local improvement} \label{sec:proof-local}
Once the iterates enter $\cR_1 \cap \cR_{\text{nice}}$, we learn from Lemma \ref{lm:obsorbing0} that vanilla Riemannnian GD will not leave $\cR_{\text{nice}}$ as long as the step size $\eta \leq \frac{1}{40 \kappa^{*4} \sigma_1^2(\Y^*)}$. Then, by the standard theory of linear convergence of Riemannnian GD in a local geodesically strongly convex region (see the proof of Theorem 11.29 of \cite{boumal2020introduction}), we see that with the choice of stepsize $\eta = \frac{1}{40 \kappa^{*4} \sigma_1^2(\Y^*)}$ the iterates will also not leave $\cR_1$ and the algorithm will converge linearly as $d([\Y_t, \Y^*]) \leq (1-1/\kappa')^{t/2} \sqrt{\kappa'} d([\Y_0, \Y^*])$, where $\kappa' = (40 \kappa^{*4} \sigma_1^2(\Y^*))/(26 \sigma^2_r(\Y^*)/75) = \frac{1500}{13} \kappa^{*6}$. Therefore, $\Y_t$ will be $\epsilon$-close to $\Y^*$ in $T_1 = 2\left\lceil  \frac{\log\left( \frac{\epsilon}{2 \sqrt{ \frac{15}{13} \kappa^{*2} \sigma_r(\Y^*) } } \right)}{\log( 1- \frac{13}{1500 \kappa^{6} } )}  \right\rceil $ number of iterations. Note that $T_1 = \mathcal{O}(\log(1/\epsilon))$, where in $\mathcal{O}(\cdot)$, we hide the dependence on $\sigma_1(\Y^*)$ and $\sigma_r(\Y^*)$.

\subsubsection{Proof of Lemma  \ref{lm:curvature-lipschitz-constants}} \label{proof:sec-app-lemma}
First, by Corollary 10.47 of \cite{boumal2020introduction}, we know the Lipschitz constant of the gradient can be controlled by the operator norm of the Hessian:
\begin{equation*}
	\begin{split}
		\overline{\Hess \, H([\Y])}[\theta_\Y, \theta_\Y] &=\|\Y\theta_\Y^\top + \theta_\Y \Y^\top \|_\F^2 + 2\langle \Y \Y^\top -\X^*, \theta_\Y \theta_\Y^\top \rangle \\
		& \leq 4 \| \Y \theta_{\Y} \|^2_{\F} + 2 \langle \theta_\Y^{\top }( \Y \Y^\top -\X^*), \theta_\Y^\top \rangle \\
		& \leq 6 \sigma^2_1(\Y) \|\theta_{\Y}\|_\F^2 + 2 \sigma^2_1(\Y^*) \|\theta_{\Y}\|_\F^2 \leq  (6 c_2^2 + 2 \sigma^2_1(\Y^*) )\|\theta_{\Y}\|_\F^2.
	\end{split}
\end{equation*} This finishes the first claim.

The third claim follows from Proposition 2 of \cite{massart2019curvature}. For the rest of the proof, we focus on showing the Lipschitz continuity of the Hessian. First, we note that by definition of the operator norm:
\[
\begin{split}
	&\Big\| \overline{\text{Hess}\, H([\mathbf{Y}])} - \Gamma_{[\mathbf{Y'}]}^{[\mathbf{Y}]} \circ \overline{\text{Hess}\, H([\mathbf{Y'}])} \circ \Gamma_{[\mathbf{Y}]}^{[\mathbf{Y'}]} \Big\| \\
	&= \sup_{\theta_{\Y} \in   \cH_\Y\widebar{\cM}_{r+}^{q}, \|\theta_{\Y}\|_{\F} \leq 1} \left| 
 \Big\langle \theta_{\mathbf{Y}}, \overline{\text{Hess}\, H([\mathbf{Y}])}[\theta_{\mathbf{Y}}] - \Gamma_{[\mathbf{Y'}]}^{[\mathbf{Y}]} \big( \overline{\text{Hess}\, H([\mathbf{Y'}])}[\Gamma_{[\mathbf{Y}]}^{[\mathbf{Y'}]} \theta_{\mathbf{Y}}] \big) \Big\rangle \right| \\
& = \sup_{\theta_{\Y} \in   \cH_\Y\widebar{\cM}_{r+}^{q}, \|\theta_{\Y}\|_{\F} \leq 1} \left| 
 \underbrace{\langle \theta_{\mathbf{Y}}, \overline{\text{Hess}\, H([\mathbf{Y}])}[\theta_{\mathbf{Y}}] \rangle}_{\text{Term 1}} - \underbrace{\Big\langle \theta_{\mathbf{Y}}, \Gamma_{[\mathbf{Y'}]}^{[\mathbf{Y}]} \big( \overline{\text{Hess}\, H([\mathbf{Y'}])}[\Gamma_{[\mathbf{Y}]}^{[\mathbf{Y'}]} \theta_{\mathbf{Y}}] \big) \Big\rangle}_{\text{Term 2}} \right|.
\end{split}
\] 

We note that {Term 2} can be simplified. A fundamental property of Riemannian parallel transport is that it is an {isometry} (see Proposition 10.36 of \cite{boumal2020introduction}). Therefore, taking the inner product of $U$ and $\Gamma V$ at $[\mathbf{Y}]$ is the same as taking the inner product of $\Gamma^{-1} U$ and $V$ at $[\mathbf{Y'}]$. Let $\theta_{\mathbf{Y'}} = \Gamma_{[\mathbf{Y}]}^{[\mathbf{Y'}]} \theta_{\mathbf{Y}}$ (which means $\theta_{\mathbf{Y}} = \Gamma_{[\mathbf{Y'}]}^{[\mathbf{Y}]} \theta_{\mathbf{Y'}}$). Term 2 becomes:
    \[
    \Big\langle \Gamma_{[\mathbf{Y}]}^{[\mathbf{Y'}]} \theta_{\mathbf{Y}}, \overline{\text{Hess}\, H([\mathbf{Y'}])}[\theta_{\mathbf{Y'}}] \Big\rangle = \langle \theta_{\mathbf{Y'}}, \overline{\text{Hess}\, H([\mathbf{Y'}])}[\theta_{\mathbf{Y'}}] \rangle.
    \]
Putting all these terms back together, we get
\begin{equation} \label{ineq:Hessian-Lipschiz}
	\begin{split}
	& \Big\| \overline{\text{Hess}\, H([\mathbf{Y}])} - \Gamma_{[\mathbf{Y'}]}^{[\mathbf{Y}]} \circ \overline{\text{Hess}\, H([\mathbf{Y'}])} \circ \Gamma_{[\mathbf{Y}]}^{[\mathbf{Y'}]} \Big\| \\
	&= \sup_{\theta_{\Y} \in   \cH_\Y\widebar{\cM}_{r+}^{q}, \|\theta_{\Y}\|_{\F} \leq 1}\Big| \overline{\text{Hess}\, H([\mathbf{Y}])}[\theta_{\mathbf{Y}}, \theta_{\mathbf{Y}}] - \overline{\text{Hess}\, H([\mathbf{Y'}])}[\theta_{\mathbf{Y'}}, \theta_{\mathbf{Y'}}] \Big|,
\end{split}
\end{equation} where
 $\theta_{\mathbf{Y'}} = \Gamma_{[\mathbf{Y}]}^{[\mathbf{Y'}]} \theta_{\mathbf{Y}}$.

Recall the formula for $\overline{\text{Hess}\, H(\mathbf{Y})}$ in \eqref{eq: gradient-Hessian-exp-H}:
\[
\overline{\text{Hess}\, H(\mathbf{Y})}[\theta_{\mathbf{Y}}, \theta_{\mathbf{Y}}] = \| \mathbf{Y}\theta_{\mathbf{Y}}^T + \theta_{\mathbf{Y}}\mathbf{Y}^T \|_{\F}^2 + 2 \langle \mathbf{Y}\mathbf{Y}^T - \X^*, \theta_{\mathbf{Y}}\theta_{\mathbf{Y}}^T \rangle.
\]
Our goal is to bound the absolute difference $\Delta = \Big|\overline{\text{Hess}\, H(\mathbf{Y})}[\theta_{\mathbf{Y}}, \theta_{\mathbf{Y}}] - \overline{\text{Hess}\, H(\mathbf{Y'})}[\theta_{\mathbf{Y'}}, \theta_{\mathbf{Y'}}] \Big|$.
Using the triangle inequality, we can split $\Delta$ into two parts, $\Delta \le \Delta_A + \Delta_B$:
\begin{itemize}
    \item \textbf{Term A (Degree-4 part):} $\Delta_A = \Big| \| \mathbf{Y}\theta_{\mathbf{Y}}^T + \theta_{\mathbf{Y}} \mathbf{Y}^T \|_{\F}^2 - \| \mathbf{Y'}\theta_{\mathbf{Y'}}^T + \theta_{\mathbf{Y'}} \mathbf{Y'}^T \|_{\F}^2 \Big|$.
    \item \textbf{Term B (Degree-2 part):} $\Delta_B = 2 \Big| \langle \mathbf{Y}\mathbf{Y}^T - \X^*, \theta_{\mathbf{Y}} \theta_{\mathbf{Y}}^T \rangle - \langle \mathbf{Y'}\mathbf{Y'}^T - \X^*, \theta_{\mathbf{Y'}} \theta_{\mathbf{Y'}}^T \rangle \Big|$.
\end{itemize}

We assume without loss of generality that $\mathbf{Y'}$ and $\mathbf{Y}$ are {optimally aligned}, meaning:
\[
\|\mathbf{Y} - \mathbf{Y'}\|_{\F} = d([\mathbf{Y'}], [\mathbf{Y}]).
\]
Let $d = d([\mathbf{Y'}], [\mathbf{Y}])$. By Lemma \ref{lm:parallel-transport-distance}, we know 
    \[
    \|\theta_{\mathbf{Y}} - \theta_{\mathbf{Y'}}\|_{\F} \le L_{\Gamma} \cdot d \cdot \|\theta_{\mathbf{Y'}}\|_{\F},
    \] 
where $L_{\Gamma} = \frac{c_2}{c_1^2}$.

{ \noindent \bf Bound Term B}.
First,
\[
\Delta_B \le 2 \Big| \langle \mathbf{Y}\mathbf{Y}^T, \theta_{\mathbf{Y}} \theta_{\mathbf{Y}}^T \rangle - \langle \mathbf{Y'}\mathbf{Y'}^T, \theta_{\mathbf{Y'}} \theta_{\mathbf{Y'}}^T \rangle \Big| + 2 \Big| \langle \X^*, \theta_{\mathbf{Y'}} \theta_{\mathbf{Y'}}^T - \theta_{\mathbf{Y}} \theta_{\mathbf{Y}}^T \rangle \Big|.
\]
We first bound $2 \big| \langle \X^*, \theta_{\mathbf{Y'}} \theta_{\mathbf{Y'}}^T - \theta_{\mathbf{Y}} \theta_{\mathbf{Y}}^T \rangle \big|$. 
Rewrite the difference of the quadratic terms by adding and subtracting $\theta_{\mathbf{Y'}} \theta_{\mathbf{Y}}^T$:
\[
\theta_{\mathbf{Y'}} \theta_{\mathbf{Y'}}^T - \theta_{\mathbf{Y}} \theta_{\mathbf{Y}}^T = \theta_{\mathbf{Y'}}(\theta_{\mathbf{Y'}} - \theta_{\mathbf{Y}})^T + (\theta_{\mathbf{Y'}} - \theta_{\mathbf{Y}})\theta_{\mathbf{Y}}^T.
\]
Substitute this into the inner product (using the trace property $\langle A, B \rangle = \text{Tr}(A^T B)$):
\begin{equation}
	\langle \X^*, \theta_{\mathbf{Y'}} \theta_{\mathbf{Y'}}^T - \theta_{\mathbf{Y}} \theta_{\mathbf{Y}}^T \rangle = \text{Tr}\Big( \X^* \theta_{\mathbf{Y'}}(\theta_{\mathbf{Y'}} - \theta_{\mathbf{Y}})^T \Big) + \text{Tr}\Big( \X^*(\theta_{\mathbf{Y'}} - \theta_{\mathbf{Y}})\theta_{\mathbf{Y}}^T \Big). \label{ineq:X-first-part}
\end{equation}
Using the trace inequality $|\text{Tr}(A B^T)| \le \|A\|_{\F} \|B\|_{\F}$ and $\|MN\|_{\F} \le \|M\|_2 \|N\|_{\F}$:
\[
\Big| \text{Tr}\Big( (\X^* \theta_{\mathbf{Y'}})(\theta_{\mathbf{Y'}} - \theta_{\mathbf{Y}})^T \Big) \Big| \le \|\X^* \theta_{\mathbf{Y'}}\|_{\F} \|\theta_{\mathbf{Y'}} - \theta_{\mathbf{Y}}\|_{\F} \le \|\X^*\|_2 \|\theta_{\mathbf{Y'}}\|_{\F} \|\theta_{\mathbf{Y'}} - \theta_{\mathbf{Y}}\|_{\F}.
\]
Applying this to both traces and using $\|\theta_{\mathbf{Y}}\|_{\F} = \|\theta_{\mathbf{Y'}}\|_{\F}$, due to the isometry of parallel transport (see Proposition 10.36 of \cite{boumal2020introduction}), yields:
\[
2 \Big| \langle \X^*, \theta_{\mathbf{Y'}} \theta_{\mathbf{Y'}}^T - \theta_{\mathbf{Y}} \theta_{\mathbf{Y}}^T \rangle \Big| \le 4 \|\X^*\|_2 \|\theta_{\mathbf{Y'}}\|_{\F} \|\theta_{\mathbf{Y'}} - \theta_{\mathbf{Y}}\|_{\F} \le 4 L_{\Gamma} \|\X^*\|_2 \cdot d \cdot \|\theta_{\mathbf{Y'}}\|_{\F}^2.
\]

{Next, we bound the $ 2 \Big| \langle \mathbf{Y}\mathbf{Y}^T, \theta_{\mathbf{Y}} \theta_{\mathbf{Y}}^T \rangle - \langle \mathbf{Y'}\mathbf{Y'}^T, \theta_{\mathbf{Y'}} \theta_{\mathbf{Y'}}^T \rangle \Big|$ term}. Add and subtract $\langle \mathbf{Y'}\mathbf{Y'}^T, \theta_{\mathbf{Y}} \theta_{\mathbf{Y}}^T \rangle$:
\[
2 \Big| \langle \mathbf{Y}\mathbf{Y}^T - \mathbf{Y'}\mathbf{Y'}^T, \theta_{\mathbf{Y}} \theta_{\mathbf{Y}}^T \rangle \Big| + 2 \Big| \langle \mathbf{Y'}\mathbf{Y'}^T, \theta_{\mathbf{Y}} \theta_{\mathbf{Y}}^T - \theta_{\mathbf{Y'}} \theta_{\mathbf{Y'}}^T \rangle \Big|
\]
For the first piece, $\|\mathbf{Y}\mathbf{Y}^T - \mathbf{Y'}\mathbf{Y'}^T\|_{\F} = \|\Y (\Y - \Y')^{\top } + (\Y - \Y')\Y^{'\top } \|_{\F} \le 2c_2\|\mathbf{Y}-\mathbf{Y'}\|_{\F} = 2c_2 d$, giving the bound:
\[
2 \Big| \langle \mathbf{Y}\mathbf{Y}^T - \mathbf{Y'}\mathbf{Y'}^T, \theta_{\mathbf{Y}} \theta_{\mathbf{Y}}^T \rangle \Big|  \leq 4 c_2 \cdot d \cdot \|\theta_{\mathbf{Y'}}\|_{\F}^2.
\]
For the second piece, we use the exact same analysis as in \eqref{ineq:X-first-part} replacing $\X^*$ with $\mathbf{Y'}\mathbf{Y'}^T$, implying that $4 \|\mathbf{Y'}\mathbf{Y'}^T\|_2 \|\theta_{\mathbf{Y'}}\|_{\F} \|\theta_{\mathbf{Y'}} - \theta_{\mathbf{Y}}\|_{\F}$ upper bounds this second piece. Since $\|\mathbf{Y'}\mathbf{Y'}^T\|_2 \le c_2^2$, this implies:
\[
2 \Big| \langle \mathbf{Y'}\mathbf{Y'}^T, \theta_{\mathbf{Y}} \theta_{\mathbf{Y}}^T - \theta_{\mathbf{Y'}} \theta_{\mathbf{Y'}}^T \rangle \Big| \leq 4 c_2^2 L_{\Gamma} \cdot d \cdot \|\theta_{\mathbf{Y'}}\|_{\F}^2.
\]
Combining all the bounds together, Term B is seen to be bounded as:
\[
\Delta_B \le \Big( 4 c_2 + 4 c_2^2 L_{\Gamma} + 4 \|\X^*\|_2 L_{\Gamma} \Big) \cdot d \cdot \|\theta_{\mathbf{Y'}}\|_{\F}^2.\]

{ \noindent \bf  Bound Term A}.
Next, we bound $\Delta_A = \Big| \| \mathbf{Y}\theta_{\mathbf{Y}}^T + \theta_{\mathbf{Y}} \mathbf{Y}^T \|_{\F}^2 - \| \mathbf{Y'}\theta_{\mathbf{Y'}}^T + \theta_{\mathbf{Y'}} \mathbf{Y'}^T \|_{\F}^2 \Big|$.
Using $|a^2 - b^2| = |a - b| \cdot (a + b)$:
\begin{itemize}
    \item \textbf{The Sum $(a+b)$:}
    $\| \mathbf{Y}\theta_{\mathbf{Y}}^T + \theta_{\mathbf{Y}} \mathbf{Y}^T \|_{\F} \le 2\|\mathbf{Y}\|_2\|\theta_{\mathbf{Y}}\|_{\F} \le 2 c_2 \|\theta_{\mathbf{Y'}}\|_{\F}$. 
    Similarly for $\mathbf{Y'}$, the sum is bounded by $4 c_2 \|\theta_{\mathbf{Y'}}\|_{\F}$.
    \item \textbf{The Difference $|a-b|$:}
    $|a-b| \le 2 \| \mathbf{Y}\theta_{\mathbf{Y}}^T - \mathbf{Y'}\theta_{\mathbf{Y'}}^T \|_{\F}$.
    Add and subtract $\mathbf{Y'}\theta_{\mathbf{Y}}^T$:
    \[
    \mathbf{Y}\theta_{\mathbf{Y}}^T - \mathbf{Y'}\theta_{\mathbf{Y'}}^T = (\mathbf{Y}-\mathbf{Y'})\theta_{\mathbf{Y}}^T + \mathbf{Y'}(\theta_{\mathbf{Y}} - \theta_{\mathbf{Y'}})^T.
    \] 
Then
    \[
    \begin{split}
    	\| \mathbf{Y}\theta_{\mathbf{Y}}^T - \mathbf{Y'}\theta_{\mathbf{Y'}}^T \|_{\F} \le \|\mathbf{Y}-\mathbf{Y'}\|_{\F} \|\theta_{\mathbf{Y}}\|_2 + \|\mathbf{Y'}\|_2 \|\theta_{\mathbf{Y}} - \theta_{\mathbf{Y'}}\|_{\F}\\
    	\le d \|\theta_{\mathbf{Y'}}\|_{\F} + c_2 L_{\Gamma} d \|\theta_{\mathbf{Y'}}\|_{\F} = (1 + c_2 L_{\Gamma}) d \|\theta_{\mathbf{Y'}}\|_{\F}
    \end{split}
    \]
\end{itemize}
Multiplying the sum and the difference gives:
\[
\Delta_A \le (4 c_2 \|\theta_{\mathbf{Y'}}\|_{\F}) \times 2(1 + c_2 L_{\Gamma}) d \|\theta_{\mathbf{Y'}}\|_{\F}
 \le 8 c_2 (1 + c_2 L_{\Gamma}) \cdot d \cdot \|\theta_{\mathbf{Y'}}\|_{\F}^2 \]

By combining Term A and Term B and plugging the value of $L_{\Gamma} = c_2/c_1^2$, we get 
\[
 \Big|\overline{\text{Hess}\, H(\mathbf{Y})}[\theta_{\mathbf{Y}}, \theta_{\mathbf{Y}}] - \overline{\text{Hess}\, H(\mathbf{Y'})}[\theta_{\mathbf{Y'}}, \theta_{\mathbf{Y'}}] \Big| \le 4 \left( 3 c_2 (\kappa^2 +1) + \sigma^2_1(\Y^*) \kappa/c_1  \right) d([\Y], [\Y'] ) \| \theta_{\mathbf{Y}} \|^2_{\F},
\] where we use the fact that $\| \theta_{\mathbf{Y}} \|^2_{\F} = \| \theta_{\mathbf{Y}'} \|^2_{\F}$ since the parallel transport is an isometry between the corresponding tangent planes. In view of \eqref{ineq:Hessian-Lipschiz}, we have finished the proof of the second claim.

\subsubsection{Proofs of Lemma \ref{lm:obsorbing0} and \ref{lm:obsorbing} }
\label{sec:AdditionalLemmas}
The proof of Lemma \ref{lm:obsorbing0} is very similar to the proof of Lemma \ref{lm:obsorbing}. Here, for simplicity, we present the proof for Lemma \ref{lm:obsorbing} only. First, let us define 
\begin{equation*}
	\begin{split}
		&\widebar{\cR}_{\text{nice}} = \Big \{ \Y = [\K^\top \, \J^{\top}]^{\top}: \K \in \bbR^{r \times r}, \J \in \bbR^{(p-r) \times r}, \\ 
		& \quad \quad \sigma_1(\Y) \leq (\frac{11}{10}+a) \sigma_1(\Y^*), \sigma_1(\J) \leq (\frac{1}{2} + a) \sigma_r(\Y^*), \sigma_r(\K) \geq (\frac{1}{2}-a) \sigma_r(\Y^*) \Big \}.
	\end{split}
\end{equation*} Since $\|\xi\|_\F \leq a \sigma_r(\Y^*)$, we know that given $\Y_t \in \cR_{\text{nice}}$, then $\Y_{t+0.5} \in \widebar{\cR}_{\text{nice}}$. Next, we show ${\Y}_{t+1} \in \cR_{\text{nice}}$ by leveraging the relations between the spectra of ${\Y}_{t+1}$ and $\Y_{t+0.5}$ provided in Lemma \ref{lm:iterate-control}. Write $\Y_{t+0.5} = [\K_{t+0.5}^\top \, \J_{t+0.5}^{\top}]^{\top}$ and $\Y_{t+1} = [\K_{t+1}^\top \, \J_{t+1}^{\top}]^{\top}$, where $\K_{t+0.5}, \K_{t+1} \in \bbR^{r \times r}$ and $ \J_{t+0.5}, \J_{t+1}  \in \bbR^{(p-r) \times r}$. Given that $\Y_{t+0.5} \in \widebar{\cR}_{\text{nice}}$, we deduce
\begin{equation*}
	\begin{split}
		\|\overline{\grad\, H([\Y_{t+0.5}])}\|_{\F} &= \|2(\Y_{t+0.5} \Y_{t+0.5}^\top - \X^*) \Y_{t+0.5}\|_{\F} \leq 2((11/10+a)^2 + 1 )(11/10+a)\sigma_1^3(\Y^*) \sqrt{r} \\
		& \leq 5 \sigma_1^3(\Y^*) \sqrt{r},
	\end{split}
\end{equation*} where the last inequality is because, under our assumptions on $\eta$ and $a$, we have $a \leq 1/100$. Thus, $\bar{\eta}_t \geq \min \{ \eta, \frac{1}{15 \kappa^* \sigma^2_1(\Y^*) \sqrt{r} }\}:= \tilde{\eta}$.

We note that our assumption on $\eta$ and $a$ implies that 
\begin{equation} \label{eq:eta-a-condition}
	\begin{split}
		2\bar{\eta}_t \leq 2\eta \leq \frac{1}{16 \kappa^{*4}  (11/10 +a)^2 \sigma^2_1(\Y^*)} \leq \frac{1}{16 \sigma^2_1(\Y_{t+0.5})} \quad \text{ and }\quad a \leq \frac{1}{50} \tilde{\eta} \sigma_r^2(\Y^*),
	\end{split}
\end{equation} where $\kappa^* = \sigma_1(\Y^*)/\sigma_r(\Y^*)$.

\vskip.2cm
{\noindent \bf Control of $\Y_{t+1}$}. By Lemma \ref{lm:iterate-control} (i), we have 
\begin{equation*}
	\begin{split}
		\sigma_1(\Y_{t+1}) &\leq (1 + 2\bar{\eta}_t \sigma^2_1(\Y^*) - 2\bar{\eta}_t \sigma^2_1(\Y_{t+0.5}) ) \sigma_1(\Y_{t+0.5}) \\
		& \overset{(a)}\leq \left(1 + 2\bar{\eta}_t \sigma^2_1(\Y^*) - 2\bar{\eta}_t (11/10 + a)^2 \sigma^2_1(\Y^*)\right) (11/10 + a) \sigma_1(\Y^*) \\
		& \leq \left( \frac{11}{10} + a \right) \left(1 - \frac{21}{100} 2\bar{\eta}_t  \sigma^2_1(\Y^*)\right)\sigma_1(\Y^*) \\
		& \leq \frac{11}{10} \sigma_1(\Y^*) +a \sigma_1(\Y^*) - \frac{231}{1000}  2\bar{\eta}_t \sigma_1^3(\Y^*)\\
		& \leq \frac{11}{10} \sigma_1(\Y^*) +a \sigma_1(\Y^*) - \frac{231}{1000}  2\tilde{\eta} \sigma_1^3(\Y^*)\\
		& \overset{ \eqref{eq:eta-a-condition} }\leq  \frac{11}{10} \sigma_1(\Y^*), 
	\end{split}
\end{equation*} where (a) is because $g_1(s)= (1 +  2\bar{\eta}_t \sigma^2_1(\Y^*) - 2\bar{\eta}_t s^2)s$ is increasing when $s \in [0, 1/\sqrt{6\bar{\eta}_t} ]$ and $ \sigma_1(\Y_{t+0.5}) \leq (11/10 + a) \sigma_1(\Y^*) \leq 1/\sqrt{32\bar{\eta}_t} $.

\vskip.2cm
{\noindent \bf Control of $\J_{t+1}$}.
\begin{equation*}
	\begin{split}
		\sigma_1(\J_{t+1}) \overset{\text{Lemma } \ref{lm:iterate-control} (ii) }\leq & \left(1 - 2\bar{\eta}_t (\sigma^2_1(\J_{t+0.5}) + \sigma^2_1(\K_{t+0.5})) \right) \sigma_1(\J_{t+0.5}) \\
		\leq  & \left(1 - 2\bar{\eta}_t \sigma^2_1(\J_{t+0.5})  \right) \sigma_1(\J_{t+0.5}) \\
		\overset{(a)}\leq & \left(1 - 2\bar{\eta}_t  (1/2 + a)^2 \sigma^2_r(\Y^*)  \right) (1/2 + a) \sigma_r(\Y^*) \\
		& \leq \left(1 - 2\bar{\eta}_t\sigma^2_r(\Y^*)/4  \right) (1/2 + a) \sigma_r(\Y^*) \\
		& \leq \frac{1}{2} \sigma_r(\Y^*) + a  \sigma_r(\Y^*) - \frac{1}{8} 2\bar{\eta}_t \sigma_r^3(\Y^*) \\
		& \leq \frac{1}{2} \sigma_r(\Y^*) + a  \sigma_r(\Y^*) - \frac{1}{8} 2\tilde{\eta} \sigma_r^3(\Y^*) \\
		& \overset{ \eqref{eq:eta-a-condition} }\leq \frac{1}{2} \sigma_r(\Y^*),
	\end{split}
\end{equation*} where (a) is because $g_2(s) = (1 - 2\bar{\eta}_t s^2 )s$ in increasing when $s \in [0, 1/\sqrt{6\bar{\eta}_t} ]$ and  $ \sigma_1(\J_{t+0.5}) \leq (1/2 + a) \sigma_r(\Y^*) \leq 1/\sqrt{32\bar{\eta}_t} $.

\vskip.2cm
{\noindent \bf Control of $\K_{t+1}$}.
\begin{equation*}
	\begin{split}
		&\sigma^2_{r}(\K_{t+1}) \\
		\geq & \left(1 + 4 \bar{\eta}_t(\sigma^2_r(\Y^*) - \sigma^2_1(\J_{t+0.5}) - \sigma^2_r(\K_{t+0.5}) ) \right) \sigma_r^2(\K_{t+0.5}) - 8 \bar{\eta}_t^2 \sigma^4_1(\Y^*) \sigma^2_1(\Y_{t+0.5}) \\
		\overset{(a)}\geq &\left(1 + 4 \bar{\eta}_t(\sigma^2_r(\Y^*) - (1/2 + a)^2\sigma^2_r(\Y^*) - \sigma^2_r(\K_{t+0.5}) ) \right) \sigma_r^2(\K_{t+0.5}) - 8 \bar{\eta}_t^2 (11/10+a)^2 \sigma^6_1(\Y^*)\\
		\overset{(b)}\geq & \left(1 + 4 \bar{\eta}_t(\sigma^2_r(\Y^*) - (1/2 + a)^2\sigma^2_r(\Y^*) - (1/2 -a)^2 \sigma_r^2(\Y^*)  ) \right) (1/2-a)^2 \sigma_r^2(\Y^*) - 8 \bar{\eta}_t^2 (11/10+a)^2 \sigma^6_1(\Y^*)\\
		\overset{(c)}\geq & \left(1 + 4 \bar{\eta}_t \frac{49}{50} \sigma^2_r(\Y^*) \right) ( \frac{1}{4} -a) \sigma_r^2(\Y^*) - 8 \bar{\eta}_t^2 (11/10+a)^2 \sigma^6_1(\Y^*) \\
		\overset{(d)}\geq & \frac{1}{4} \sigma^2_r(\Y^*) - a \sigma^2_r(\Y^*) + 2\bar{\eta}_t \frac{294}{625} \sigma^4_r(\Y^*) - 8 \bar{\eta}_t^2 (11/10+a)^2 \sigma^6_1(\Y^*) \\
		\geq & \frac{1}{4} \sigma^2_r(\Y^*) - a \sigma^2_r(\Y^*) + \tilde{\eta} \frac{294}{625} \sigma^4_r(\Y^*) + \bar{\eta}_t \frac{294}{625} \sigma^4_r(\Y^*) - 8 \bar{\eta}_t^2 (11/10+a)^2 \sigma^6_1(\Y^*) \\
		\overset{(e)}\geq & \frac{1}{4} \sigma^2_r(\Y^*),
	\end{split}
\end{equation*} where (a) is because $ \sigma_1(\Y_{t+0.5}) \leq (\frac{11}{10}+a) \sigma_1(\Y^*)$ and $\sigma_1(\J_{t+0.5}) \leq (1/2 + a) \sigma_r(\Y^*)$; (b) is because $g_3(s) = \left(1 + 4 \bar{\eta}_t(\sigma^2_r(\Y^*) - (1/2 + a)^2\sigma^2_r(\Y^*) - s ) \right) s$ is increasing on $s \in [0, 1/(8\bar{\eta}_t)]$ and $ \frac{1}{32 \bar{\eta}_t} \geq  \sigma_r^2(\K_{t+0.5}) \geq (1/2 - a)^2 \sigma_r^2(\Y^*)$; (c) is because $a \leq 1/100$ based on \eqref{eq:eta-a-condition}; (d) is again based on $a \leq 1/100$; (e) is because \eqref{eq:eta-a-condition} implies that the following two conditions hold:
\begin{equation*}
	\begin{split}
		 \tilde{\eta} \frac{294}{625} \sigma^4_r(\Y^*) \geq a \sigma^2_r(\Y^*) \quad \text{ and }
		 \quad  \bar{\eta}_t \frac{294}{625} \sigma^4_r(\Y^*) \geq 8 \bar{\eta}_t^2 (11/10+a)^2 \sigma^6_1(\Y^*). 
	\end{split}
\end{equation*} In summary, we have proved that, provided the two conditions on $\eta$ and $a$ in \eqref{eq:eta-a-condition} hold, then we have ${\Y}_{t+1} \in \cR_{\text{nice}}$. In addition, we note that the two conditions in \eqref{eq:eta-a-condition} are satisfied whenever 
\begin{equation*}
	\begin{split}
		2\eta \leq \frac{1}{20 \kappa^{*^4} \sigma^2_1(\Y^*) } \quad \text{ and }\quad a \leq \frac{1}{50} \tilde{\eta} \sigma_r^2(\Y^*).
	\end{split}
\end{equation*} Finally, we note $a \leq 1/2000$ under these two conditions, so $\widebar{\cR}_{\text{nice}} \subseteq \widetilde{\cR}_{\text{nice}}$. This finishes the proof of this lemma.

\subsubsection{Auxiliary Lemmas} \label{proof:sec-app-add-lemma}

\begin{Lemma}\label{lm:iterate-control}
	Suppose ${\Y}_{t+1} = \Y_{t} - \eta (\Y_{t} \Y_{t}^{\top} - \X^*)\Y_{t}$ where $\X^*$ is a rank $r$ diagonal PSD matrix.  Write $\Y_t = [\K_t^\top \, \J_t^{\top}]^{\top}$ and $\Y_{t+1} = [\K_{t+1}^\top \, \J_{t+1}^{\top}]^{\top}$ where $\K_t, \K_{t+1} \in \bbR^{r \times r}$ and $ \J_t, \J_{t+1}  \in \bbR^{(p-r) \times r}$. Suppose $\sigma_1(\Y_t) \leq 1/\sqrt{16\eta}$, then 
	\begin{itemize}
		\item (i) $\sigma_1(\Y_{t+1}) \leq (1 + \eta \sigma_1(\X^*) - \eta \sigma^2_1(\Y_t) ) \sigma_1(\Y_t).$
		\item (ii) $\sigma_1(\J_{t+1}) \leq \left(1 - \eta (\sigma^2_1(\J_t) + \sigma^2_1(\K_t)) \right) \sigma_1(\J_t).$
		\item (iii) $\sigma^2_{r}(\K_{t+1}) \geq \left(1 + 2 \eta(\sigma_r(\X^*) - \sigma^2_1(\J_t) - \sigma^2_r(\K_t) ) \right) \sigma_r^2(\K_t) - 2 \eta^2 \sigma^2_1(\X^*) \sigma^2_1(\Y_t)$.
	\end{itemize}
\end{Lemma}
\begin{proof}
	The proof of these three claims is very similar to the proof of Lemma A.4-A.6 in \cite{chen2023fast} after a slight and straightforward modification, since here we have $\sigma_{r+1}(\X^*) = 0$. For simplicity, we omit the details of the proof. 
\end{proof}

\begin{Lemma}[Lemma C.7 of \cite{jiang2023algorithmic}] \label{lm:K-lower-bound}
	Let $\mathbf{K}$ be an $r \times r$ matrix and ${\bf \Lambda}_r = \text{diag}(\lambda_1^*, \dots, \lambda_r^*) \in \mathbb{R}^{r \times r}$ be a diagonal matrix with $\lambda_1^* \geq \dots \geq \lambda_r^* > 0$. Suppose $\sigma_r(\mathbf{K}) > 0$, $\eta \sigma_1({\bf \Lambda}_r - \mathbf{K} \mathbf{K}^{\top}) \leq 1/2$, $\sigma_1(\mathbf{K}) \leq 1/\sqrt{3\eta}$, $2 \eta^2 \sigma_1({\bf \Lambda}_r \mathbf{K} \mathbf{K}^{\top}) < 1, \text{ and}$

\[
\tilde{\mathbf{K}} = \mathbf{K} + \eta ({\bf \Lambda}_r - \mathbf{K} \mathbf{K}^{\top}) \mathbf{K}.
\]
Then we have $\sigma_r(\tilde{\mathbf{K}}) \geq (1 - 2 \eta^2 \sigma_1({\bf \Lambda}_r \mathbf{K} \mathbf{K}^{\top})) (1 + \eta \lambda_r^*) \sigma_r(\mathbf{K}) (1 - \eta \sigma_r^2(\mathbf{K}))$.
\end{Lemma}

\begin{Lemma}\label{lm:parallel-transport-distance}
	Recall the definition of $\cR_{c_1, c_2}$ in \eqref{def:Rc1c2}. Given any $\Y, \Y' \in \cR_{c_1, c_2}$, and any $\theta_{\Y'} \in  \cH_{\Y'}\widebar{\cM}_{r+}^{q}$, we have 
	\begin{equation*}
		\left\|\theta_{\Y'} - \Gamma_{[\Y']}^{[\Y]} \theta_{\Y'} \right\|_\F \leq \frac{c_2}{c_1^2} d([\Y], [\Y']) \|\theta_{\Y'}\|_{\F},
	\end{equation*} where $\Gamma_{[\Y']}^{[\Y]}$ denotes a parallel transport that transports $v \in  \cH_{\Y'}\widebar{\cM}_{r+}^{q} $ to $ \Gamma_{[\Y']}^{[\Y]} v \in  \cH_\Y\widebar{\cM}_{r+}^{q} $.
\end{Lemma}
\begin{proof} First, we set up the notation.
	\begin{itemize}
    \item We choose $\Y', \Y \in \cR_{c_1, c_2}$ such that they are {optimally aligned}, meaning their Euclidean distance equals the Riemannian distance: $\|\Y - \Y'\|_\F = d([\Y'], [\Y]) := d$.
    \item Let $\Y(t)$ for $t \in [0, 1]$ be the {horizontal geodesic} connecting $\Y'$ to $\Y$ in the total space. Because geodesics have constant speed, we have $\|\dot{\Y}(t)\|_{\F} = d$ for all $t \in [0,1]$. Moreover, by Proposition 4.4 and 4.5 of \cite{massart2020quotient}, we know $\Y(t) = \Y' + t(\Y - \Y')$.
    \item Let $\theta(t)$ be the parallel transport of the initial horizontal vector $\theta(0) = \theta_{\Y'}$ along $\Y(t)$. By definition, $\theta(1) = \theta_{\Y}$.
\end{itemize}

{\noindent \bf Step 1: The ODE for Parallel Transport on a Quotient Manifold}.
In a Riemannian submersion with the standard Euclidean metric, the parallel transport of a horizontal vector $\theta(t)$ along a curve $\Y(t)$ is governed by the condition that the ordinary Euclidean derivative $\dot{\theta}(t)$ must be entirely vertical. 

The vertical space at $\Y(t)$ consists of matrices of the form $\Y(t)\bOmega$, where $\bOmega \in \mathbb{R}^{r \times r}$ is a skew-symmetric matrix ($\bOmega^T = -\bOmega$).
Therefore, the parallel transport curve $\theta(t)$ satisfies the Initial Value Problem (ODE):
\[
\dot{\theta}(t) = \Y(t)\bOmega(t), \quad \theta(0) = \theta_{\Y'}.
\]
To bound the Euclidean difference $\|\theta(1) - \theta(0)\|_{\F}$, we must find a way to bound the magnitude of $\bOmega(t)$.

{\noindent \bf Step 2: Differentiating the Horizontal Condition}.
Since $\theta(t)$ evolves by parallel transport, it must remain a horizontal vector at every point along the curve $\Y(t)$. By Lemma \ref{lm: psd-quotient-manifold1-prop}, the condition for $\theta(t) \in  \cH_{\Y(t)}\widebar{\cM}_{r+}^{q}$ is that $\Y(t)^T \theta(t)$ is symmetric. Equivalently:
\[
\Y(t)^T \theta(t) - \theta(t)^T \Y(t) = 0.
\]
Since this holds for all $t$, we can differentiate it with respect to $t$:
\[
\dot{\Y}^T \theta + \Y^T \dot{\theta} - \dot{\theta}^T Y - \theta^T \dot{\Y} = 0
\]
\textit{(Here we drop the $(t)$ temporarily for cleaner notation).} Substitute our ODE $\dot{\theta} = \Y\bOmega$ and $\dot{\theta}^T = -\bOmega \Y^T$ (since $\bOmega$ is skew-symmetric) into the equation:
\[
\dot{\Y}^T \theta + \Y^T (\Y\bOmega) - (-\bOmega \Y^T) Y - \theta^T \dot{\Y} = 0
\]
Rearranging the terms to isolate $\bOmega$ yields a continuous Sylvester equation:
\[
\Y^{\top} \Y \bOmega + \bOmega \Y^{\top} \Y = \theta^T \dot{\Y} - \dot{\Y}^T \theta.
\]
{\noindent \bf Step 3: Bounding the Skew-Symmetric Matrix $\bOmega(t)$}.
We view the left side of the equation as a linear operator $\mathcal{L}(\bOmega) = \Y^{\top} \Y \bOmega + \bOmega \Y^{\top} \Y$.
Because $\Y^{\top} \Y$ is symmetric and positive definite, the eigenvalues of the operator $\mathcal{L}$ are exactly the sums of pairs of eigenvalues of $\Y^{\top} \Y$. That is, the eigenvalues of $\mathcal{L}$ are $\sigma_i^2(\Y) + \sigma_j^2(\Y)$, see the proof of Theorem VII.2.1 of \cite{bhatia2013matrix}, where $\sigma_i(\Y)$ are the singular values of $\Y$.

Because our curve $\Y(t)$ is entirely contained within the restricted domain $\cR_{c_1,c_2}$, we know that for all $t$:
\[
\sigma_r(\Y(t)) \geq t\sigma_r(\Y) + (1-t) \sigma_r(\Y') \geq c_1 > 0.
\] 
Therefore, the smallest eigenvalue of the operator $\mathcal{L}$ is bounded below by $2c_1^2$. By the properties of linear operators, this bounds the Frobenius norm of the inverse operator:
\[
\|\bOmega\|_{\F} \le \frac{1}{2c_1^2} \big\| \theta^T \dot{\Y} - \dot{\Y}^T \theta \big\|_{\F} \le \|\theta^T \dot{\Y}\|_{\F} + \|\dot{\Y}^T \theta\|_{\F} \le 2 \|\theta(t)\|_{\F} \|\dot{\Y}(t)\|_2.
\]
Since $\dot{\Y}(t)$ is the velocity vector of the geodesic, $\|\dot{\Y}(t)\|_2 \le \|\dot{\Y}(t)\|_{\F} = d$.
Moreover, given that parallel transport is an isometry (see Proposition 10.36 of \cite{boumal2020introduction}), $\|\theta(t)\|_{\F} = \|\theta(0)\|_{\F} = \|\theta_{\Y}\|_{\F}$. Substituting these bounds back into the inequality for $\bOmega$:
\begin{equation} \label{ineq:Omega-t-bound}
	\|\bOmega(t)\|_{\F} \le \frac{1}{2c_1^2} \cdot 2 \|\theta_{\Y}\|_{\F} \cdot d = \frac{d}{c_1^2} \|\theta_{\Y}\|_{\F}
\end{equation}

{ \noindent \bf  Step 4: Integrating to Bound the Ambient Difference}.
We want to bound the total Euclidean change in $\theta$ from $t=0$ to $t=1$. By the Fundamental Theorem of Calculus and the triangle inequality for integrals:
\[
\theta_{\Y} - \theta_{\Y'} = \theta(1) - \theta(0) = \int_0^1 \dot{\theta}(t) dt,
\]
\[
\|\theta_{\Y} - \theta_{\Y'}\|_{\F} \le \int_0^1 \|\dot{\theta}(t)\|_{\F} dt.
\]
Substituting $\dot{\theta}(t) = \Y(t)\bOmega(t)$:
\[
\|\dot{\theta}(t)\|_{\F} = \|\Y(t)\bOmega(t)\|_{\F} \le \|\Y(t)\|_2 \|\bOmega(t)\|_{\F} \leq  c_2 \left( \frac{d}{c_1^2} \|\theta_{\Y}\|_{\F} \right),
\]
where the last inequality is because $\Y(t) \in \cR_{c_1, c_2}$ and \eqref{ineq:Omega-t-bound}.

Because this bound is independent of $t$, we may integrate it from 0 to 1 to deduce:
\[
\|\theta_{\Y} - \theta_{\Y'}\|_{\F} \le \left( \frac{c_2}{c_1^2} \right) \cdot d \cdot \|\theta_{\Y}\|_{\F}
\]
This finishes the proof.
\end{proof}

\subsection{Random Initialization Guarantee} \label{sec:random-initialization}
\begin{Lemma}
\label{lem:RandomIntial}
	Given any $c \in (0,1/(2\sqrt{3}))$. Suppose $\widetilde{\Y}_0 = c \frac{\sigma_r^2(\Y^*)}{\sigma_1(\Y^*) p} \mathbf{N}_0$ with entries of $\mathbf{N}_0 \in \bbR^{p \times r}$ drawn i.i.d. from $N( 0, \frac{1}{p} )$. Let $\K_0 = \U^{*\top } \widetilde{\Y}_0$ and $\J_0 = \U^{*\top }_{\perp} \widetilde{\Y}_0$, where $\U^* \in \st(r,p)$ spans the top $r$ left singular subspace of $\Y^*$ and $\U^*_{\perp}\in \st(p-r,p)$ is its orthogonal complement. Then with probability at least $1- C_1 c - \exp(-c_1 r) - 2\exp(-p/2)$ for some universal constants $C_1, c_1 > 0$, the initialization condition required in Theorem \ref{th:application} is satisfied.
\end{Lemma}
\begin{proof}
	First, by Corollary 5.35 of \cite{vershynin2010introduction}, we know that, with probability at least $1 - 2 \exp(-p/2)$, we have 
	\begin{equation} \label{ineq:concentration1}
		\begin{split}
			\max\{\sigma_1(\J_0), \sigma_r(\K_0) \} \leq \sigma_1(\widetilde{\Y}_0) \leq  \sqrt{3} c \frac{\sigma_r^2(\Y^*)}{\sigma_1(\Y^*) p} \leq \frac{1}{2} \sigma_r(\Y^*) \leq \sigma_1(\Y^*).
		\end{split}
	\end{equation} At the same time, we note that $\U^{*\top } \mathbf{N}_0$ has i.i.d. $N(0,1/p)$ entries. Thus, by Theorem 1.1 of \cite{rudelson2009smallest}, we know that for every $c'' > 0$, with probability at least $1 - C_1c'' - \exp(-c_1r) $, we have 
	\begin{equation} \label{ineq:concentration2}
		\sigma_r(\K_0) \geq  c\frac{\sigma_r^2(\Y^*)}{\sigma_1(\Y^*) \sqrt{pr}} c'',
	\end{equation} where $C_1,c_1 > 0$ are two universal constants. Thus, $\sigma_r(\K_0) $ is at least of order $1/p$. Assuming the events \eqref{ineq:concentration1} and \eqref{ineq:concentration2} hold, we may take $c'' /8 = 3 c$ to deduce
\begin{equation*}
	\begin{split}
		\sigma^2_1(\J_0) \leq \left(  \sqrt{3} c \frac{\sigma_r^2(\Y^*)}{\sigma_1(\Y^*) p}  \right)^2 \leq  c\frac{\sigma_r^2(\Y^*)}{\sigma_1(\Y^*) \sqrt{pr}} c'' \cdot \frac{\sigma^2_r(\Y^*)}{8 \sigma_1(\Y^*)}\leq  \frac{\sigma^2_r(\Y^*)}{8 \sigma_1(\Y^*)} \sigma_r(\K_0).
	\end{split}
\end{equation*} This finishes the proof of this lemma.
\end{proof}

}

\section{Proofs in Section \ref{sec: landscape-convex-f} } \label{proof-sec: h-landscape-f-convex}
\subsection{Proof of Theorem \ref{th: local-landscape-f-convex}}

	First, given any non-zero $\theta_{\widehat{\Y}} \in \cH_{ \widehat{\Y} } \widebar{\cM}_{r+}^q$, we have $\theta_{\widehat{\Y}} \theta_{\widehat{\Y}}^\top + \widehat{\Y} \widehat{\Y}^\top$ is a tangent vector of the set $\{ \bbS^{p \times p} \ni \X \succcurlyeq 0 , \rank(\X) \leq 2r  \}$ \cite[Chapter 6]{rockafellar2009variational}. Since $\widehat{\Y} \widehat{\Y}^\top$ is a rank $r$ local minimizer of $\min_{\bbS^{p \times p} \ni \X \succcurlyeq 0 , \rank(\X) \leq 2r  } f(\X)$, the first-order optimality condition implies that $ \langle  \nabla f(\widehat{\Y} \widehat{\Y}^\top ) , \C - \widehat{\Y} \widehat{\Y}^\top \rangle \geq 0 $ holds for any tangent vector $\C$. Thus, we have $ \langle  \nabla f(\widehat{\Y} \widehat{\Y}^\top ) ,  \theta_{\widehat{\Y}} \theta_{\widehat{\Y}}^\top \rangle \geq 0 $ by setting $\C =\theta_{\widehat{\Y}} \theta_{\widehat{\Y}}^\top + \widehat{\Y} \widehat{\Y}^\top$.
	
	Thus, by Lemma \ref{lm: gradient-hessian-exp-PSD}, we have
	\begin{equation*}
		\begin{split}
			\overline{\Hess \, h([\widehat{\Y}])}[\theta_{\widehat{\Y}}, \theta_{\widehat{\Y}}] & = \nabla^2 f(\widehat{\Y} \widehat{\Y}^\top)[\widehat{\Y}\theta_{\widehat{\Y}}^\top + \theta_{\widehat{\Y}} \widehat{\Y}^\top , \widehat{\Y} \theta_{\widehat{\Y}}^\top + \theta_{\widehat{\Y}} \widehat{\Y}^\top ] + 2\langle \nabla f(\widehat{\Y} \widehat{\Y}^\top ), \theta_{\widehat{\Y}} \theta_{\widehat{\Y}}^\top \rangle \\
			& \geq \nabla^2 f(\widehat{\Y} \widehat{\Y}^\top)[\widehat{\Y}\theta_{\widehat{\Y}}^\top + \theta_{\widehat{\Y}} \widehat{\Y}^\top , \widehat{\Y}\theta_{\widehat{\Y}}^\top + \theta_{\widehat{\Y}}\widehat{\Y}^\top ] > 0,
		\end{split}
	\end{equation*} where the last inequality is because $f$ satisfies the $(r,2r)$-restricted strict convexity property and because $\widehat{\Y}\theta_{\widehat{\Y}}^\top + \theta_{\widehat{\Y}} \widehat{\Y}^\top$ is nonzero by virtue of Lemma \ref{lm: norm-bound-Ytheta-thetaY}\nc. This implies that there is a neighborhood $\cN$ around $[\widehat{\Y}]$ such that $\overline{\Hess \, h([\Y])} \succcurlyeq 0$ for any $[\Y] \in \cN$. The result follows by combining the above with Theorem \ref{th: convexity-radius-Mq}. 

\section{Additional proofs in Section \ref{sec: H-landscape-analysis}} \label{proof-sec: add-H-landscape}

\subsection{ Proof of Lemma \ref{lm: positive-definite-Hessian-radius-embedded-geometry} }
\label{app:Lemma5}

Suppose $\X \in \cM^e_{r+}$ has eigendecomposition $\U \bSigma \U^\top$, $\xi_\X = [\U \quad \U_\perp] \begin{bmatrix}
			\S & \D^\top\\
			\D & \0
		\end{bmatrix} [\U \quad \U_\perp]^\top \in T_{\X}\cM^e_{r+}$. Then by \cite[Proposition 2]{luo2021nonconvex}, we have
	\begin{equation*}
		\Hess\, \widetilde{H}(\X)[\xi_\X, \xi_\X] = \|\xi_\X\|_\F^2 + 2\langle \X - \X^* , \U_\perp  \D \bSigma^{-1} \D^\top \U_\perp^\top \rangle.
	\end{equation*}
	
	Moreover,
	\begin{equation*}
		\begin{split}
			& 2\langle \X - \X^*  ,  \U_\perp  \D \bSigma^{-1} \D^\top \U_\perp^\top \rangle \\
			= & 2\langle \X - \X^*  ,  \U_\perp  \D \U^\top \U \bSigma^{-1}  \U^\top \U \D^\top\U_\perp^\top \rangle \\
			= & 2\langle \X - \X^*  ,  P_{\U_\perp} \xi_\X P_\U  \X^{-1} P_\U \xi_\X  P_{\U_\perp}  \rangle.
		\end{split}
	\end{equation*}
		Thus
		\begin{equation*}
		\begin{split}
			& \left| 2\langle \X - \X^*  ,  \U_\perp  \D \bSigma^{-1} \D^\top \U_\perp^\top \rangle\right| \\
			 \leq & 2 \|\X - \X^*\|_\F \| P_{\U_\perp} \xi_\X P_\U \|_\F^2 \sigma_1(\X^{-1}) = 2 \|\X - \X^*\|_\F \| \xi_\X \|_\F^2/ \sigma_r(\X).
		\end{split}
		\end{equation*}
		
		Finally, $\|\X - \X^*\|_\F \leq \mu' \sigma_r(\X^*)$ implies $\sigma_r(\X) \geq \sigma_r(\X^*) - \| \X - \X^* \| \geq (1 - \mu')\sigma_r(\X^*)$. Thus
	\begin{equation*}
		\begin{split}
			\Hess\, \widetilde{H}(\X)[\xi_\X, \xi_\X] & \geq  \|\xi_\X\|_\F^2 - 2 \|\X - \X^*\|_\F \| \xi_\X \|_\F^2/ \sigma_r(\X)  \\
			& \geq \|\xi_\X\|_\F^2  - 2 \|\X - \X^*\|_\F \| \xi_\X \|_\F^2/((1-\mu') \sigma_r(\X^*)) \geq (1 - \frac{2\mu'}{1-\mu'} )\|\xi_\X\|_\F^2.
		\end{split}
	\end{equation*}	
	
\subsection{Proof of Proposition \ref{prop: hH-grad-hessian-connection} }	
	First,
	\begin{equation*}
		\begin{split}
			&\| \overline{\grad\, H([\Y])} - \overline{\grad\, h([\Y])}  \|_\F \\
		\overset{ \text{Lemma }\ref{lm: gradient-hessian-exp-PSD}, \eqref{eq: gradient-Hessian-exp-H} }	= &  \max_{\bDelta: \|\bDelta\|_\F = 1} 2 \langle (\nabla f(\Y \Y^\top) - (\Y \Y^\top -\X^*)) \Y , \bDelta \rangle \\
			=& \max_{\bDelta: \|\bDelta\|_\F = 1} 2 \langle \nabla f(\Y \Y^\top) - (\Y \Y^\top -\X^*) , \bDelta \Y^\top \rangle \\
			  = &  \max_{\bDelta: \|\bDelta\|_\F = 1} 2 \left( \langle \nabla f(\Y \Y^\top) - \nabla f(\X^*) -  (\Y \Y^\top -\X^*) , \bDelta \Y^\top \rangle + \langle \nabla f(\X^*) , \bDelta \Y^\top \rangle  \right) \\
			  \overset{\text{Lemmas } \ref{lm: RIP-imply-gradient-bound}, \ref{lm: charac of Schatten-q norm}  } \leq &  \max_{\bDelta: \|\bDelta\|_\F = 1} 2 \delta \| \bDelta \Y^\top  \|_\F \|\Y\Y^\top - \X^*\|_\F + 2 \| \bDelta \Y^\top\|_\F \| ( \nabla f(\X^*) )_{\max(r)} \|_\F \\
			  \leq & 2 \delta \|\Y  \| \|\Y\Y^\top - \X^*\|_\F + 2 \|\Y\| \| ( \nabla f(\X^*) )_{\max(r)} \|_\F
		\end{split}
	\end{equation*}
	Second,
	\begin{equation*}
		\begin{split}
			& \left| \overline{\Hess\, H([\Y])}[\theta_\Y, \theta_\Y] - \overline{\Hess\, h([\Y])}[\theta_\Y, \theta_\Y]  \right| \\
			\overset{\text{Lemma }\ref{lm: gradient-hessian-exp-PSD}, \eqref{eq: gradient-Hessian-exp-H} } \leq  & \left| \nabla^2 f(\Y \Y^\top)[\Y\theta_\Y^\top + \theta_\Y \Y^\top , \Y\theta_\Y^\top + \theta_\Y \Y^\top ] - \|\Y\theta_\Y^\top + \theta_\Y \Y^\top \|_\F^2 \right| \\
			& + \left| 2\langle \nabla f(\Y \Y^\top )- (\Y \Y^\top -\X^*), \theta_\Y \theta_\Y^\top \rangle   \right| \\
			\overset{(a)}\leq & \delta  \|\Y\theta_\Y^\top + \theta_\Y \Y^\top \|_\F^2 + \left| 2\langle \nabla f(\Y \Y^\top ) - \nabla f(\X^*) - (\Y \Y^\top -\X^*), \theta_\Y \theta_\Y^\top \rangle   \right| \\
			& +  \left| \langle 2 \nabla f(\X^*), \theta_\Y \theta_\Y^\top \rangle   \right| \\
			\overset{ \text{Lemmas }\ref{lm: RIP-imply-gradient-bound}, \ref{lm: charac of Schatten-q norm} } \leq & \delta  \|\Y\theta_\Y^\top + \theta_\Y \Y^\top \|_\F^2  + 2 \delta \|\Y\Y^\top - \X^*\|_\F \|\theta_\Y \theta_\Y^\top\|_\F + 2 \| ( \nabla f(\X^*) )_{\max(r)} \|_\F \|\theta_\Y \theta_\Y^\top\|_\F,
		\end{split}
	\end{equation*} here (a) is because $f$ satisfies the $(2r,4r)$-restricted strong convexity and smoothness properties with parameter $\delta$ and $\rank(\Y\Y^\top) = r$, $\rank(\Y\theta_\Y^\top + \theta_\Y \Y^\top) \leq 2r$. This finishes the proof of this proposition.

\section{Auxiliary Lemmas for Theorem \ref{th: convexity-radius-Mq}} \label{proof-sec: add-lemma-convexity-radius}
\begin{Lemma}(\cite[Theorem 6.3]{massart2020quotient}) \label{lm: injec-radius}
	For any $\Y \in \bbR^{p \times r}_*$, the injectivity radius of $\cM_{r+}^q$ at $[\Y]$ is $\sigma_r(\Y)$.
\end{Lemma}

\begin{Lemma} \label{lm: variation-formula-for-singular-values}
	Suppose $\A(t) \in \bbR^{p_1 \times p_2}$ depends smoothly on a time variable $t$, so that the singular value $\sigma_i(t) = \sigma_i(\A(t))$ and left (right) singular vectors $\u_i(t) = \u_i(\A(t))$ ($\v_i(t) = \v_i(\A(t))$) also depend smoothly on $t$. Then for every $i = 1, \ldots, p_1 \wedge p_2$:
	\begin{equation*}
		\begin{split}
			\dot{\sigma}_i(t) = \langle \u_i(t), \dot{\A}(t) \v_i(t) \rangle \quad \text{ and }\quad 
			\ddot{\sigma}_i(t) =\langle \u_i(t), \ddot{\A}(t) \v_i(t) \rangle + \left\langle \dot{\A}(t), \frac{d (\u_i(t) \v_i(t)^\top)}{dt} \right\rangle.
		\end{split}
	\end{equation*} Here for a smooth function $\phi$ of $t$, $\dot{\phi}(t) := \frac{d\phi }{dt}$, $\ddot{\phi}(t) := \frac{d^2 \phi}{dt^2}$.
\end{Lemma}
\begin{proof}
	First notice $\langle \u_i(t), \A(t) \v_i(t) \rangle = \sigma_i(t)$. Differentiating with respect to $t$ on both sides yields
	\begin{equation}\label{eq: first-order-derive}
	\begin{split}
		 \dot{\sigma}_i(t)&=\langle \dot{\u}_i(t), \A(t) \v_i(t) \rangle + \langle \u_i(t), \dot{\A}(t) \v_i(t) \rangle + \langle \u_i(t), \A(t) \dot{\v}_i(t) \rangle\\
		\Longrightarrow \dot{\sigma}_i(t)&=\sigma_i(t) \langle \dot{\u}_i(t), \u_i(t) \rangle + \langle \u_i(t), \dot{\A}(t) \v_i(t) \rangle + \sigma_i(t)\langle \v_i(t), \dot{\v}_i(t) \rangle\\
		\overset{(a)}\Longrightarrow \dot{\sigma}_i(t) &=  \langle \u_i(t), \dot{\A}(t) \v_i(t) \rangle.
	\end{split}
	\end{equation} Here (a) is because $\langle \u_i(t), \u_i(t) \rangle = 1$, so $\langle \dot{\u}_i(t), \u_i(t) \rangle = 0$. Similarly we have $\langle \v_i(t), \dot{\v}_i(t) \rangle = 0$.
	
	Differentiating with respect to $t$ on both sides of $\dot{\sigma}_i(t) = \langle \u_i(t), \dot{\A}(t) \v_i(t) \rangle$, we get
	\begin{equation*}
		\begin{split}
			\ddot{\sigma}_i(t) &= \langle \u_i(t), \ddot{\A}(t) \v_i(t) \rangle + \langle \dot{\u}_i(t), \dot{\A}(t) \v_i(t) \rangle + \langle \u_i(t), \dot{\A}(t) \dot{\v}_i(t) \rangle \\
			\Longrightarrow  \ddot{\sigma}_i(t) &= \langle \u_i(t), \ddot{\A}(t) \v_i(t) \rangle + \langle \frac{d (\u_i(t) \v_i(t)^\top)}{dt}, \dot{\A}(t) \rangle .
		\end{split}
	\end{equation*}
\end{proof}

\begin{Lemma} \label{lm: positive-det}
	Given any $\Y, \Y' \in \bbR^{p \times r}_*$ such that $d([\Y'], [\Y]) < \sigma_r(\Y)$. Let $\O'  = \argmin_{\O \in \bbO_r} \| \Y' \O - \Y  \|_\F$. Then $\Y^\top (\Y + t(\Y' \O' - \Y) )$ is nonsingular and $\det(\Y^\top (\Y + t(\Y' \O' - \Y) ) ) > 0$ for all $t \in [0,1]$.
\end{Lemma}
\begin{proof} Suppose $\U$ spans the top $r$ eigenspace of $\Y \Y^\top$. By Lemma \ref{lm: logarithm-map}, we know $\Y' \O' - \Y \in \cH_\Y \widebar{\cM}_{r+}^{q}$, so we can represent $\Y' \O' - \Y$ by $\Y (\Y^\top \Y)^{-1} \S + \U_\perp \D$ for some $\S \in \bbS^{r \times r}, \D \in \bbR^{(p-r) \times r}$ by Lemma \ref{lm: psd-quotient-manifold1-prop}.

For the sake of contradiction, suppose there exists a $t^* \in [0,1]$ such that $\Y^\top (\Y + t^*(\Y' \O' - \Y))$ is singular. We show next that this implies $d([\Y'],[\Y]) = \|\Y'\O' - \Y\|_\F \geq \sigma_r(\Y)$, thus contradicting our assumption. 

Since $\Y^\top (\Y + t^*(\Y' \O' - \Y))$ is singular, the matrix $\Y (\Y^\top \Y)^{-1}\Y^\top (\Y + t^*(\Y' \O' - \Y)) = \Y + t^* \Y (\Y^\top \Y)^{-1}\S$ is singular also. It follows that
\begin{equation*}
	\begin{split}
		\|\Y'\O' - \Y\|^2_\F = \|\Y (\Y^\top \Y)^{-1} \S\|^2_\F + \|\U_\perp \D\|_\F^2 \geq \|t^*\Y (\Y^\top \Y)^{-1} \S\|^2_\F \geq \min_{\widetilde{\Y} \in \bbR^{p \times r} \setminus \bbR^{p \times r}_* }\| \widetilde{\Y} - \Y \|_\F \overset{ (a) }= \sigma_r(\Y).
	\end{split}
\end{equation*} Here (a) is by the Schmidt-Mirsky theorem, see \cite[Chapter 1, Theorem 4.32]{stewart1998matrix}. We thus conclude that $\Y^\top (\Y + t(\Y' \O' - \Y) )$ is nonsingular for all $t \in [0,1]$. 

Using now the continuity of $\det(\Y^\top (\Y + t(\Y' \O' - \Y) ) )$ with respect to $t$, the fact that $\Y^\top (\Y + t(\Y' \O' - \Y) )$ is nonsingular for all $t \in [0,1]$, and the fact that at $t = 0$ we have $\det(\Y^\top \Y) > 0$, we deduce that $\det(\Y^\top (\Y + t(\Y' \O' - \Y) ) ) > 0$ for all $t \in [0,1]$. This completes the proof.
\end{proof}

The fact that the matrix $\Y^\top (\Y + t(\Y' \O' - \Y) )$ is nonsingular for all $t \in [0,1]$ has been proved in \cite[Proposition 6.2]{massart2020quotient}. Here we have presented the proof again for completeness, emphasizing the fact that the determinant does not change sign. This fact will be needed in the next lemma.
  \nc

\begin{Lemma}[Condition Number for the Perturbation of Orthogonal Procrustes Problem] \label{lm: procrustes-condition-number}
	Let $\Y, \Y' \in \bbR^{p \times r}_*$ such that $d([\Y'], [\Y]) < \sigma_r(\Y)$, and let $\Delta_\Y$ and $\Delta_{\Y'}$ be two $\mathbb{R}^{p\times r}$ matrices. Let $\O_t = \argmin_{\O \in \bbO_r} \| (\Y' + t \Delta_{\Y'} ) \O - (\Y + t\Delta_{\Y} )  \|_\F$ for $t \geq 0$. Then
	\begin{equation*}
		\left\| \frac{d \O_t}{dt}\Big|_{t=0}  \right\|_\F \leq \sqrt{2} \left( \frac{\| \Delta_{\Y'} \|_\F }{ ( \sigma_r(\Y)^2 + \sigma^2_{r-1}(\Y) )^{1/2} } + \frac{\| \Delta_{\Y} \|_\F}{( \sigma_r(\Y)^2 + \sigma^2_{r-1}(\Y) )^{1/2} - d([\Y'], [\Y]) }   \right). 
	\end{equation*}
\end{Lemma}
\begin{proof}
	Let us introduce another alignment matrix $\O'_t$: $$\O'_t = \argmin_{\O \in \bbO_r} \| (\Y' + t \Delta_{\Y'} ) \O_0 \O - (\Y + t\Delta_{\Y} )  \|_\F.$$ 
	It is easy to see that $\O_0 \O'_t = \O_t$ and $\O_0' = \I_r$. By Lemma \ref{lm: positive-det}, we know that when $d([\Y'], [\Y]) < \sigma_r(\Y)$, then $ \det(  \Y^\top \Y' \O_0  ) > 0 $, and thus $\det((\Y + t\Delta_{\Y} )^\top (\Y' + t \Delta_{\Y'} ) \O_0 ) >0$ for all small enough $t$. \nc By Lemma \ref{lm: logarithm-map}, we have $\O'_t = \V_t \U_t^\top$, where $(\Y + t\Delta_{\Y} )^\top (\Y' + t \Delta_{\Y'} ) \O_0$ has SVD $\U_t \bSigma_t \V_t^\top$, so for small enough $t$ we have $\sign(\det(\O'_t)) = \sign(\det((\Y + t\Delta_{\Y} )^\top (\Y' + t \Delta_{\Y'} ) \O_0 ))$.
	
	By the continuity of the determinant, we have $\det(\O'_t) > 0$ for small enough $t$ because $\sign(\det(\O'_0)) = \sign( \det(\Y^\top \Y' \O_0) ) > 0 $. This implies that the best alignment matrix between $\Y + t\Delta_{\Y}$ and $(\Y' + t \Delta_{\Y'})\O_0$ is a rotation matrix, i.e., it has determinant $1$, for small enough $t$. Then by \cite[Corollary 2.1]{soderkvist1993perturbation}, we have
	\begin{equation} \label{ineq: Ot2_perturbation}
	\begin{split}
		\|\O'_t - \O'_0\|_\F &\leq \sqrt{2} t \left( \frac{\| \Delta_{\Y'} \O_0 \|_\F }{ ( \sigma_r(\Y)^2 + \sigma^2_{r-1}(\Y) )^{1/2} } + \frac{\| \Delta_{\Y} \|_\F}{( \sigma_r(\Y)^2 + \sigma^2_{r-1}(\Y) )^{1/2} - d([\Y'], [\Y]) }  \right) + O(t^3)\\
		& = \sqrt{2} t \left( \frac{\| \Delta_{\Y'}\|_\F }{ ( \sigma_r(\Y)^2 + \sigma^2_{r-1}(\Y) )^{1/2} } + \frac{\| \Delta_{\Y} \|_\F}{( \sigma_r(\Y)^2 + \sigma^2_{r-1}(\Y) )^{1/2} - d([\Y'], [\Y]) }  \right) + O(t^3).
	\end{split} 
	\end{equation}
	Finally, we have
	\begin{equation*}
		\begin{split}
			& \quad \left\| \frac{d \O_t}{dt}\Big|_{t=0}  \right\|_\F  = \left\| \lim_{t \to 0} \frac{ \O_t - \O_0}{t}  \right\|_\F = \left\| \lim_{t \to 0} \frac{ \O_0\O'_t - \O_0 \O'_0}{t}  \right\|_\F 
			\\& = \left\| \lim_{t \to 0} \frac{ \O'_t - \O'_0}{t}  \right\|_\F  =\lim_{t \to 0} \left\|  \frac{ \O'_t - \O'_0}{t}\right\|_\F 
			\\&  \leq \nc  \sqrt{2} \left( \frac{\| \Delta_{\Y'} \|_\F }{ ( \sigma_r(\Y)^2 + \sigma^2_{r-1}(\Y) )^{1/2} } + \frac{\| \Delta_{\Y} \|_\F}{( \sigma_r(\Y)^2 + \sigma^2_{r-1}(\Y) )^{1/2} - d([\Y'], [\Y]) }   \right),
		\end{split}
	\end{equation*}
	where the last inequality follows from \eqref{ineq: Ot2_perturbation}.
\end{proof}

\begin{Remark}
Lemma \ref{lm: procrustes-condition-number} provides a generalization of Corollary 2.1 in \cite{soderkvist1993perturbation}. There,  $\O_t$ is defined as a minimizer among \textit{rotation} matrices, i.e. elements in $\mathbb{O}_r$ with determinant one, whereas here we are interested in arbitrary orthonormal matrices. In our proof, however, we show that we can reduce our setting to that in \cite{soderkvist1993perturbation}.
\end{Remark}

\nc

\subsection{Proof of Lemma \ref{lm: totally-normal-neigh}} \label{proof-sec: lm-totally-normal-neigh}
	For given $[\Y'] \in B_x([\Y])$,  let $\Q$ be defined as in Lemma \ref{lm: logarithm-map}, i.e. $\Q$ is the best matrix in $\mathbb{O}_r$ aligning $\Y$ and $\Y'$\nc. Then
	\begin{equation} \label{ineq: sigma-r-lowe-bound}
		\begin{split}
			\sigma_r(\Y') = \sigma_r(\Y' \Q) = \sigma_r(\Y' \Q - \Y + \Y) \geq \sigma_r(\Y) - \|\Y' \Q - \Y\| &\geq \sigma_r(\Y) - \|\Y' \Q - \Y\|_\F\\
			& > \sigma_r(\Y) - x,
		\end{split}
	\end{equation}
	 where the first inequality follows from Weyl's theorem \cite[Theorem  4.29]{stewart1998matrix}\nc.
	Combining with Lemma \ref{lm: injec-radius}, \eqref{ineq: sigma-r-lowe-bound} implies that the injectivity radius at $[\Y']$ is no smaller than $\sigma_r(\Y) - x$.
	
	Moreover, for any other $[\Y''] \in B_x([\Y])$, we have
\begin{equation*}
	d([\Y''],[\Y']) \leq d([\Y''],[\Y]) + d([\Y'],[\Y]) < \frac{2}{3} \sigma_r(\Y) \leq \sigma_r(\Y) - x.
\end{equation*} This implies $B_{\sigma_r(\Y) - x}([\Y']) \supset B_x([\Y]) $ and finishes the proof of this lemma.

\section{Additional Lemmas} \label{proof-sec: additional-lemmas}

The next lemma establishes a one-to-one correspondence between PSD matrices with rank $r$ and the space $\cM_{r+}^q$.

\begin{Lemma}(\cite[Proposition A.1]{massart2020quotient}) \label{lm: completeness-quotient-PSD} \nc Let $\Y_1, \Y_2 \in \bbR_*^{p \times r}$. Then $\Y_1 \Y_1^\top = \Y_2 \Y_2^\top$ if and only if $\Y_2 = \Y_1 \O$ for some $\O \in \bbO_r$. Moreover, $\{\X:\bbS^{p \times p} \ni \X \succcurlyeq 0, \rank(\X) = r \} = \{\Y \Y^\top: \Y \in \bbR^{p \times r}_* \}.$
\end{Lemma}

\nc

The next lemma states that this correspondence is a locally bi-Lipschitz map and quantifies the corresponding local Lipschitz constants. 
\nc 
\begin{Lemma} \label{lm: distUY-UUYY-transfer}
	For any $\Y_1, \Y_2 \in \bbR^{p \times r}_*$, we have
	\begin{equation} \label{ineq: distUY-UUYY-transfer-1}
		d^2([\Y_1], [\Y_2]) \leq \frac{1}{2(\sqrt{2}-1) \sigma^2_r(\Y_2) } \|\Y_1 \Y_1^\top - \Y_2 \Y_2^\top \|_\F^2,
	\end{equation}
	and 
	\begin{equation}\label{ineq: distUY-UUYY-transfer-2}
		\|(\Y_1 - \Y_2 \Q) (\Y_1 - \Y_2 \Q)^\top\|_\F^2 \leq 2\|\Y_1 \Y_1^\top - \Y_2 \Y_2^\top \|_\F^2,
	\end{equation}
	where $\Q = \argmin_{\O \in \bbO_r}  \|\Y_1 - \Y_2 \O\|_\F$.
	
	In addition, for any $\Y_1, \Y_2 \in \bbR^{p \times r}_*$ \nc obeying $d([\Y_1],[\Y_2]) \leq \frac{1}{3} \sigma_r(\Y_2)$, we have
	\begin{equation}\label{ineq: distUY-UUYY-transfer-3}
		\|\Y_1 \Y_1^\top - \Y_2 \Y_2^\top \|_\F \leq \frac{7}{3} \|\Y_2 \| d([\Y_1], [\Y_2]).
	\end{equation}
\end{Lemma}
\begin{proof}
	The first statement is from \cite[Lemma 5.4]{tu2016low}, the second statement is from \cite[Lemma 6]{ge2017no} and the third statement is a slight modification of \cite[Lemma 5.3]{tu2016low}.
\end{proof}

\begin{Lemma}(\cite[Proposition 2]{luo2021geometric}) \label{lm: norm-bound-Ytheta-thetaY} Let $\Y \in \bbR^{p \times r}_*$, and let $\X = \Y \Y^\top$. Then $2 \sigma^2_r(\Y) \|\theta_\Y\|_\F^2 \leq \|\Y \theta_\Y^\top + \theta_\Y \Y^\top \|_\F^2 \leq 4 \sigma^2_1(\Y) \|\theta_\Y\|_\F^2$ holds for all $\theta_\Y \in \cH_\Y \widebar{\cM}_{r+}^{q}$.
\end{Lemma}

\begin{Remark}
The previous result can be interpreted as a quantitative bound for the metric distortion of the differential of the correspondence between $\X$ and $[\Y]$. 
Indeed, the tangent plane of $\{\X:\bbS^{p \times p} \ni \X \succcurlyeq 0, \rank(\X) = r \}$ at $\X= \Y \Y^\top $ consists of matrices of the form $(\Y \theta_\Y^\top + \theta_\Y \Y^\top  )$ for $\theta_\Y \in \cH_\Y \widebar{\cM}_{r+}^{q} $, as has been discussed in \cite{luo2021geometric}.  { We observe that from Lemma \ref{lm: distUY-UUYY-transfer} one could directly obtain the bound $ 2(\sqrt{2} -1) \sigma^2_r(\Y) \|\theta_\Y\|_\F^2  \leq \|\Y \theta_\Y^\top + \theta_\Y \Y^\top \|_\F^2$}, while Lemma \ref{lm: norm-bound-Ytheta-thetaY} provides a better constant for this inequality. The upper bound $\|\Y \theta_\Y^\top + \theta_\Y \Y^\top \|_\F^2 \leq 4 \sigma^2_1(\Y) \|\theta_\Y\|_\F^2$ follows from straightforward linear algebra considerations.

\end{Remark}

\nc

\begin{Lemma}(\cite[Lemma 10]{zhu2017global}) \label{lm: RIP-imply-gradient-bound}
	Suppose $f$ satisfies $(2r,4r)$-restricted strong convexity and smoothness properties with parameter $0\leq \delta < 1$ as in Definition \ref{def: RSC-RSM}. Then for any $p$-by-$p$ \nc real-valued matrices $\C,\D,\H$ with $\rank(\C), \rank(\D) \leq r$ and $\rank(\H) \leq 2r$, we have
	\begin{equation*}
		\left| \langle \nabla f(\C) - \nabla f(\D) - (\C - \D)  , \H \rangle  \right| \leq \delta \| \C - \D \|_\F \|\H\|_\F.
	\end{equation*}
\end{Lemma}

\begin{Lemma}(\cite[Lemma 3]{luo2021schatten})\label{lm: charac of Schatten-q norm}
Let $\X$ be a $p_1$-by-$p_2$ real-valued matrix. For any non-negative integer $r \leq p_1 \wedge p_2$, we have
    \begin{equation}\label{eq: truncated schatten q norm}
        \|\X_{\max(r)}\|_\F = \sup_{\|\B\|_\F \leq 1, \rank(\B) \leq r}  \langle \B, \X \rangle
    \end{equation} 
    If $\rank(\X) \leq r$, then 
    \begin{equation}\label{eq: schatten q norm of rank r matrix}
         \|\X\|_\F = \sup_{\|\B\|_\F \leq 1, \rank(\B) \leq r}  \langle \B, \X \rangle.
    \end{equation}
\end{Lemma}

\bibliographystyle{alpha}
\bibliography{reference.bib}

\end{document}